\documentclass[11pt,reqno,a4paper]{extarticle}

\usepackage[left=3cm,right=3cm,top=2.2cm,bottom=2.2cm]{geometry}
\usepackage[osf]{newtxtext}

\usepackage{amsmath,amssymb,mathrsfs,stmaryrd}
\usepackage{amsthm}
\usepackage{mathtools}
\usepackage{empheq}
\usepackage{enumerate,enumitem}
\usepackage{dsfont}
\usepackage[hidelinks]{hyperref}
\usepackage{esint} 
\usepackage{soul}
\usepackage{xcolor}
\usepackage{titlefoot}
\usepackage{authblk} % For \affil command
\theoremstyle{plain}
\newtheorem{lemma}{Lemma}[section]
\newtheorem{proposition}{Proposition}[section]
\newtheorem{corollary}{Corollary}[section]
\newtheorem{theorem}{Theorem}[section]
\newtheorem{assumption}{Assumption}[section]
\newtheorem{definition}{Definition}[section]
\newtheorem{notation}{Notation}[section]

\theoremstyle{definition}
\newtheorem{nota}{Notation}[section]

\newtheorem{example}{Example}[section]
\newtheorem{remark}{Remark}[section]

\newcommand{\n}[1]{\ensuremath{\llbracket#1\rrbracket}}
\newcommand{\nn}[2]{\ensuremath{\llbracket#1\rrbracket_{-#2}^{[#2\alpha]}}}
\newcommand{\nnn}[1]{\ensuremath{\llbracket#1\rrbracket_{-1}^{[\alpha ]}}}
\newcommand{\N}{\ensuremath{\mathbb N_0}} % natural integers
\newcommand{\R}{\ensuremath{\mathbb R}} % real numbers
\newcommand{\braket}[3]{\ensuremath{\prescript{}{#1}{\big\langle }#2\big\rangle_{#3}}}
\newcommand{\BRAKET}[3]{\ensuremath{\prescript{}{#1}{\Big\langle}#2\Big\rangle_{#3}}}

\renewcommand{\d}{\ensuremath{\mathrm d}}
\newcommand{\DD}{\ensuremath{\mathbb{D}}}
\newcommand{\ind}{\mathbf{1}}

\newcommand{\f}{\ensuremath{\mathbf f}}
\newcommand{\loc}{\ensuremath{{\mathrm{loc}}}}
\newcommand{\V}{\ensuremath{\mathcal V}}
\renewcommand{\H}{\ensuremath{\mathcal H}}

\newcommand{\sym}{\ensuremath{\mathrm{sym}}}
\renewcommand{\L}{\ensuremath{{[\mathbf \B]}}}
\newcommand{\ito}{\ensuremath{{\mathrm{It\hat o}}}}
\renewcommand{\qed}{\hfill$\blacksquare$}

\newcommand{\B}{\ensuremath{{\mathbf B}}}
\newcommand{\X}{\ensuremath{{\mathbf X}}}

\newcommand{\w}{\ensuremath{{\mathrm{w}}}}

\newcommand{\XX}{\ensuremath{{\mathbb X}}}
\newcommand{\A}{\ensuremath{{\mathbb A}}}
\newcommand{\YY}{\ensuremath{{\mathbb Y}}}

\newcommand{\LL}{\ensuremath{{\mathbb L}}}
\newcommand{\Z}{\ensuremath{{\mathbf Z}}}

\usepackage{graphicx}
\newcommand{\bbGamma}{{\mathpalette\makebbGamma\relax}}
\newcommand{\makebbGamma}[2]{%
	\raisebox{\depth}{\scalebox{1}[-1]{$\mathsurround=0pt#1\mathbb{L}$}}%
}

\DeclareMathOperator*{\esssup}{ess\,sup}
\DeclareMathOperator*{\essinf}{ess\,inf}

\DeclareMathOperator{\id}{id}

\makeatletter

\@addtoreset{equation}{section}
\makeatother

\begin{document}
\title{An It\^o Formula for rough partial differential equations and some applications}
%\title[An It\^o Formula for rough PDEs]
\date{}

\author[1]{Antoine Hocquet}
\author[2]{Torstein Nilssen}

\affil[1]{\small Institut f\"ur Mathematik,  Technische Universit\"at Berlin, Stra\ss e des 17. Juni 136, D-10623 Berlin, Germany}
\affil[2]{\small Mathematical Institute, University of Agder, Universitetsveien 25, 4630 Kristiansand, Norway}

\maketitle

\unmarkedfntext{\textit{Mathematics Subject Classification (2020) ---} 60L50 $\,$\textbullet$\,$ 60H15 $\,$\textbullet$\,$ 35A15 $\,$\textbullet$\,$ 35B50 $\,$\textbullet$\,$ 35D30}

\unmarkedfntext{\textit{Keywords and phrases ---} 
	Rough paths $\,$\textbullet$\,$ Rough PDEs $\,$\textbullet$\,$ Energy method $\,$\textbullet$\,$ Weak solutions $\,$\textbullet$\,$ It\^o formula $\,$\textbullet$\,$ Maximum principle $\,$\textbullet$\,$ Renormalized solutions}

\unmarkedfntext{\textit{Mail ---} antoine.hocquet@wanadoo.fr $\,$\textbullet$\,$ torstein.nilssen@uia.no}

\begin{abstract}
	We investigate existence, uniqueness and regularity for solutions of rough parabolic equations of the form $\partial _tu-A_tu-f=(\dot X_t(x) \cdot \nabla  + \dot Y_t(x))u$ on $[0,T]\times\mathbb{R}^d.$ To do so, we introduce a concept of ``differential rough driver'', which comes with a counterpart of the usual controlled paths spaces in rough paths theory, built on the Sobolev spaces $W^{k,p}.$ We also define a natural notion of geometricity in this context, and show how it relates to a product formula for controlled paths. In the case of transport noise (i.e.\ when $Y=0$), we use this framework to prove an It\^o Formula (in the sense of a chain rule) for Nemytskii operations of the form $u\mapsto F(u),$ where $F$ is $C^2$ and vanishes at the origin. Our method is based on energy estimates, and a generalization of the Moser Iteration argument to prove boundedness of a dense class of solutions of parabolic problems as above. In particular, we avoid the use of flow transformations and work directly at the level of the original equation. We also show the corresponding chain rule for $F(u)=|u|^p$ with $p\geq 2,$ but also when $Y\neq 0$ and $p\geq 4.$ As an application of these results, we prove existence and uniqueness of a suitable class of $L^p$-solutions of parabolic equations with multiplicative noise. Another related development is the homogeneous Dirichlet boundary problem on a smooth domain, for which a weak maximum principle is shown under appropriate assumptions on the coefficients.
\end{abstract}

\tableofcontents

\section{Introduction}
\label{sec:intro}

\textbf{Motivations.}
Consider a stochastic partial differential equation with multiplicative noise of the form
\begin{equation}
\label{typical}
\d u_t -\Delta u_t\d t= \partial _iu_t\d X^i_t(x) + u_t\d X^0_t(x)\,,
\quad \text{on}\quad (0,T]\times\R^d
\end{equation}
where $\partial _i=\frac{\partial }{\partial x_i},$ $T\in(0,\infty)$ denotes a fixed time horizon, $(X^i)_{i=0,\dots,d}$ denotes some $Q$-Wiener process (sufficiently smooth in $x$), and throughout the paper we use Einstein's summation convention over repeated indices.
For now the product with the above differentials is subject to different possible meanings (for instance Stratonovitch or It\^o).

Equations such as \eqref{typical} arise in a number of different stochastic models. To name a few, this includes filtering theory \cite{gyongy2006zakai}, McKean-Vlasov equations \cite{kotelenez2010class}, or pathwise stochastic control problems (see for instance \cite[Example 2]{caruana2011rough} and references therein).
In the more general context of a degenerate left hand side, this type of noise appears in stochastic transport equations (with $X^0=0$), where a regularization by noise phenomenon is observed \cite{flandoli2010well,nilssen2015rough,mohammed2015sobolev,catellier2016rough}, or in stochastic conservation laws, see \cite{guess2016regularization} for an overview.
We also mention the works \cite{delarue2016rough,cannizzaro2018multidimensional} where the authors solve an equation similar to \eqref{typical}, with the difference that they consider a vector field $X^i_t(x)$ which is rough with respect to the space-like variable.

The way \eqref{typical} is usually dealt with is by definition of an appropriate functional setting, in which standard It\^o calculus tools can be used. We refer for instance to the classical works of Pardoux, Krylov and Rozovskii \cite{pardoux1980stochastic,krylov1981stochastic}. Although these approaches are quite sucessful, it is well-known that the solution map $X\mapsto u$ is not continuous in general. This constitutes an important motivation for introducing a \emph{rough paths formulation} of \eqref{typical} (in particular because the examples given above display a need for stability results, see \cite{friz2014rough}).
Rough parabolic differential equations such as \eqref{typical} have been investigated in
\cite{caruana2009partial,friz2011splitting,caruana2011rough,friz2014rough} where a viscosity formulation is proposed, based on ideas of Lions and Souganidis \cite{lions1998fully,lions1998fully2}.
Despite their success, these papers appeal to an extensive use of flow transformation techniques, which has some conceptual disadvantages. In particular, they have to make the assumption that the solutions are obtained as limits of approximations. 
To the best of our knowledge, the Feynmann-Kac representation technique used in \cite{diehl2017stochastic}, constitutes the first attempt to deal with \eqref{typical} directly (there is also the semigroup approach of Gubinelli, Deya and Tindel \cite{gubinelli2010rough,deya2012nonlinear}, but their results do not seem to cover the case of a gradient noise as above).

One of our main purposes in this paper is to pursue the variational approach initiated by Deya, Gubinelli, Hofmanov\'a and Tindel in \cite{deya2016priori}, by defining, among other things, a suitable functional setting for generalized versions of \eqref{typical}. In this sense, we will particularly emphasize the topological aspects associated with \eqref{typical}, for instance by introducing the controlled paths spaces $\mathcal D^{\alpha ,p}_B,$ as well as their parabolic counterpart $\H^{\alpha ,p}_B$ (see sections \ref{sec:controlled_paths} and \ref{sec:space}).
Working with classical PDE techniques such as energy estimates and maximum principles, our contribution can be seen as an attempt to extend Krylov's analytic approach \cite{krylov1999analytic} to the RPDE context. One of the key concepts we will use here is that of an unbounded rough driver, as introduced by Bailleul and Gubinelli in \cite{bailleul2017unbounded}. More specifically, we will introduce a notion of \emph{differential} rough driver, which is a particular case of the former (see Definition \ref{def:rough_driver}). 
We will also provide a natural, intrinsic notion of geometricity for differential rough drivers. As shown in Lemma \ref{lem:multiplication}, geometric differential rough drivers display remarkable algebraic properties.
In particular, they are simultaneously symmetric, closed and renormalizable in the sense of \cite[definitions 5.3, 5.4 \& 5.7]{bailleul2017unbounded}. In contrast with the previous works \cite{deya2016priori,hocquet2017energy,hofmanova2018navier}, we will be able to consider these objects ``as such'', in the sense that we will not refer to any (geometric) finite-dimensional rough path. This observation, which can be seen as one of our main contributions, allows us to gain generality in the statements and, hopefully, to improve the clarity of the presentation.
\\

\textbf{The importance of geometricity and its relation to stochastic parabolicity.}
In contrast with the recent developments on rough parabolic equations \cite{gubinelli2010rough,hairer2013solving,hairer2013rough,hairer2014theory,gubinelli2015paracontrolled,otto2016quasilinear,bailleul2016higher} (for results related to It\^o Formula in this case, see \cite{zambotti2006ito,bellingeri2018ito}), the noise term in \eqref{typical} is not singular with respect to the space-variable, so that in appearance \eqref{typical} does not fall into the category of ``singular PDEs''. However, difficulties arise from the fact that for all times $t$ the operation $u\mapsto X_t \cdot \nabla u$ is unbounded. A side effect of this property is that the low time-regularity of solutions implies in turn low \emph{space}-regularity, as can be seen by the scaling properties of the equation.
In the case of $X=W$ being a Brownian motion and $X^0=0$, it is easily seen that for $\varepsilon>0$ the transform $(t,x)\to (\varepsilon ^2t,\varepsilon x)$ leaves the equation invariant (using the scaling properties of $W$). Leaving aside mathematical rigor,
this type of invariance indicates that \eqref{typical} cannot be considered as a perturbation of a heat equation at small scales. In this sense, the equation \eqref{typical} is not really parabolic and the use of semigroups and variation of constants formulae is inoperative (we nevertheless refer to the recent works \cite{gerasimovics2018hormander,gerasimovics2019non} in a similar but ``subcritical'' context). The situation can go even worse if $X=W^H$ is a fractional Brownian motion with hurst index $1/3<H<1/2,$ a case that is covered by our results.
In this case, the transport term $\partial _t - \dot W^H_t\cdot \nabla $ dominates, even though the drift term has two spatial derivatives. This might be a loose explanation why some of the arguments below seem to have a transport flavour (the bounds \eqref{bounds:renormalization} which are needed in the tensorization argument of Section \ref{sec:space} can be understood as a ``commutator lemma'' \`a la Di Perna Lions \cite{diperna1989ordinary}; see Appendix \ref{app:renorm}).
As a matter of fact, the fractional Brownian case enters the category of ``supercritical'' equations in the sense of \cite[Section 8]{hairer2014theory}, and this is so regardless of the space dimension $d$.

In this context, the assumption that $X$ is geometric turns out to be essential.
To illustrate why, let us go back to the standard Brownian motion case, more precisely let $d=1,$ consider $X_t=bW_t$, $b\in\R$ being a constant, and for simplicity take $X^0=0$. Assume for a moment that \eqref{typical} is understood in the sense of It\^o, so that the corresponding rough path formulation would violate geometricity. Computing formally the It\^o Formula for the square of the $L^2$-norm of the solution, one sees that the correction term is given by $\int_{\R^d}b ^2(\partial _{x}u)^2,$ which dangerously competes with the conservative term $-2\int_{\R^d}(\partial _xu)^2$ brought by the Laplacian. In particular, the usual technique to obtain an a priori estimate for $u$ fails unless $1/2b ^2< 1,$ which is a condition known as \emph{strong parabolicity}. This assumption is in fact necessary to ensure well-posedness as can be seen by taking the spatial Fourier transform in the equation (we refer the reader to \cite[Section III.3]{krylov1981stochastic}).
If on the other hand \eqref{typical} is understood in the Stratonovitch sense, the latter problem disappears, and this is to be related to the fact that a Stratonovitch equation satisfies a ``standard'' chain rule of the form 
\begin{equation}
\label{abstract_chain_rule}
\d (F(u))=F'(u)\circ \d u
\end{equation}
(meaning in particular that no correction term of the previous form appears). Besides introducing a new functional framework for \eqref{typical}, our main objective in this paper is to investigate the chain rule \eqref{abstract_chain_rule}, which will be systematically addressed in the transport-noise case, assuming ``geometricity of the driving noise'' (understood at the level of the differential rough driver, see Definition \ref{def:geometric}). In the stochastic setting, the geometricity assumption essentially means that the iterated integrals which define the second level $\LL_t$ of $\X_t$ should be understood in the Stratonovitch sense. Nevertheless, we point out that \eqref{typical} can always be translated in terms of an equivalent Stratonovitch equation. If strong parabolicity is assumed, it is straightforward to check that the corrected equation has still the parabolic form \eqref{rough_paths_eq}, and hence our main results still apply in this practical case.
\\

\textbf{Settings and summary of the results.}
In this paper, we interpret \eqref{typical} as the rough equation
\begin{equation}
\label{rough_paths_eq}
\left\{
\begin{aligned}
&\d u_t-(A_tu +f_t(x))\d t= \d \B _t u_t\enskip,
\quad \text{on}\enskip (0,T]\times \R^d
\\
&u_0\enskip \text{given in}\enskip L^p(\R^d)\,,
\end{aligned}\right.
\end{equation}
where the unknown $u_t(x)$ is seen as a path with values in the Lebesgue space $L^p(\R^d),$ for some $p\in[1,\infty].$
Here 
\[\B=(B^1,B^2)\]
denotes some kind of two-step ``enhancement'' of the time-dependent family of differential operators
\begin{equation}
\label{dB}
B _t=B^1_{0t}:= X _t^i(x)\partial _i + X^0_t(x)\,,\quad t\in[0,T],
\end{equation}
for $(X _t^i(x))_{0\leq i\leq d}$ sufficiently regular in space.
From the point of view of the coefficient path, it will be seen that $(t\mapsto X_t(x))$ must be accompanied with an additional object 
\[\LL^i_{st}(x),\quad i=0,\dots d,\quad 0\leq s\leq t\leq T,\quad x\in\R^d\,,
\]
akin to the usual \emph{L\'evy area} for two-step geometric rough paths with real-valued coordinates. The knowledge of $\LL^i$ is necessary (and sufficient) to give a proper meaning for \eqref{rough_paths_eq}. As will be seen in the manuscript, it is heuristically filling the gaps in order to make sense of the (a priori ill-defined) iterated integral 
\begin{equation}
\label{fill_gaps}
\begin{aligned}
B^2_{st}
&\phantom{:}=
\iint_{s<r_1<r_2<t} \d B _{r_2} \circ \d B _{r_1}
\\
&:= \frac12X^i_{st}X^j_{st}\partial _{ij} + (\LL^i_{st} + X^i_{st}X^0_{st})\partial _i+ \LL^0_{st} +  \frac12(X^0_{st})^2\,,
\\
&\quad \quad \quad \quad \quad \quad 
\quad \quad \text{for}\quad  0\leq s\leq t\leq T\,\quad \text{and}\quad x\in \R^d\,,
\end{aligned}
\end{equation} 
where '$\circ$' denotes the composition of linear operators. In particular, there is a one-to-one correspondence between $\B$ and the enhancement $(X,\LL)$ of its coefficient path.
Throughout the paper, the pair $\X=(X,\LL)$ is therefore considered as part of the data, and so is $\B$ through \eqref{fill_gaps}.
For simplicity, the path $X$ will be assumed to have bounded $q$-variation with $q=\frac1\alpha $ (including the $\alpha $-H\"older case), for some $\alpha >1/3$.
% For convenience, the equation will sometimes be alternatively written as
It will be sometimes more convenient to rewrite equation \eqref{rough_paths_eq} under the following form
\begin{equation*}
\left\{ 
\begin{aligned}
&\d u - (A_tu+f_t)\d t = (\d \X \cdot \nabla  + \d \X^0)u_t\quad \text{on}\enskip (0,T]\times\R^d
\\
&u_0\in L^p\,,
\end{aligned}\right.
\end{equation*}
which has the advantage of being more explicit.

In keeping with Gubinelli's approach \cite{gubinelli2004controlling}, the integration map which is implicitly associated with the right hand side of \eqref{rough_paths_eq}, only makes sense on a set of paths $u\colon[0,T]\to L^p$ that are \emph{controlled by $B$}, a notion that will be introduced in Section \ref{subsec:notion}.
Concerning the left hand side of \eqref{rough_PDE_gene}, we will assume throughout the paper that $A_t$ is a time-dependent family of elliptic operators on divergence form
\begin{equation}\label{nota:A}
A_tu(x)=\partial _{i} (a^{ij}(t,x)\partial _{j} u(x)),
\end{equation}
whith coefficients $a^{ij}$ being possibly discontinuous but bounded above and below (see assumption \ref{ass:A}). Correspondingly, the free term $f$ will be an element of the Sobolev space $L^2(0,T;H^{-1}).$
Our first main achievement is to prove well-posedness for \eqref{rough_paths_eq}, for a class of controlled paths $u\colon[0,T]\to L^2(\R^d)$ having finite energy
\[
\sup\limits_{t\in[0,T]} |u_t|_{L^2}^2 + \int_0^T|\nabla u_t|_{L^2}^2\d t <\infty\,,
\]
in the case where $\B$ is geometric. This will be stated in Theorem \ref{thm:free_intro}, completing the results of \cite{hocquet2017energy}.

Next, we will address the problem of writing an It\^o formula for solutions of \eqref{rough_paths_eq}, where in addition of geometricity, we will assume that $\B$ is ``transport-like'', that is:
\begin{equation}
\label{rho_zero}
X^0 =0\quad \text{in \eqref{dB}.}
\end{equation}
The problem of writing a chain rule for \eqref{rough_paths_eq} arises in a very natural way when studying the well-posedness of \eqref{typical}, as illustrated by the previous paragraph and the search for an energy estimate (this corresponds to the choice $F(z)=z^2$ in \eqref{abstract_chain_rule}).
The justification of the chain rule is also useful to establish comparison principles, where the corresponding choice of function would be for instance $F(z)=z^{\pm},$ or a suitable regularized version thereof.
Under the assumption \eqref{rho_zero}, we will prove that a chain rule like \eqref{abstract_chain_rule} holds for any $F\in C^2(\R,\R)$ with $F(0)=F'(0)=0$ and $|F''|_{L^{\infty}}<\infty.$ Concretely, we will see that 
\begin{equation}
\label{ito_intro}
\d (F(u))-F'(u)(A_tu+f)\d t= \d\X\cdot \nabla (F(u))
\end{equation}
(see Theorem \ref{thm:ito_transport} for a precise statement).
The formula \eqref{ito_intro} will be applied in particular to obtain a weak maximum principle for an appropriate subclass of problems of the form \eqref{rough_paths_eq}, as will be stated in Theorem \ref{thm:max_principle}.
We insist on the fact that, because of the lack of space-regularity of solutions, \eqref{ito_intro} is not a trivial statement.
In particular, the solution $u$ fails in general to satisfy the hypotheses of \cite[Proposition 7.6]{friz2014course},
see Remark \ref{rem:not_trivial}.
Note that in some sense, \eqref{ito_intro} can be seen as a parabolic analogue to the renormalization property for transport equations in Sobolev spaces \cite{diperna1989ordinary,ambrosio2004transport,de2007ordinary}. Roughly speaking, renormalized solutions could be defined as elements $u$ of the controlled path space so that \eqref{ito_intro} holds for any $F$ as above; hence \eqref{ito_intro} shows that solutions of finite energy are renormalized.
On the other hand, if $u$ is renormalized, taking $F=(\cdot )^2$ will show that $u$ is itself an $L^2$-solution,
and hence \eqref{ito_intro} can be understood as the statement that the two notions are equivalent.

Regarding applications, the chain rule for the $L^p$-norm of solutions $u\colon[0,T]\to L^p$ (that is \eqref{ito_intro} with $F(z)=|z|^p$) is of particular interest for SPDE purposes. In the stochastic setting, this echoes the works of Krylov and Kim for stochastic equations in $L^p$ spaces \cite{kim2014theory,krylov2011ito,krylov2013relatively}, where the corresponding It\^o Formula is an essential tool.
In this paper, we will investigate the  analogue for rough paths, that is for every $L^p$-solution $u$ of \eqref{rough_paths_eq}, and under some mild assumptions on $f$ and $u_0$, we will see that
\begin{equation}
\label{ito_Lp_intro}
\d |u|^p-pu|u|^{p-2}(A_tu+f)\d t= (\d\X\cdot \nabla  + p\X^0)|u|^p\,
\end{equation}
as long as $p\geq 4$
(this can be relaxed to $p\geq 2$ when $X^0=0$).
We note that, since $F''$ is not bounded, \eqref{ito_Lp_intro} is not a simple consequence of \eqref{ito_intro}, even when the multiplicative part is zero.
Nevertheless, using rough paths stability results that come for free with our formulation, it will be seen that \eqref{ito_Lp_intro} admits a relatively simple proof. In our way to prove this formula, we shall also address existence and uniqueness for a suitable class of $L^p$-solutions of parabolic equations with multiplicative noise.

Due to the relative length of this paper, and since this drastically complicates the algebra,
we chose to postpone the treatment of a more general It\^o Formula (taking for instance $X^0\neq 0$ in \eqref{dB}, or even a non-geometric $\B$) in a future work. Similarly, we could have considered an additional rough input of additive form.
More general operators $A$ (for instance adding a perturbation $b^i(t,x)\partial _i u + c(t,x)u$ with integrability conditions on $b,c,$ see \cite{hocquet2017energy}) and more general boundary problems, could also be investigated following the same ideas, but for the sake of simplicity we restrain from doing so.
\\

\textbf{Organization of the paper}
Our main results concerning existence, uniqueness, stability and the chain rule for \eqref{rough_paths_eq}, will be given in Section \ref{sec:main}, where we also introduce notations and definitions. In particular, we introduce an intrinsic formulation of \eqref{rough_paths_eq}, in the spirit of \cite{deya2016priori}.
We will complete our results by a criterion for boundedness of solutions, a chain rule for the $L^p$-norm of solutions, and a weak maximum principle for the Dirichlet problem on a bounded domain.
In Section \ref{sec:controlled_paths} we state some facts that will be used throughout the paper, such as the Sewing Lemma or the so-called ``Rough Gronwall'' argument (as stated in \cite{deya2016priori}). The main novelty of this section is that we introduce a notion of controlled path space $\mathcal D^{\alpha ,p}_B,$ with respect to a differential rough driver $\B$. We then state and re-prove the so-called ``remainder estimates'' as given by Deya, Gubinelli, Hofmanov\'a and Tindel in \cite[Theorem 2.5]{deya2016priori}. We provide an alternative formulation of this result, which has the conceptual advantage of being understood as an a priori estimate in $\mathcal D_B^{\alpha ,p}$ (as in the usual finite-dimensional controlled path picture).
In Section \ref{sec:space}, we define a suitable functional setting for rough parabolic equations by introducing the parabolic spaces $\H^{\alpha ,p}_B$. We will then state one of the core arguments of this paper, which is the ``product formula'' (Proposition \ref{pro:product}). By reiteration of the product, we will obtain the chain rule on monomials of any bounded solution, and on polynomials by linearity.

In Section \ref{sec:free}, we use this result to solve a class of rough, non-degenerate parabolic equation with free terms in the space $L^2(H^{-1}).$ This is done via energy estimates, and the use of the Rough Gronwall Lemma.
In Section \ref{sec:boundedness} we show, using a Moser Iteration, that a ``relatively large'' class of solutions to rough parabolic problems of the form \eqref{rough_paths_eq} is made of elements which are locally bounded. This observation, together with the fact that a chain rule holds for polynomials of a bounded solution, will then allow us to prove the claimed It\^o formula in Section \ref{sec:proof:ito}.
The corresponding proof for the $L^p$-norm, as well as the solvability for an appropriate class of $L^p$-solutions, will be dealt with at the end of Section \ref{sec:proof:ito}. It is based on a different argument using approximation and stability results for rough partial differential equations.

Section \ref{sec:max} is devoted to the proof of Theorem \ref{thm:max_principle}. After proving the solvability of the homogeneous Dirichlet problem on a smooth, bounded domain, we show, using our It\^o Formula, that the solutions satisfy a weak maximum principle.

In Appendix \ref{sec:appendix}, we shall give the proof of some technical facts verified by any geometric differential rough driver, generalizing \cite[Section 3.2]{deya2016priori}. Finally, Appendix \ref{app:further} will be devoted to a quick discussion on the uniqueness of the Gubinelli derivative, and on the ``non-commutative brackets'' $\L_{st}=B^2_{st}-\frac12B^1_{st}\circ B^1_{st}.$

\section{Preliminaries and main results}
\label{sec:main}
\subsection{Notation}
Throughout the paper, the notation $K\subset\subset \R^d$ stands for ``$K$ is a compact set in $\R^d$''.
The symbol $T>0$ refers to a finite, fixed time-horizon.

By $\mathbb N,$ we denote the set of natural integers $1,2,\dots,$ and we let $\N:=\mathbb N\cup \{0\},$ while $\mathbb Z:=\N \cup(-\mathbb N).$ Real numbers are denoted by $\R,$ and we let moreover $\R_+:=[0,\infty).$

Given Banach spaces $X,Y,$
we will denote by $\mathscr L(X,Y)$ the space of linear, continuous maps from $X$ to $Y,$ endowed with the operator norm.
For $f$ in $X^*:=\mathscr L(X,\R),$ we denote the dual pairing by
\[
\braket{X^*}{ f,g }{X}
\]
(i.e.\ the evaluation of $f$ at $g\in X$).
When they are clear from the context, we will simply omit the underlying spaces and write $\langle f,g\rangle$ instead.
\\

\textbf{Sobolev spaces and scales.}
For an open smooth domain $U\subset\R^d,$ we will consider the usual Lebesgue and Sobolev spaces in the space-like variable: $L^p(U)$, $W^{k,p}(U),$ for $(k,p)\in\mathbb Z\times(1,\infty]$ or $p=1$ and $k\in \N,$ and we distinguish the case $p=2$ by writing $H^k(U):=W^{2,k}(U)$; the corresponding norms will be simply denoted by $|\cdot |_{L^p(U)},{\nolinebreak|\cdot |_{W^{k,p}(U)}}, |\cdot |_{H^k(U)}.$ With the exception of Section \ref{subsec:controlled}, the notations $L^p,$ $W^{k,p}$ and $H^k$ refer to the whole space scenario $U=\R^d.$
These spaces have local (resp.\ weak) analogues $L^p_\loc,W^{k,p}_\loc,H^k_\loc$ (resp.\ $L^p_\w,W^{k,p}_\w,H^k_\w$) which are defined as usual.
When $k$ is negative, we adopt the convention that $W^{k,1}$ is the range of the linear operator
\[ (\f^\gamma )\in L^p\big(\R^d;\R^{\sum_{|\gamma |\leq -k}|\gamma |}\big)\mapsto (\partial _\gamma  \f^\gamma )_{|\gamma |\leq -k}\]
where $|\gamma |:=\gamma _1+\dots+\gamma _d,$ and the derivatives are understood in distributional sense.
Correspondingly, the norm of $f\in W^{k,1}$ is defined as the infimum of the $L^1$-norms of any possible antiderivative $\f^{\gamma }$ of $f$ .
Note that with this convention, $W^{k,1}$ identifies only with a proper subspace of the dual $(W^{|k|,\infty}_0)^*,$ however this is coherent with the case $p>1$ (see for instance \cite{brezis2010functional}).
If $U\subset\R^d$ is a domain whose boundary is smooth and if $p\in[1,\infty],$ we define the spaces $W_0^{k,p}$ as
\[
W_0^{k,p} ( U):=\left\{f\in W^{k,p}\enskip \text{s.t.\ } (\nu\cdot \nabla )^jf=0\enskip \text{for}\enskip j\in\N,\enskip j< k-1/p \right\}.
\]
where $\nu $ denotes the outward unit vector associated to $\partial  U.$

In the sequel, we call a \emph{scale} any graded family of topological vector spaces of the form $(E_k,|\cdot |_{k})_{k\in I}$ with $I\subset \mathbb Z$ such that $E_{k}$ is continuously embedded into $E_{k-1},$ for each $k\in I.$ Note that, in the paper the set $I:=\{-3,-2,-1,0,1,2,3\}$ will be sufficient for our purposes.
\\

For $0\leq s\leq t\leq T$ and $f=f_r(x)$ we use the notation 
\[
\|f\|_{L^r(s,t;L^q)}:=\left(\int_{s}^t\left(\int_ {\R^d} |f_\tau (x)|^q\d x\right)^{r/q}\d \tau \right)^{1/r},
\]
and for simplicity we will sometimes write
$\|f\|_{L^r(L^q)}$
as a shorthand for $\|f\|_{L^r(0,T;L^q)}.$
Furthermore, the space of continuous functions with values in a Fr\'echet space $E$ will be denoted by $C(0,T;E).$ 
It is itself a Fr\'echet space, equipped with the family of semi-norms $\|f\|_{C(0,T;E),\gamma }:=\sup_{r\in I}\gamma (f_r),$ for any semi-norm $\gamma$ of $E.$
\\

\textbf{Controls and $p$-variation spaces.}
We will denote by $\Delta ,\Delta _2$ the simplices
\begin{equation}\label{nota:simplexes}
\begin{aligned}
&\Delta:=\{(s,t)\in [0,T]^2\,,\,s\leq t\}\,,
\\
&\Delta _2:= \{(s,\theta ,t)\in [0,T]^3\,,\,s\leq \theta \leq t\}\,.
\end{aligned}
\end{equation}
If $E$ is a vector space and $g\colon[0,T]\to E,$ we define a two-parameter element $\delta g$ as
\[
\delta g_{st}:= g_t - g_s,\quad  \text{for}\enskip  (s,t)\in\Delta .
\]
Similarly, we define another operation $\tilde\delta $  by letting,
for any $g:\Delta \to E,$   $\tilde \delta g$ be the quantity
\[
\tilde \delta g_{s\theta t}:=g_{st}-g_{s\theta }-g_{\theta t},\quad \text{for}\enskip (s,\theta ,t)\in\Delta _2,
\]
and we recall that $\mathrm{Ker}\tilde \delta=\mathrm{Im}\delta .$
As usual in the framework of controlled paths, we will omit the symbol $\tilde {\phantom{e}}$ on the second operation,
and write $\delta $ instead of $\tilde \delta .$

We call \emph{control} on $[0,T]$ any continuous, superadditive map $\omega \colon \Delta \to \R_+,$ 
namely $\omega $ is such that for all $(s,\theta ,t)\in \Delta _2$
\begin{equation}
\label{axiom:control}
\omega (s,\theta )+\omega (\theta ,t)\leq \omega (s,t)
\end{equation}
(this implies in particular that $\omega (t,t)=0$ for any $t\in[0,T]$).

If $E$ is equipped with a family of semi-norms, and $\alpha >0,$
we denote by 
$\V_ 1^{\alpha }(0,T;E)$
the set of continuous paths $g\colon[0,T]\to E$, such that for each semi-norm $\gamma$, there exist a control $\omega_\gamma :\Delta \to\R_+$ with
\begin{equation}
\label{def:omega_a}
\gamma \big(\delta g_{st}\big)\leq \omega_\gamma (s,t)^{\alpha  }\,,
\end{equation} 
for every $(s,t)\in\Delta .$
Similarly, we denote by $\V_ 2^{\alpha }(0,T;E)$ the set of $2$-index maps $g:\Delta \to E$ such that $g_{tt}=0$ for every $t\in [0,T]$ and 
\begin{equation}
\label{def:omega_a_2}
\gamma \big(g_{st}\big) \leq \omega_\gamma  (s,t)^{\alpha }\,,
\end{equation} 
for all $(s,t) \in \Delta ,$ and some family of controls $\omega_\gamma.$
If $E$ is a Banach space and $\gamma =|\cdot |_{E}$, one defines a norm $\n{\cdot }_{\V_2^\alpha }$ on $\V^\alpha _2(0,T;E)$ by taking the infimum of $\omega (0,T)^\alpha $ over every possible control $\omega $ such that \eqref{def:omega_a_2} holds. This quantity is in fact equal to the usual $q$-variation norm where $q:=\frac1\alpha $, as seen for instance in \cite[Lemma 3.2]{hocquet2017energy}.

By $ \V^{\alpha }_{2,\text{loc}}(0,T;E)$ we denote the space of maps $g:\Delta\to E$ such that there exists a countable covering $\{I_k\}_k$ of $I$ satisfying $g \in  \V^{\alpha }_2(I_k;E)$ for any $k$.
We also define the set $\V^{1+}_2(0,T;E)$ of ``negligible remainders'' as
\[
\V^{1+}_2(0,T;E):=\bigcup_{\alpha >1} \V_ 2^{\alpha }(0,T;E),
\]
and similarly for $\V^{1+}_{2,\text{loc}}(0,T;E).$

\subsection{Rough drivers}
\label{subsec:DRD}

Before giving definitions, let us quickly explain our approach.
For simplicity, let $A=0$, assume that $f$ is smooth and consider a family $ B_t :=(X_t(x)\cdot \nabla + X_t^0(x))$ of first-order differential operators, where for each $i=0\dots d,$ $X_t^i(x)$ is smooth with respect to $x$ (for fixed $t$), and $\alpha $-H\"older in $t$ for each $x$, while $\alpha >1/3.$
Integrating formally \eqref{rough_paths_eq} in time, we have
\begin{align}
\nonumber
u_t-u_s-\int_s^tf_r\d r 
&=( B_t- B_s) u_s + \int_s^t\d  B_r(u_r-u_s)
\\
\nonumber
&= \delta B_{st}u_s + \iint_{s<r_1<r<t}\d  B_r [\d  B_{r_1}u_{r_1} + \d{r_1}f_{r_1} ]\,.
\\
\nonumber
&= B^1_{st}u_s + \left(\iint_{s<r_1<r<t}\d  B_r \circ \d  B_{r_1}\right)u_s
\\
\nonumber
&\quad \quad + \left(\iint_{s<r_1<r<t}\d  B_r \circ \d  B_{r_1}\right)[u_{r_1}-u_s] + o(t-s)\,.
\end{align}
One expects any ``reasonable'' solution to satisfy an estimate of the form $|u_t-u_s|_{W^{-1,p}}\lesssim (t-s)^\alpha $, so that in particular
\[
\Big|\Big(\iint_{s<r_1<r<t}\d  B_r\circ\d  B_{r_1}\Big)(u_{r_1}-u_s)\Big|_{W^{-3,p}} \lesssim (t-s)^{3\alpha } =o(t-s)\,.
\]
Combined with the above, we thus find the Euler-Taylor type expansion
\begin{equation}
\label{abstract_edo}
u^{\natural}_{st}:=\delta u_{st}-\int_s^tf_r\d r - \left(B^{1}_{st} + B^2_{st}\right)u_s\enskip \in \enskip o(t-s),
\end{equation} 
where we introduce the two-index map $\B=(B^1,B^2)$ defined as
\begin{equation}
\label{intro:increment}
\left\{
\begin{aligned}
B^1_{st}
&:= B_t- B_s ,\quad 
\\
B^2_{st}
&:=\iint_{s<r_1<r<t}\d  B_r\circ \d  B_{r_1}\,,
\end{aligned}\right.
\quad \enskip 0\leq s\leq t\leq T\,.
\end{equation}
When $\alpha > 1/2,$ the operators $B^2_{st}$ are canonically defined via an immediate non-commutative generalization of Young Theorem \cite{young1936inequality}. This is in contrast with the case $\alpha \leq 1/2,$ where \eqref{intro:increment} does not make sense in general. Indeed, while the definition of $B^1_{st}$ seems not problematic for $ B$ continuous (just let $B_{st}^1:=\delta  B_{st})$, this is not the case of the second component in general. If $ B(n)\to  B$ uniformly on $[0,T],$ a limit point of $\{\iint_{s<r_1<r_2<t}\d  B_{r_2}(n)\circ\d  B_{r_1}(n),n\in\mathbb{N}\}$, if it exists, will depend on the choice of the approximating sequence. On the other hand, any limit ought to satisfy the constraint
\begin{equation}
\label{algebraic_relation}
B^2_{st}- B^2_{s\theta } - B^2_{\theta t} = B^1_{\theta t}\circ B^1_{s\theta},\quad \text{for any}\enskip 0\leq s\leq \theta \leq t\leq T,
\end{equation} 
which reflects the linearity of the integral, and its additivity with respect to the domain of integration.
An essential insight of rough paths theory is that, assuming that $B^2_{st}$ is given with \eqref{chen_relations_gene} together with suitable analytic conditions, then one can simply \emph{define} the solution $u$ to \eqref{rough_paths_eq} by the Euler-Taylor expansion \eqref{abstract_edo}.
Following Davie's interpretation of rough differential equations \cite{davie2007differential}, we will therefore say that $u$ is a solution to \eqref{rough_PDE_gene} if \eqref{abstract_edo} holds. The fact that such expansion is sufficient to fully caracterize the solution $u$ is not obvious, and is in fact a consequence of the so-called ``Sewing Lemma'', which for convenience will be stated in Proposition \ref{pro:sewing}.

\bigskip

The previous discussion depicts a non-commutative generalization of the usual rough paths theory, which has been already discussed e.g.\ in \cite{feyel2008non,coutin2014perturbed,bailleul2017unbounded,bailleul2019rough}. In this picture, real numbers -- in which the coordinates of a path $Z\colon[0,T]\to \R^m$ live -- are substituted by elements of an algebra (here a space of differential operators), and the constraint \eqref{algebraic_relation} corresponds to Chen's relations.
What plays here the role of the driving rough path for controlled differential equations is the pair $\B=(B^1,B^2).$ It is called an \emph{unbounded rough driver} (URD), and was first considered by Bailleul and Gubinelli \cite{bailleul2017unbounded} (see also \cite{deya2016priori,hocquet2017energy,hofmanova2018navier}). In the present work, we chose to restrict our attention to a subclass of URDs  that are given by differential operators. Such objects will be referred to as \emph{differential} URDs (or simply ``differential rough drivers'').
In the sequel we will denote by $\DD_i,i=1,2,$ the space of differential operators of order $i$, that is:
\begin{equation}
\left[\begin{aligned}
&\mathbb D_1:=\Big\{ X^i(x)\partial _i + Y (x),\enskip \text{such that}\enskip (X,Y)\in W^{3,\infty}\times W^{2,\infty}\Big\},
\\
&\mathbb D_2:=\Big\{ \XX ^{ij}(x)\partial _{ij}+\YY^i(x)\partial _i + \mathbb Z (x),
\\
&\quad \quad \quad \quad \quad \quad \quad 
\text{such that}\enskip 
(\XX,\YY ,\mathbb Z )\in W^{3,\infty}\times W^{2,\infty}\times W^{1,\infty}
\Big\}.
\end{aligned}\right.
\end{equation}
The space-regularity of the above coefficients is precisely enough to make sense of \eqref{rough_paths_eq} and obtain energy estimates for it.
It is indeed easily seen that the composition of two elements of $\DD_1$ is an element of $\DD_2,$ while we also have the property that
\[\mathbb D_i\subset \bigcap _{k=-3+i}^3\mathscr L(W^{k,p},W^{k-i,p})\quad  \text{for} \enskip i=1,2\enskip 
\text{and}\enskip p\in[1,\infty]\,.
\]
These properties which will be extensively used in the sequel.

We have the following definition.
\begin{definition}[unbounded rough driver]
	\label{def:rough_driver}
	Let $\alpha >1/3.$
	
	\noindent A $2$-index family $\B_{st}\equiv(B^1_{st},B^2_{st})_{(s,t)\in\Delta }$ of linear operators in $L^2(\R^d)$ is called a $\V^\alpha $-{\normalfont\textbf{unbounded rough driver}} if and only if: 
	\begin{enumerate}[label=(URD\arabic*)]
		\item \label{RD1}
		$B^i$ takes values in $\cap _{k=-3+i}^3\mathscr L(H^k,H^{k-i})$ for $i=1,2,$ and there exists a control $\omega_B :\Delta\to\R_+ $ such that
		\begin{equation}
		\label{bounds:rough_drivers}
		|B^i_{st}|_{\mathscr L(H^k,H^{k-i})}\leq\omega _B(s,t)^{i\alpha}\,, 
		\end{equation}
		for every $(s,t)\in\Delta ,$ any $i\in\{1,2\}$ and $k=-3+i,\dots,3.$
		\item \label{RD2}
		Chen's relations hold true, namely, for every $(s,\theta ,t)\in\Delta _2,$ we have in the sense of linear operators:
		\begin{equation}\label{chen}
		\delta B^1_{s\theta t}=0\,,\quad \delta B^2_{s\theta t}=B^1_{\theta t}\circ B^1_{s\theta } \,.
		\end{equation}
	\end{enumerate}
	
	Moreover, we will say that 
	\begin{enumerate}[label=(URD$\star$)]
		\item $\B$ is {\normalfont\textbf{differential}} if $\B$ is an unbounded rough driver such that 
		\[B^i_{st}\in \mathbb D_i,\quad  \text{for} \enskip i=1,2\enskip \text{and}\enskip (s,t)\in\Delta .\]
	\end{enumerate}
	
	Finally, let $(E_k)_{k\in I}$ be a scale such that there exists $p\in[1,\infty]$ with the property that
	$E_k\hookrightarrow W^{k,p}$ for each $k\in I.$ We will say that
	\begin{enumerate}[label=(URD$\star\star$)]
		\item \label{acts_Ek}
		$\B$  {\normalfont\textbf{acts on the scale $(E_k)_{-3\leq k\leq 3}$}} if
		\[
		B^i_{st}E_k\subset E_{k-i}\,,\quad  -3\leq k-i\leq k\leq 3,\enskip (s,t)\in\Delta ,
		\]
		and if the estimate \eqref{bounds:rough_drivers} is satisfied with $(H^k)$ being replaced by $(E_k).$
	\end{enumerate}
\end{definition}
\begin{remark}
	Regarding the definition of $\DD_i$ for $i=1,2,$ any differential, unbounded rough driver can in fact be extended to a family of differential operators acting on the Sobolev scale $(W^{k,p})_{-3\leq k\leq 3},$ for each $p\in [1,\infty].$ For simplicity, in the following we will use the same symbol $\B$ for every such extension.
\end{remark}

Note that, if $ B\colon[0,T]\to \DD_1$ is a continuous path with finite variation (with respect to the operator-norm of $\cap _{k=-2}^3\mathscr L(H^k,H^{k-1})$), 
one can always define the \emph{canonical lift} $S_2(B)$ 
as the differential rough driver $\B\equiv(B^1,B^2)$ given by
\begin{equation}
\label{can_lift}
S_2(B):=
\B\enskip \text{with}\enskip 
\left\{\begin{aligned}
&B^1_{st}:= B_t-B_s \in \DD_1
\quad 
\text{and}
\\
&B^2_{st}:= \int_s^t \d B_r\circ (B_r-B_s)\in \DD_2\,.
\end{aligned}\right.
\end{equation}  
The above integral is well-defined in the sense of Riemann-Stieltjes, in the space $\DD_2$ endowed with the natural operator-norm topology.

This basic observation leads us to the following definition.
\begin{definition}[Geometric differential rough driver]
	\label{def:geometric}
	Let $p\in[1,\infty].$
	Given a differential rough driver $\B$ with regularity $\alpha >1/3,$
	we will say that $\B$ is {\normalfont\textbf{geometric}} if there exists a sequence of paths $ B(n) \in C^1(0,T;\DD_1),n\geq 0,$  such that letting 
	\[\B(n):=S_2( B(n))\,,\]
	it holds
	\begin{multline}
	\label{topology_B}
	\rho _{\alpha } (\B(n),\B)\enskip :=\enskip 
	\sum\nolimits _{k=-2}^3\| B(n)-B\|_{L^\infty(0,T;\mathscr L(H^k,H^{k-1}))}
	\\
	+\sum\nolimits_{i=1}^2\sum\nolimits_{k=-3+i}^{3}\n{B^{i}(n)-B^{i}}_{\V^{i\alpha}(0,T;\mathscr L(H^{k},H^{k-i}))}
	\underset{n\to\infty}{\longrightarrow}0.
	\end{multline}
\end{definition}

\begin{example}
	\label{ex:URD_gene}
	Recall that a continuous, $m$-dimensional, $q$-rough path with $q=\frac1\alpha ,$ is a pair \begin{equation}
	\label{analytic_conditions}
	\mathbf Z\equiv(Z_{st}^{1,\mu },Z_{st}^{2,\mu \nu  })_{\substack{1\leq \mu , \nu  \leq m\\(s,t)\in \Delta }}\quad \text{in}\quad \V^{\alpha }_2(0,T;\R^{m}) \times \V^{2\alpha}_2(0,T;\R^{m\times m}),                                                                                                                                                                                                                                                  \end{equation}
	such that Chen's relations  hold, namely:
	\begin{equation}\label{chen_relations_gene}
	\delta Z_{s\theta t}^{1,\mu } =0\,,\quad \delta Z_{s\theta t}^{2,\mu \nu  }=Z_{s\theta }^{1,\mu } Z_{\theta t}^{1,\nu } \,,\quad\text{for}\enskip (s,\theta ,t)\in\Delta _2\,,\enskip 
	\enskip 1\leq \mu , \nu  \leq m.
	\end{equation}
	Roughy speaking, the relations \eqref{chen_relations_gene} indicate that $Z^{1,\mu }_{st}$ has the form $Z_t^\mu -Z_s^\mu \equiv \int_s^t\d Z^\mu _r$ for some path $Z\colon[0,T]\to \R^m$ while $Z^{2,\mu \nu }_{st}$ should be thought of as a prescribed value for $\iint_{s<r_2<r_1<t}\d Z^\nu_{r_1} \d Z^\mu_{r_2} .$ If $Z$ is smooth, we can define a canonical lift $\Z$ via \eqref{can_lift}, replacing the operation $\circ$ by the tensor product.
	By definition, the set of geometric rough paths corresponds to the closure of such canonical lifts, with respect to the natural $q$-variation metric. We refer the reader to the monographs \cite{lyons2002system,friz2010multidimensional,friz2014course} for a thorough introduction to geometric rough paths.

	Now, consider a rough path $\Z$,
	and let $\sigma \in W^{3,\infty}(\R^d;\R^{m\times d})$, $\rho \in W^{2,\infty}(\R^d,\R^m),$
	and for $(s,t)\in\Delta ,$ $i=1,2,$ define $\B\equiv(B^1,B^2)$ as:
	\[
	\left\{\begin{aligned}
	B^1_{st} 
	&:= Z^{1,\mu}_{st} (\sigma ^{\mu}_j  \partial _j    + \rho ^\mu  ),\quad
	\\
	B^2_{st}
	&:=Z_{st}^{2,\mu\nu }(\sigma ^{\mu }_j \partial _j +\rho ^{\mu})(\sigma ^{\nu }_i \partial _{i }  + \rho ^{\nu } ),
	\end{aligned}\right.
	\]
	for every $(s,t)\in\Delta .$
	It is straighforward to check that $\B$ satisfies \ref{RD1}-\ref{RD2}. Hence it is a differential rough driver. Moreover, it is geometric if $\Z$ is geometric.
\end{example}

Given $B\in\V^\alpha (0,T;\DD_1),$ by definition of $\DD_1$ it is always possible to write $B_t$ in terms of some family of bounded and measurable coefficients $X^i_t(x),$ $i=0,\dots d$ so that
\begin{equation}
\label{form_B}
\begin{aligned}
B_t
&:=X^i_t(x)\partial _i +X^0_t(x)\,.
\end{aligned}
\end{equation} 
In Appendix \ref{app:algebraic}, we shall see that there is a one-to-one correspondence between coefficients and elements of $\DD_1,$ and that it yields a continuous isomorphism, see \eqref{iso_j}. In particular, we can assume without any loss of generality that $X^i\in \V^\alpha (0,T; W^{3,\infty}),i=1,\dots ,d,$ while $X^0\in \V^\alpha (0,T; W^{2,\infty}).$ For notational convenience, in the remainder of the paper we shall assume that $B$ has the form \eqref{form_B} . Moreover, we will make use of the shorthand notation
\[
X_{st}:=X_t-X_s\,,
\]
hence blurring the difference between the value $X_t(x)$ of the coefficient path associated with $B_t$ and that of its increments $\delta X_{st}(x).$

It turns out that, for geometric differential rough drivers, there is an ensemble of very convenient algebraic rules, as illustrated in the following result.
We insist on the fact that these rules are a consequence of the geometricity assumption: no further assumption is required on $\B$.
The proof of the following lemma is rather simple and merely algebraic, hence we postpone it until Appendix \ref{app:algebraic}.

\begin{lemma}
	\label{lem:multiplication}
	Let $\B$ be a geometric differential rough driver such that $B^1_{st}=\delta B_{st}$ where $B_t$ is as before.
	The following assertions are true:
	\begin{enumerate}[label=(\roman*)]
		\item \label{weak_G_1} (Weak geometricity I)
		There exist coefficients $\LL^i\in \V_2^{2\alpha }(W^{2,\infty}),$ $i=0,\dots d$ such that 
		\begin{equation}
		\label{B2}
		B^2_{st}= \frac12X_{st}^iX_{st}^j\partial _{ij} + \left(\LL^i_{st} + X^0_{st}X^i_{st} \right)\partial _i + \LL^0_{st} + \frac12 (X^0_{st})^2\,.
		\end{equation} 
		
		\item (Generalized Chen's relations) For each $(s,\theta ,t)\in \Delta _2,$ a.e.\ in $\R^d$, it holds
		\begin{equation}
		\label{gene:chen}
		\delta \LL_{s\theta t}^i 
		= X_{\theta t}^j\partial _j(X_{s\theta }^i)\,,
		\quad i=0,\dots d\,.
		\end{equation}
		
		\item \label{weak_G_2} (Weak geometricity II) We have
		\[ B^2_{st}=\frac12B^1_{st}\circ B^1_{st} + [\B]_{st}
		\]
		where the ``bracket''  $[\B]$ is a family of first-order differential operators, explicitly given by:
		\[
		[\B]_{st}:= \left(\LL^i _{st}-\frac12X^j_{st}\partial _j X_{st}^i\right)\partial _i+ \LL^0_{st}-\frac12 X^j_{st}\partial _jX^0_{st}\,.
		\]
	\end{enumerate}
\end{lemma}

\begin{notation}
	\label{not:summarize}
	For convenience, we will summarize the above properties by using the shorthand notation
	\begin{equation}
	\label{shorthand_notation}
	\B \sim \X=(X^i,\LL^i)_{i=0,\dots,d}.
	\end{equation}
\end{notation}

\begin{remark}
	If $(t\mapsto X_t\in W^{3,\infty})$ has finite variation, $\LL$ is explicitly given as
	\[
	\begin{aligned}
	\LL^i_{st}
	:=\int_s^t \d X_r \cdot \nabla (X^i_{sr}),\quad i=0,\dots d\,.
	\end{aligned}
	\]
	Roughly speaking, $\LL$ can be thought of as a  differential rough driver analogue of the usual \emph{L\'evy area} for rough paths, in the sense that the knowledge of $\LL$ is enough to compute the second level $B^2$ of $\B,$ as is the case for a geometric rough path (see \cite[Definition 13.2]{friz2010multidimensional}).
	
	In fact, if $\B$ is the pair defined in Example \ref{ex:URD_gene} with $\Z$ geometric and $\rho =0$, a routine calculation shows the identity
	\[
	\LL_{st}\cdot \nabla =\frac12Z^\mu_{st} Z^\nu _{st}(\sigma ^\mu \cdot \nabla \sigma ^\nu )\cdot \nabla 
	+\frac12\A_{st}^{\mu \nu }[\sigma ^\mu\cdot \nabla  ,\sigma ^{\nu }\cdot \nabla ]
	\]
	where we denote by $\sigma ^\mu \cdot \nabla :=\sigma ^{\mu i}\partial _i,$ while $\mathbb A_{st}$ is the L\'evy area of $\Z,$ and $[\cdot ,\cdot ]$ denotes the usual Lie bracket of vector fields.
\end{remark}

\begin{remark}
	As for the usual geometric rough paths, the question may arise whether the algebraic constraints \ref{weak_G_1} and \ref{weak_G_2} imply the geometricity of $\B$ (see \cite[Chapter 9]{friz2010multidimensional}). We conjecture that, upon taking $\alpha $ slighlty smaller, and under ``reasonable'' conditions on the regularity of the coefficients, the answer should be positive. However, we prefer to leave this issue for future investigations.
\end{remark}

\subsection{Notions of solution}
\label{subsec:notion}

In the whole paper, we consider an ansatz of the form
\begin{equation}
\begin{aligned}
\label{rough_PDE_gene}
\d v=(\partial _if^i +f^0)\d t+\d\B  g,\quad 
\text{on}\enskip [0,T]\times \R^d,
\\
v_0=v^0\in L^p\,,
\end{aligned}
\end{equation}
with $\B \sim (X^i,\LL^i)_{i=0,\dots,d}$ being geometric. The drift term $f^i,i=0,\dots d$ is $p$-integrable as a mapping from $[0,T]$ into $L^{p}$ for some $p\in[1,\infty),$ and the derivation $\partial _i=\frac{\partial }{\partial x_i}$ is understood in distribution sense. By assumption, $g$ will be \emph{controlled} by $B_t$, and so the solution $v$ should be. This means the following.

\begin{definition}
	\label{def:controlled_by}
	Given $g\in L^\infty(0,T;L^p) \cap \V^\alpha (0,T;W^{-1,p})$ we will say that $g$ is {\normalfont\textbf{controlled by $B$}}, if there exists $g'\in L^\infty(0,T;L^p) \cap \V^{\alpha }(0,T;W^{-1,p})$ such that the element $R^g$ of $\V_2^\alpha (0,T;W^{-1,p})$ defined as
	\begin{equation}
	\label{R_g}
	R^g_{st}:=\delta g_{st}-B_{st}^1g_s',\quad \text{for every}\enskip (s,t)\in\Delta ,
	\end{equation}
	verifies
	\begin{equation}
	\|R^g\|_{\V_2^{2\alpha }(0,T;W^{-2,p})}<\infty
	\end{equation}
	(notice the loss of a space-derivative in the above).
	Abusively, we call $g'$ ``the Gubinelli derivative'' of $g,$ though $g'$ could be non-unique in principle (at least without any further assumption on $B,$ see Section \ref{app:true_roughness}).
\end{definition}

That the unknown $v$ should be controlled by $B$ implies in particular boundedness for the path $v\colon[0,T]\to L^p$ and also weak-star continuity (hence allowing to give a meaning to the initial condition). In a large part of the sequel we will encounter the situation where $v=g=g'$ but this fact is not needed in the definition of a solution, so we will keep things on the more general form \eqref{rough_PDE_gene} for the moment.

The following notion of solution was introduced in \cite{bailleul2017unbounded}, see also \cite{deya2016priori}.
\begin{definition}[weak-solution]
	\label{def:weak_sol}
	Let $T>0$, $\alpha \in(1/3,1/2]$ and fix $p\in [1,\infty].$
	Assume that we are given
	$f^i\in L^1(0,T;L^p),$ $i=0,\dots d,$ and that $g$ is controlled by $B$ with Gubinelli derivative $g',$
	with $g,g'$ both belonging to $L^\infty(0,T;L^p).$
	A mapping $v\colon[0,T]\to L^p$ is called an {\normalfont\textbf{$L^p$-weak solution}} to the
	rough PDE \eqref{rough_PDE_gene} if it fulfills the following conditions
	\begin{enumerate}[label=(\arabic*)]
		\item $v\colon[0,T]\to L^p$ is bounded as a path taking values in $L^p$; moreover, $v$ belongs to $\V^\alpha (0,T;W^{-1,p})$;
		\item for every $\phi \in L^{p'}$ with $1/p+1/p'=1,$ $\lim_{t\to 0}\int_{\R^d}(v_t-v^0)\phi \d x= 0$;
		\item for every $\phi \in W^{2,p'}$ with $1/p+1/p'=1$, and every $(s,t)\in\Delta :$
		\begin{equation}\label{nota:solution}
		\int_{\R^d}\delta v_{st}\phi \d x=\iint_{[s,t]\times\R^d} (-f^i\partial _i\phi + f^0\phi ) \d x\d r+ \int_{\R^d}\left(g_sB^{1,*}_{st}\phi +g'_sB^{2,*}_{st}\phi\right)\d x +\langle v_{st}^\natural,\phi \rangle\,,
		\end{equation}
		for some $v^\natural\in \V_{2,\loc}^{1+}(0,T;W^{-3,p}).$
	\end{enumerate}
\end{definition}

The notion of weak solution fulfills the minimal requirements under which remainder estimates (and thus estimates on rough integrals) can be obtained, see Proposition \ref{pro:apriori}. In the sequel however, we will mostly work in a parabolic context, where solutions happen to live in a ``better space'' than the one described above.
This motivates the introduction of the following.

\begin{definition}[Energy solution]
	\label{def:var_sol}
	Letting $p,p'\in[1,\infty]$ so that $1/p+1/p'=1,$ we will say that
	$v$ is an {\normalfont\textbf{$L^p$-energy solution}} of \eqref{rough_PDE_gene} if it is a weak solution such that additionally
	\begin{equation}
	\label{Lp_sol} 
	v\in L^p(0,T; W^{1,p}).
	\end{equation} 
	
	Similarly, we will say that $v$ is a $L^p_\loc$-energy solution (or $L^p(U)$-energy solution if $U\subset \R^d$) if it fulfills the above properties, where each occurence of the Sobolev spaces in the space-like variable is replaced by its local counterpart.
\end{definition}

\subsection{Rough parabolic equations}
In this section, we consider the rough parabolic equation
\begin{equation}
\label{rough_parabolic}
\left\{
\begin{aligned}
&\d u_t-(A_tu +f_t(x))\d t= \left(\d \X_t^i\partial _i +\d \X_t^0\right)u_t \enskip,
\quad \text{on}\enskip (0,T]\times \R^d\,,
\\
&u_0=u^0\in L^p(\R^d)\,,
\end{aligned}\right.
\end{equation}
where 
\begin{equation}
\label{A}
A_t=\partial _i(a^{ij}(t,\cdot )\partial _j\, \cdot \,)
\end{equation}
is given, and we assume the following on $a.$

\begin{assumption}
	\label{ass:A}
	The coefficients $a=(a^{ij})_{1\leq i,j\leq d}$ are measurable, symmetric in $i$ and $j$ and moreover there
	exists a constant $\lambda>0$ such that for a.e.\ $(t,x)\in [0,T]\times \R^d:$
	\begin{equation}\label{uniform_ellipticity}
	\lambda \sum\nolimits_{i=1}^d\xi _i^2\leq \sum\nolimits_{1\leq i,j\leq d} a^{ij}(t,x)\xi _i\xi _j\leq \lambda ^{-1}\sum\nolimits_{i=1}^d\xi _i^2
	\,,\quad \text{for all}\enskip \xi \in \R^d\,.
	\end{equation}
\end{assumption}

Concerning the rough part, the following hypotheses will be assumed throughout the paper.

\begin{assumption}
	\label{ass:B}
	For some fixed $\alpha >1/3,$ we are given a coefficient path
	\[
	\big(t\mapsto (X_t^i)_{i=0,\dots ,d}\big) \in \V^\alpha \left(0,T;W^{2,\infty}\times (W^{3,\infty})^d\right)\,,
	\]
	while $(t\mapsto X^0) \in \V^\alpha (0,T;W^{2,\infty}).$
	These coefficients are given together with a two parameter family 
	\[
	\big((s,t)\mapsto (\LL_{st}^i)_{i=0,\dots,d})\big)\in \V^{2\alpha }_2\left(0,T;W^{1,\infty}\times (W^{2,\infty})^d\right)
	\]
	which satisfies the generalized Chen's relation 
	\[
	\delta \LL^i_{s\theta t}=X^j_{\theta t}\partial _j(X_{s\theta }^i)
	\]
	for each $0\leq s\leq\theta \leq t\leq T$ and $i=0,\dots ,d.$

	We then let $B_t:=X_t\cdot \nabla +X^0_t$ 
	and define a differential rough driver $\B\sim \X=(X,\LL)$ as in Lemma \ref{lem:multiplication}. Hence, it corresponds to the pair $(B^1,B^2)$ where
	\[\left[\begin{aligned}
	&B^1_{st}=B_t-B_s=
	X_{st}^i\partial _i + X^0_{st}\,,
	\\
	&B^2_{st}= \frac12X^i_{st}X^j_{st}\partial _{ij} +(\LL^i_{st} + X^0_{st}X^i_{st})\partial _i + \LL^0_{st} + \frac12(X^0_{st})^2\,\quad 
	\quad (s,t)\in \Delta ,
	\end{aligned}\right.
	\]
	where we recall that $X_{st}:=X_t-X_s.$
	
	Furthermore, we assume that $\B$ is geometric, in the sense of Definition \ref{def:geometric}.
\end{assumption}

Our first result is about the solvability of \eqref{rough_parabolic}, and completes the results obtained in the previous work \cite{hocquet2017energy}.
The proof will be given in Section \ref{sec:free}.
\begin{theorem}
	\label{thm:free_intro}
	Let $f\in L^2(0,T;H^{-1}),$ fix $u_0\in L^2$ and consider a geometric, differential rough driver $\B\sim (X^i,\LL^i)_{i=0,\dots,d}$ as in Assumption \ref{ass:B}.
	There exists a unique $L^2$-energy solution $u=u(u_0,f;\B)$ to \eqref{rough_parabolic}. 
	
	In addition, the solution map is continuous in the following sense
	\begin{enumerate}[label=(\arabic*)]
		\item
		for every $(u_0,f)\in L^2\times L^2(H^{-1})$, the map $\B\mapsto u(u_0,f;\B) \in L^\infty(L^2)\cap L^2(H^1)$ is weakly-star continuous with respect to the rough driver distance $\rho _\alpha $ introduced in \eqref{topology_B}.
		\item
		for $\B$ fixed the map $u(\cdot ,\cdot ;\B):L^2\times L^2(H^{-1}) \to L^\infty(L^2)\cap L^2(H^1)$ is continuous, with respect to the strong topologies.
	\end{enumerate}
\end{theorem}

Before we state our second main result, we shall first define a set of admissible functions $F:\R\to\R$ for which right-composition with a solution is possible.
We let
\begin{equation}
C^2_\mathrm{adm}:=\{F\in C^2(\R;\R),\,\text{s.t.\ }F(0)=F'(0)= 0 \enskip \text{and}\enskip |F''|_{L^\infty}<\infty\}.
\end{equation}
With this definition, we have the following result.

\begin{theorem}[It\^o Formula]
	\label{thm:ito_transport}
	Let $A$ satisfying Assumption \ref{ass:A}, let $\B\sim (X,\LL)$ such that Assumption \ref{ass:B} holds with $X^0=0$.
	Let $u$ be an $L^2$-energy solution of \eqref{rough_parabolic}.
	The following assertions are true.
	
	\begin{enumerate}[label=(\roman*)]
		\item For every $F\in C^2_\mathrm{adm}$ it holds the chain rule
		\begin{equation}
		\label{chain_rule_thm}
		\d F(u) = F'(u)(Au + f) \d t + \d\B  [F(u)],
		\end{equation}
		in the sense that the path $[0,T]\to L^1,$ $t\mapsto F(u_t)$ is controlled by $B$ with Gubinelli derivative $(F(u_t))'=F(u_t)$ and is an $L^1$-energy solution to the above equation.
		More explicitly, we have for any $\phi \in W^{3,\infty}$ and $0\leq s\leq t\leq T:$
		\begin{multline}
		\label{chain_rule_thm_expl}
		\int_{\R^d}\delta F(u)_{st}\phi\d x +\iint_{[s,t]\times\R^d} [F''(u)a^{ij}\partial _iu\partial _ju\phi 
		+F'(u)a^{ij}\partial _iu\partial _j\phi  ] \d x\d r 
		\\
		= \int_{\R^d}F(u_s)(B^{2,*}_{st}+B^{2,*}_{st})\phi \d x + \langle F^{\natural}_{st},\phi \rangle
		\end{multline}
		for a uniquely determined remainder term $F^\natural\in\V_{2,\loc}^{1+}(0,T;W^{-3,1}).$
		
		\item
		If $F\in C^2,$ then \eqref{chain_rule_thm} holds locally. Namely,  $t\mapsto F(u_t)$ is controlled by $B$ in the $L^1_{\loc}$-sense while \eqref{chain_rule_thm_expl}
		is true for any $\phi \in W^{2,\infty}_\loc$ and a remainder $F^\natural$ in $\V_{2,\loc}^{1+}(0,T;W^{-3,1}_\loc).$
	\end{enumerate}
\end{theorem}

\begin{remark}
	\label{rem:not_trivial}
	The formula \eqref{chain_rule_thm} is by no means trivial, no matter how smooth $F\colon \R\to \R$ is as a function.
	In fact, the rough integral \begin{equation}
	\label{integral_F_u}
	\int_s^tDF(u_r)[\d u_r]
	\end{equation}
	is not even well-defined a priori for an $L^2$-energy solution $u$ of \eqref{rough_parabolic}, and this is so regardless of the regularity of $F.$
	
	To wit, note that the expression \eqref{integral_F_u} implicitly assumes that $u\colon[0,T]\to L^2$ is enhanced to a rough path $\mathfrak u=(\mathfrak u^1,\mathfrak u^2)$. 
	In particular, one aims to find a topological vector space $K$ such that $L^2$ is continuously embedded in $K$ and such that $\mathfrak u^i\colon[0,T]^2\to K^{\otimes i},$ for $i=1,2.$ 
	Leaving aside the question of the choice of tensor product for $K^{\otimes 2}$ (and whether a sense can be given to the rough integral $\mathfrak u^2_{st}\equiv\int_s^t\delta u_{sr}\otimes \d u_r$ in $K^{\otimes 2}$), we see that $K$ must be chosen such that
	\begin{align}
	\label{cond:u1}
	\mathfrak u^1\equiv\delta u \in \V_2^{\alpha }(0,T;K)\,.
	\end{align}
	For an $L^2$-energy solution $u$, we only expect that $\delta u\in \V^\alpha _2(0,T;H^{-1})$ (see Section \ref{sec:controlled_paths}), and hence the condition \eqref{cond:u1} imposes that $H^{-1}\hookrightarrow K.$
	In particular, this requires that the nonlinear operator 
	\[
	\bar F:H^{-1}\to L^1,\quad u\mapsto \bar F(u):=F(u(\cdot ))
	\]
	be of class $C^1,$ which is cleary not the case of any smooth function $F$.
\end{remark}

A core argument in the proof of Theorem \ref{thm:ito_transport} is the fact that for an appropriate subclass of free terms $f,$ the solutions of \eqref{rough_parabolic} are bounded. This is stated in the following result.

\begin{theorem}
	\label{thm:boundedness}
	Let 
	\[
	f\in L^r(0,T;L^q)\,,
	\]
	where the exponents $r\in(1,\infty]$ and $q\in(1\vee\frac{d}{2},\infty)$ are subject to the conditions 
	\begin{equation}
	\label{cond:r_q_strict_intro}
	\frac{1}{r} + \frac{d}{2q}<1.
	\end{equation}
	Then, the solution $u$ obtained from Theorem \ref{thm:free_intro} is locally bounded, away from $t=0$. Precisely, for any $\tau >0$, and any compact set $K\subset\subset\R^d$,
	it holds the estimate
	\[
	\|u\|_{L^\infty([\tau ,T]\times K)}\leq C\left(\tau ,K,|u_0|_{L^2},\lambda ,\|f\|_{L^r(L^q)},\omega _B,\alpha ,r,q\right)\,\,,
	\]
	where the above constant only depends on the indicated quantities.
\end{theorem}

Note that the chain rule given in Theorem \ref{thm:ito_transport} does not apply directly for the $L^p$-norm case since $F=|\cdot |^p$ is not admissible. 
Fortunately, we can show the following.

\begin{corollary}
	\label{cor:L_p_transport}
	Let $p\geq 2,$ $\B \sim \X=(X,\LL)$ be as in Theorem \ref{thm:ito_transport}, and take $f\in L^p(0,T;W^{-1,p}).$
	Assume that $u$ is an $L^p$-energy solution of \eqref{rough_parabolic}.
	
	Then, $|u|^p$ is an $L^1$-energy solution of
	\begin{equation}
	\label{ito_p_transport}
	\d (|u|^p) =p u|u|^{p-2}(Au + f)\d t+\d \X\cdot \nabla (|u|^p ).
	\end{equation}
\end{corollary}

In general, when $\B$ is geometric and such that $X^0\neq0,$ we can write a similar chain rule for the $L^p$-norm of $u,$ assuming that $p\geq 4.$
This is stated in the next result.

\begin{theorem}
	\label{thm:L_p}
	Fix $p\geq 4,$ and assume that $\B \sim (X,\LL)$ satisfies Assumption \ref{ass:B}.
	For every $f\in L^1(0,T;W^{-1,p})\cap L^2(0,T;H^{-1}) $ and $u_0\in L^p,$ there exists a unique $L^2$-energy solution $u$ to \eqref{rough_parabolic} such that $\iint_{[0,T]\times\R^d}|u|^{p-2}|\nabla u|^2\d x\d t<\infty$.
	
	Moreover, it holds in the $L^1$-sense:
	\begin{equation}
	\label{ito_power}
	\d (|u|^p) =p u|u|^{p-2}(Au + f)\d t+\d \B^{(p)} (|u|^p ).
	\end{equation}
	where $\B^{(p)}$ is given by
	\[
	\left\{\begin{aligned}
	&B^{(p),1}_{st}:= X^i_{st}\partial _i  + pX_{st}^0
	\\
	&B^{(p),2}_{st}:=\frac12X_{st}^iX_{st}^j\partial _{ij} + (\LL^i_{st} + pX_{st}^0X^i_{st})\partial _i  + p\LL_{st}^0+ \frac{p^2}{2}(X_{st}^0)^2.
	\end{aligned}\right.
	\]
\end{theorem}

\begin{remark}
	The previous theorem implies in particular that $L^p$-energy solutions are unique, since in that case, H\"older Inequality yields
	\[
	\||u|^{p-2}|\nabla u|^2\|_{L^1(0,T;L^1)}\leq \left(\sup_{t\in[0,T]}|u_t|_{L^p}^{p-2}\right)\int_{0}^T |\nabla u_t|^2_{L^p}dt\,,
	\]
	and the above right hand side is finite by assumption. However the existence of $L^p$-energy solutions is not guaranteed without any additional assumption.
\end{remark}

We now give a by-product of our results concerning the following homogeneous Dirichlet problem with transport noise
\begin{equation}
\label{dirichlet}
\left\{
\begin{aligned}
&\d u_t-A_tu\d t=\d\Z^\mu_t \sigma ^\mu (x)\cdot \nabla  u_t\enskip, \quad \text{on}\enskip \R_+\times D\,,
\\
&u(0)=u_0\,,
\\
&u_t|_{\partial D}=0\quad \text{(trace sense)},\quad \text{for all}\enskip t\geq 0\,,
\end{aligned}\right.
\end{equation}
where $\Z^\mu \sigma ^\mu$ is given the enhancement of Example \ref{ex:URD_gene} with $\rho =0$ and where $\Z$ is geometric. Moreover, we assume that the coefficients $\sigma ^\mu , \mu =1,\dots,d,$ have compact support in $D.$
With this assumption, it is easily seen that $\B$ acts on the scales $(W_0^{k,p}(D))_{-3\leq k\leq 3}$ for any $p\in[1,\infty],$ in the sense of Definition \ref{def:rough_driver}-\ref{acts_Ek}.

We have the following result.
\begin{theorem}[weak maximum principle for \eqref{dirichlet}]
	\label{thm:max_principle}
	Assume that $ D\subset\R^d$ is an open domain which is smooth and bounded.
	Let $A$ be as in Assumption \ref{ass:A} and define $\Z\sigma \cdot \nabla$ as above.
	Assume furthermore that
	\begin{equation}
	\label{sigma_boundary}
	\sigma \in W^{3,\infty}_0(D;\R^{m\times d}).
	\end{equation} 
	There exists a unique solution $u$ of the Dirichlet problem \eqref{dirichlet}, by which we mean that $u$ is an $L^2(D)$-energy solution with the following additional property
	\begin{equation}
	\label{boundary_u}
	u\in L^2(0,T; W^{1,2}_0( D)).
	\end{equation} 
	
	Moreover, $u$ belongs to $\in L^\infty([0,T]\times D)$ and we have the following maximum principle for $u$:
	\begin{equation}
	\label{comparison}
	\min\left(0,\essinf\nolimits_{ D} u_0\right)\leq u(t,x)\leq \max\left(0,\esssup\nolimits_{ D} u_0\right)\quad \text{a.e.\ for}\enskip (t,x)\in\enskip [0,T]\times  D.
	\end{equation}
\end{theorem}

\section{Controlled paths}
\label{sec:controlled_paths}
\subsection{Some useful facts}

For pedagogical purposes, we first recall some elements of Rough Path Theory from the point of view adopted in \cite{gubinelli2004controlling}.
The main problem addressed by this theory is, roughly speaking, to give a meaning to incremental equations of the form 
\begin{equation}
\label{generic_RP}
u_t-u_s= \int_s^t H,\quad \text{for}\enskip (s,t)\in\Delta ,\quad \text{(}u_0\enskip \text{given),}
\end{equation}
where $\Delta \ni (s,t)\mapsto H_{st}$ is a ``jet'' associated to the quantity one wishes to integrate.
A concrete example is given by the Riemmann-Stieljes integral
$\int_s^tH\equiv\int_s^tf_r\d Z_r$
where $f$ and $Z$ are $\alpha $-H\"older with $\alpha >1/2,$ an associated first order approximation of which is provided by the jet
\begin{equation}
\label{H_young}
H_{st}:= f_s \delta Z_{st}\,.
\end{equation}
The value of $\int_s^tf\d Z$ is obtained by taking the limit of the Riemann sums $\sum_{i=1}^n H_{t_it_{i+1}}$ as $n\to\infty$ and $\max|t_{i+1}-t_i|\to0.$
Suppose now that the integrand $f$ is itself expressed as an integral against $Z,$ say $\delta f_{st}:=\int_s^tg\d Z$ for some $g\in C^1.$ Then, a better approximation of the former integral is given by the Milstein-type jet
\begin{equation}
\label{H_compensated}
\tilde H_{st}:=f_s\delta Z_{st} + g_s\int_s^t\delta Z_{sr}\d Z_r,
\end{equation}
as easily seen by Taylor formula.
When $\alpha\leq 1/2,$ the first choice \eqref{H_young} may generate divergent Riemann sums, which leads us to investigate generalizations of \eqref{H_compensated}. If $Z$ is endowed with an enhancement to a rough path $\mathbf Z\equiv (Z^1,Z^2),$ and if we replace the iterated integral in \eqref{H_compensated} by its postulated value $Z^2_{st}$,
the expression \eqref{H_compensated} is still meaningful.

The so-called Sewing Lemma \cite{gubinelli2004controlling} asserts that if $\alpha >1/3,$ then there is a unique couple $(u,u^\natural)$ such that $u_t-u_s=\tilde H_{st}+ u^\natural_{st}$ and 
\begin{equation}
\label{u_st_H}
|u^\natural_{st}|\lesssim (t-s)^{3\alpha }[\delta \tilde H]_{3\alpha }\,,
\end{equation} 
where $[\delta \tilde H]_{3\alpha }$ is the generalized $3\alpha $-H\"older seminorm of the 3-parameter quantity
\[
\delta \tilde H_{s\theta t}\equiv \tilde H_{st}-\tilde H_{s\theta }-\tilde H_{\theta t},\quad (s,\theta ,t)\in\Delta _2\,.
\]
The quantity $\int_s^t H:=H_{st}+u^\natural_{st}$ is called the \emph{rough integral} of $H,$ and it is consistent with usual Riemann-Stieljes integration when $H_{st}=f_s\delta Z_{st}.$

The following result, which is of fundamental importance in this paper, summarizes what we discussed above.
In the statement below, we assume for simplicity that $E$ is a Banach space, but it could easily be replaced by a Fr\'echet space (e.g.\ the Sobolev spaces $W^{k,p}_{\loc}$, or the Schwartz distributions), with $\omega $ being dependent on the semi-norm considered.

\begin{proposition}[Sewing Lemma]
	\label{pro:sewing}
	Let $H \colon \Delta \rightarrow E$ and $C>0$ be such that 
	\begin{equation}
	\label{a_gamma}
	\left|\delta H_{s\theta t}\right|\leq C\omega (s,t)^{a }\,
	, \quad 0 \leq s \leq \theta \leq t \leq T
	\end{equation}
	for some $a > 1$, and some control function $\omega ,$ and denote by $[\delta H]_{a,\omega }$ the smallest possible constant $C$ in the above bound. 
	
	There exists a unique pair $I\colon [0,T] \rightarrow E$ and $I^{\natural} : \Delta \rightarrow E$ satisfying
	\[
	\delta I_{st} = H_{st} + I_{st}^{\natural}
	\]
	where for $0\leq s\leq t\leq T$,
	\[
	|I_{st}^{\natural}| \leq C_a [\delta H]_{a,\omega } \omega (s,t)^{a}\,,
	\]
	for some constant $C_a$ only depending on $a$. In fact, $I$ is defined via the Riemann type integral approximation
	\begin{equation} \label{RiemannSum}
	I_t = \lim \sum_{ i=1 }^n H_{t^n_i t^n_{i+1}} \,,
	\end{equation}
	the above limit being taken along any sequence of partitions $\{t^n,n\geq 0\}$ of $[0,t]$ whose mesh-size converges to $0$.
\end{proposition}

Besides rough integration, one of the main tools that we shall use in the sequel is a Gronwall-type argument which is well-adapted to incremental equations of the form \eqref{generic_RP}, but in a more general, $q$-variation context.
We will extensively make use of the following version of this result, whose proof is due to \cite{deya2016priori}.

\begin{lemma}[Rough Gronwall]
	\label{lem:gronwall}
	Let $G\colon[0,T]\to \R_+$ be a path such that there exist constants $\kappa,L>0,$ a control $\omega ,$ and a superadditive map $\varphi $ with:
	\begin{equation}\label{rel:gron}
	\delta G_{st}\leq \left(\sup_{s\leq r\leq t} G_r\right)\omega (s,t)^{\kappa }+\varphi (s,t),
	\end{equation}
	for every $(s,t)\in\Delta$ under the smallness condition $\omega (s,t)\leq L$.
	
	Then, there exists a constant $\tau _{\kappa ,L}>0$ such that
	\begin{equation}
	\label{concl:gron}
	\sup_{0\leq t\leq T}G_t\leq \exp\left(\frac{\omega (0,T)}{\tau _{\kappa ,L}}\right)\left[G_0+\sup_{0\leq t\leq T}\left|\varphi (0,t)\right|\right].
	\end{equation}
\end{lemma}

\subsection{Integration in $\mathcal D^{\alpha ,p}_B$}
\label{subsec:controlled}
In this paragraph we consider a smooth domain $U\subset\R^d$
and we fix $p,p'\in[1,\infty]$ so that $1/p+1/p'=1.$
For notational simplicity, we will omit the domain of integrability and denote by $L^p=L^p(U),$ $W^{k,p}=W^{k,p}(U),$ and so on. In the remainder of the section, we will assume that 
\begin{align}
\label{ass:B_controlled}
\B\equiv(B^1,B^2)\enskip 
&\text{is a}\enskip \V^\alpha \text{-unbounded rough driver acting on the scale}\enskip (W^{k,p})_{k=-3}^3\,,
\end{align}
under the assumption that
$\alpha >1/3$.

For $k\geq 0,$ and $y\in \V_2^{k\alpha }(0,T;W^{-k,p}),$ we shall use the notations
\[
\nn{y}{k}(s,t):= \|y\|_{\V_2^{k\alpha} (s,t;W^{-k,p})},\quad \text{for}\enskip (s,t)\in\Delta ,
\]
and
\[
\nn{y}{k}:=
\nn{y}{k}(0,T)\,.
\]
These are motivated by the (tautological) fact that for $y$ as above the quantity $\omega (s,t):=\nn{y}{k}(s,t)^{\frac{1}{k\alpha }}$ defines a control which is larger than $|\delta y_{st}|_{W^{-k,p}}$ for each $(s,t)\in\Delta $ (it is in fact the smallest one).

We now introduce what in the context of unbounded rough drivers plays the role of the usual controlled path space. Note that the definition below only makes use of the first level $B=B^1_{0\cdot }$ of $\B,$ which is why we write $\mathcal D_B^{\alpha ,p}$ instead of $\mathcal D_{\B}^{\alpha ,p}$.
\begin{definition}[Controlled path space]
	\label{def:controlled}
	Given $\alpha \in(1/3,1/2],$ we define the {\normalfont\textbf{controlled path space}} $\mathcal D_B^{\alpha,p}\equiv\mathcal D_B^{\alpha,p}([0,T]\times U)$
	as the linear space of couples 
	\[(g,g')\in \Big(L^\infty(0,T;L^p)\cap \V^{\alpha }(0,T;W^{-1,p})\Big)^2\]
	such that $g$ is controlled by $B$ with Gubinelli derivative $g' $ (in the sense of Definition \ref{def:controlled_by}).
	
	Furthermore, equipped with the norm
	\begin{equation}
	\label{DB_norm}
	\| (g,g')\|_{\mathcal D^{\alpha,p}_B([0,T]\times U)}:=\|(g,g')\|_{L^\infty(0,T;L^p(U))} + \nn{R^g}{2} + \nnn{\delta g'},
	\end{equation}
	the space $\mathcal D_B^{\alpha,p}([0,T]\times U) $ forms a Banach space.
\end{definition}

Consider $(g,g')\in\mathcal D_B^{\alpha,p}$ and let $f\in L^1(0,T;W^{-3,p}).$
Applying Proposition \ref{pro:sewing} with the choices
\[E:=W^{-3,p}\,,\quad 
H_{st}:=\int_s^tf_r\d r +B^1_{st}g_s + B^2_{st}g'_s\,,
\]
it is easily seen that there exists a unique couple $(u,u^\natural)\in C(0,T;W^{-3,p})\times\V^{1+}_2(0,T;W^{-3,p})$
such that for any $(s,t)\in\Delta :$
\begin{equation}
\label{eq:u_natural}
u_t-u_s= \int_s^tf_r\d r + B^1_{st}g_s + B^2_{st}g'_s + u^\natural_{st}\,.
\end{equation}
Indeed, we have using Chen's relations
\[
-\delta H_{s\theta t}= B^1_{\theta t}R^{g}_{s\theta } + B^2_{\theta t}\delta g'_{s\theta }\,,\quad (s,\theta ,t)\in\Delta _2\,,
\]
and therefore
\[
|\delta H_{s\theta t}|_{W^{-3,p}}\lesssim \omega _B(s,t)^{\alpha }\nn{R^{g}}{2}(s,t) + \omega _B(s,t)^{2\alpha }\n{\delta g'}_{-1}^{[\alpha ]}(s,t)\,,
\]
which is finite by definition of the controlled path space.
Hence, the sewing lemma (Proposition \ref{pro:sewing}) applies, which shows existence and uniqueness of $(u,u^\natural)$ satisfying \eqref{eq:u_natural}, as an equality in $W^{-3,p}.$

In the sequel, the following suggestive notation will be adopted
\begin{equation}
\label{weak_rel}
\d u = f\d t+ \d\B (g,g')\,.
\end{equation}
or simply 
\begin{equation}
\label{weak_rel_2}
\d u= f\d t + \d\B g
\end{equation} 
if $g=g'$.
We point out that \eqref{weak_rel_2} does not necessarily mean that $u$ is a weak solution, because Definition \ref{def:weak_sol} involves some assumptions on the regularity of $u.$ The remainder of this section will address these regularity issues.

\subsection{Remainder estimates}
Conversely, starting from the relation \eqref{weak_rel_2} for some $g\in L^\infty(0,T;L^p),$
one would like to know under which conditions on $f$ and $g$ does the solution $u$ belong to the controlled paths space $\mathcal D^{\alpha ,p}_B$.
A first observation in this direction is the following.

\begin{proposition}
	\label{pro:R_g}
	Consider $f\in L^1(0,T;W^{-2,p})$ let $(g,g)\in \mathcal D^{\alpha ,p}_B,$ and assume that $v$ satisfies
	\[\d v = f\d t + \d \B g\,,
	\]
	(see \eqref{eq:u_natural}).
	Then, $v$ is controlled by $B$ with Gubinelli derivative $v'=g$.
	Moreover, the following estimate holds on $R^v_{st}\equiv \delta v_{st}-B^1_{st}g_s$:
	\begin{multline*}
	\nn{R^v}{2}(s,t)\leq C\Big[ \int_s^t|f_r|_{W^{-2,p}}\d r + \omega _B(s,t)^{2\alpha }\|v,g\|_{L^\infty(s,t;L^p)}\Big]
	\\
	+\frac12\left(\nn{R^{g}}{2}(s,t)
	+\omega _B(s,t)^\alpha \nnn{\delta g}(s,t) \right)\,.
	\end{multline*}
	
	In particular, if one assumes $v=g,$ this yields the bound
	\begin{multline}
	\label{estim:v_sharp}
	\nn{R^v}{2}(s,t)\leq 2C\Big[ \int_s^t|f_r|_{W^{-2,p}}\d r + \omega _B(s,t)^{2\alpha }\|v\|_{L^\infty(s,t;L^p)}
	+\omega _B(s,t)^\alpha \nnn{\delta v}(s,t) \Big].
	\end{multline}
\end{proposition}

Before we proceed to the proof of Proposition \ref{pro:R_g}, let us observe the following.
There exists a family $(J_\eta )_{\eta \in(0,1)}$ of bounded linear maps
$J_\eta\in L\big(W^{k,p},W^{k,p}\big),\eta \in(0,1),$
$k\in\mathbb Z$ being arbitrary, 
such that:
\begin{align}
\label{J1}
&\text{\textbullet }\quad \quad 
J_\eta \quad \text{maps}\quad  W^{k,p}\quad  \text{into}\quad C^\infty ,\quad 
\text{for every }\enskip \eta \in(0,1)\,.
\intertext{For some constant $C_J>0,$ for any $\ell\in\N$ with $|k-\ell|\leq 2:$ 
	if $0\leq k\leq \ell\leq 3 ,$ then}
\label{J2}
&\text{\textbullet }\quad \quad 
|J_\eta |_{\mathscr L(W^{k,p}, W^{\ell ,p} )}\leq \frac{C_J}{\eta ^{\ell-k} }\,,\quad \text{for all}\quad \eta \in(0,1)\,.
\intertext{Finally, if $0\leq \ell \leq k\leq 3,$ then}
\label{J3}
&\text{\textbullet }\quad \quad 
|\id-J_\eta |_{\mathscr L(W^{k,p}, W^{\ell ,p})}\leq C_J\eta^{k-\ell } \,,
\quad \text{for all}\quad \eta \in(0,1)\,.
\end{align}
In the case when $U\equiv \R^d$ and $p\in[1,\infty]$ it suffices to consider $J_\eta f:= \eta ^{-d}\rho (\frac{\cdot }{\eta })*f,$ where $\rho $ is a radially symmetric, smooth function integrating to one. For the general case, we refer for instance to \cite[Appendix A.3]{hocquet2017energy})

From now on, we shall refer to $(J_\eta )_{\eta \in(0,1)}$ as a \emph{family of smoothing operators}.

\bigskip

With this observation at hand, we can now proceed to the proof of the above result.

\begin{proof}[Proof of Proposition \ref{pro:R_g}]
	Note that
	\[
	R^v_{st} := \delta v_{st} - B^1_{st}g_{s}\equiv \int_s^tf_r\d r + B^2_{st}g_{s} + v_{st}^\natural\,.
	\]
	Using \eqref{J2}-\eqref{J3}, we can interpolate these two different expressions for $R^g,$ by writing
	\begin{equation}
	\label{Rv_estim}
	\begin{aligned}
	|R^v_{st}|_{W^{-2,p}}
	&\leq |J_\eta (\int_s^t f\d r  + B_{st}^2g_s + v_{st}^\natural)|_{W^{-2,p}} + |(\id-J_\eta )[\delta v_{st}-B^1_{st}g_s]|_{W^{-2,p}}
	\\
	&\lesssim |\int_s^tf_r\d r|_{W^{-2,p}} + |B^2_{st}g|_{W^{-2,p}} + \frac{|v^\natural_{st}|_{W^{-3,p}}}{\eta}
	\\
	&\quad \quad \quad \quad \quad 
	+\eta ^22\|v\|_{L^\infty(L^p)} +\eta \omega _B(s,t)^\alpha  \|g\|_{L^\infty(L^p)}.
	\end{aligned}
	\end{equation} 
	
	In order to estimate $v^\natural,$ note that Chen's relations \eqref{chen} imply
	\[
	-\delta (B^1g + B^2g)_{s\theta t}= B^1_{\theta t}(\delta g_{s\theta }-B^1_{s\theta }g_s) + B^2_{\theta t}\delta g_{s\theta }
	=B^1_{\theta t}R^g_{s\theta } + B^2_{\theta t}\delta g_{s\theta }\,,\quad \text{for}\enskip (s,\theta ,t)\in\Delta _2\,.
	\]
	From this and the Sewing Lemma, we infer that
	\begin{equation}
	\label{pre:estim:v_nat}
	|v^\natural_{st}|_{W^{-3,p}}\leq C(\alpha )\left(\omega _B(s,t)^{\alpha }\nn{R^v}{2}(s,t) + \omega _B(s,t)^{2\alpha }\nnn{\delta v}(s,t)\right) \,.
	\end{equation}
	
	Now, since \eqref{Rv_estim} is true for arbitrary $\eta \in(0,1),$ we can choose $\eta :=\zeta \omega _B(s,t)^\alpha $ for some $\zeta >0$ big enough. We obtain from \eqref{pre:estim:v_nat}:
	\[
	\begin{aligned}
	|R^v_{st}|_{W^{-2,p}}
	&\leq \left(\int_s^t|f_r|_{W^{-2,p}}\d r\right) 
	+ \omega _B(s,t)^{2\alpha }\|v,g\|_{L^\infty(s,t;L^p)}
	+\frac{\nn{v^\natural}{3}(s,t)}{\zeta \omega _B(s,t)^\alpha }
	\\
	&\leq \left(\int_s^t|f_r|_{W^{-2,p}}\d r\right) 
	+ \omega _B(s,t)^{2\alpha }\|v,g\|_{L^\infty(s,t;L^p)}
	\\
	&\quad \quad \quad \quad \quad 
	+ \frac12\left( \nn{R^{g}}{2}(s,t) +\omega _B(s,t)^\alpha  \nnn{\delta g}(s,t) \right),
	\end{aligned}
	\]
	provided that $\omega _B(s,t)\leq L(\alpha ).$

	This shows the claimed property.
\end{proof}

Consider an equation of the form
$\d v= f\d t + \d\B v\,,$
with $f\in L^1(0,T;W^{-2,p}),$ and define the remainder $u^\natural\in \V^{1+}_2(0,T;W^{-3,p})$ as in \eqref{eq:u_natural}, namely
\begin{equation}
\label{def:remainder}
v^\natural_{st}:=\delta v_{st}-\int_s^tf_r\d r - (B^1_{st}+ B^2_{st})v_s,\quad (s,t)\in\Delta \,.
\end{equation}
As was observed in \cite{deya2016priori}, it is possible in this case to obtain a priori estimates on $v^\natural$ in $\V^{3\alpha }(W^{-3,p})$,
explicitly in terms of $\|f\|_{L^1(0,T;W^{-2,p})}$ and $\|v\|_{L^\infty(L^p)}$ only.
This is the content of the following result, which will be an essential tool in the sequel.

\begin{proposition}[Remainder estimates]\label{pro:apriori}
	Fix $\alpha \in(1/3,1/2],$ $p\in[1,\infty]$ and let $v\in L^\infty(0,T;L^p)$ such that 
	\begin{equation}
	\label{eq:v_gene}
	\d v=f\d t+ \d\B v,
	\end{equation} 
	for some $f\in L^p(0,T;W^{-2,p}).$
	
	Then, the remainder $v^\natural$ defined by \eqref{def:remainder} has locally finite $\frac{1}{3\alpha }$-variation.
	Moreover, there are constants $C,L>0$ depending only on $\alpha ,$
	such that for each $(s,t)\in\Delta $ satisfying
	\[
	\omega _B(s,t)\leq L,
	\]
	it holds
	\begin{equation}
	\label{estimate_remainder}
	\nn{v^{\natural}}{3}(s,t)\leq C\left(\omega _B(s,t)^{3\alpha }\| v\|_{L^\infty(s,t;L^p)}+
	\omega _B(s,t)^{\alpha }\int_s^t|f_{r}|_{W^{-2,p}}\d r\right).
	\end{equation}
	
	As a consequence, any $v$ satisfying the Euler-Taylor expansion \eqref{def:remainder} is controlled by $B$ with Gubinelli derivative $v'=v$,
	that is $\|(v,v)\|_{\mathcal D^{\alpha ,p}_B}<\infty.$
	In addition, it holds the a priori estimates 
	\begin{align}
	\label{precise_est_1}
	&\n{\delta v_{st}}_{-1}^{[\alpha ]}(s,t)\leq C \left[\left(\int_s^t|f_r|_{W^{-2,p}}\d r\right)^{\alpha } + \omega _B(s,t)^\alpha \|v\|_{L^\infty(s,t;L^p)}\right]
	\\
	\label{precise_est_2}
	&\nn{R^v_{st}}{2} (s,t)\leq C \left[ \int_s^t|f_r|_{W^{-2,p}}\d r  + \omega _B(s,t)^{2\alpha }\|v\|_{L^\infty(s,t;L^p)} \right]
	\end{align}
	for any $(s,t)\in \Delta $ such that $\omega _B(s,t)+\int_s^t|f_r|_{W^{-2,p}}\d r\leq L,$ where $L(\alpha )>0.$
\end{proposition}

Note that \eqref{estimate_remainder} is implicitly contained in \cite{deya2016priori}. Since our notations and settings are different, we provide a full proof.

\begin{proof}[Proof of Proposition \ref{pro:apriori}.]
	\textit{Proof of \eqref{estimate_remainder}.}
	By definition of a weak solution, there exists some $z\in(1,3\alpha ]$ such that $v^\natural$ has finite $1/z$-variation, namely:
	\[
	\omega_z (s,t):= \n{v^\natural}_{\V^z_2(s,t;W^{-3,p})}^{1/z}<\infty.
	\]
	Furthermore, we recall the following property (see \cite{hocquet2017energy}): for any $(s,t)\in\Delta,$
	\begin{equation}
	\label{alternative_caract}
	\omega _z(s,t)=\inf \left\{\omega (s,t),\,\omega:\Delta _{[s,t]}\to\R_+ \enskip \text{control such that}\enskip  (\omega )^z\geq |v^\natural|_{W^{-3,p}}\right\}.
	\end{equation} 
	
	Applying $\delta $ to both sides of \eqref{def:remainder} and making use of Chen's relations \eqref{chen}, we have for $(s,\theta ,t)\in\Delta _2,$ 
	\[
	\begin{aligned}
	\delta v_{s\theta t}^\natural 
	= B^1_{\theta t}(\delta v_{s\theta }-B^1_{s\theta }v_s)+ B^2_{\theta t}\delta v_{s\theta }
	\equiv B^1_{\theta t}R^{v}_{s\theta }+ B^2_{\theta t}\delta v_{s\theta }\,,
	\end{aligned}
	\]
	by definition of $R^v$ in \eqref{R_g}.
	Taking the $W^{-3,p}$-norm and then using \eqref{estim:v_sharp}, we obtain
	\begin{multline}
	\label{prelim:v_natural}
	|\delta v^\natural_{s\theta t}|_{W^{-3,p}}
	\leq \omega _B(s,t)^\alpha \nn{R^v}{2}(s,t) + \omega _B(s,t)^{2\alpha }\nnn{\delta v}(s,t)
	\\
	\lesssim
	\omega _B(s,t)^\alpha\int_s^t|f|_{W^{-2,p}}\d r + \omega _B(s,t)^{3\alpha }\|v\|_{L^\infty(s,t;L^p)}
	+ \omega _B(s,t)^{2\alpha }\nnn{\delta v}(s,t)
	\end{multline}
	so that the problem boils down to estimating the term $\nnn{\delta v}(s,t).$
	To obtain such an estimate, we proceed as in the proof of Proposition \ref{pro:apriori}, writing
	\begin{equation}
	\begin{aligned}
	\delta v_{st}
	&= (\id -J_\eta)\delta v_{st} + J_\eta\delta v_{st}
	\\
	&=(\id -J_\eta)\delta v_{st} + J_\eta (\int_s^tf_r\d r+B^1_{st}v_s + B^2_{st}v_s + v^\natural_{st})
	\end{aligned}
	\end{equation}
	where $J_\eta ,\eta \in(0,1),$ denotes a family of smoothing operators. Making use of the properties \eqref{J1}--\eqref{J3} we obtain
	\begin{multline*}
	|\delta v_{st}|_{W^{-1,p}}
	\lesssim \eta \|v\|_{L^\infty(s,t;L^p)} + \frac1\eta \int_s^t|f_r|_{W^{-2,p}}\d r
	+\omega _B(s,t)^\alpha \|v_s\|_{L^\infty;L^p} 
	\\
	+ \frac{\omega _B(s,t)^{2\alpha }}{\eta }\|v_s\|_{L^\infty(s,t;L^p)} + \frac{\omega _z(s,t)^{z}}{\eta ^2}
	\end{multline*}
	by definition of the control $\omega _z.$
	Going back to \eqref{prelim:v_natural}
	and making the choice
	\begin{equation}
	\label{choice}
	\eta :=\zeta \omega _B(s,t)^{\alpha },
	\end{equation} 
	for some parameter $\zeta>0$ (to be fixed later), we obtain the inequality
	\begin{multline}
	|\delta v^\natural_{s \theta t}|_{W^{-3,p}}
	\leq 
	C_J\Big(
	\omega _B(s,t)^\alpha \int_s^t|f_r|_{W^{-2,p}}\d r(1+\zeta ^{-1})
	\\
	+
	\omega _B(s,t)^{3\alpha }\|v\|_{L^\infty(L^p)}(1+\zeta +\zeta ^{-1})
	+\omega _z(s,t)^z \zeta ^{-2}\Big).
	\end{multline}
	Observe further that in \eqref{choice}, $\eta $ must belong to the interval $(0,1)$ by definition of a family of smoothing operators, which will always be true if $(s,t)\in \Delta $ is chosen so that
	$\omega _B(s,t)< L:=\zeta ^{-1/\alpha }.$
	If we fix $\zeta >0$ sufficiently large so that
	\begin{equation}
	\label{lambda}
	\frac{C_{\mathrm{sewing}}(z)C_J}{\zeta ^2}\leq \frac12
	\end{equation} 
	$C_{\mathrm{sewing}}(z)$ being the constant of the Sewing Lemma, this leads to the smallness assumption:
	\begin{equation}
	\omega _B(s,t)\leq L:=(C_{\mathrm{sewing}}(z)C_J)^{-1/(2\alpha) }\,.
	\end{equation} 
	Now, applying Proposition \ref{pro:sewing} and using \eqref{lambda}, we see that for any $(s,t)\in\Delta $ with $\omega _B(s,t)\leq L,$ it holds
	\[
	|v^{\natural}_{st}|_{W^{-3,p}}
	\leq C_z\Big(\omega _B(s,t)^{3\alpha }\|v\|_{L^\infty(L^p)}
	+\omega _B(s,t)^\alpha \int_s^t|f_r|_{W^{-2,p}}\d r
	\Big) +\frac12 \omega _z(s,t)^z,
	\]
	for some universal constant $C_z>0.$
	By the inequality $(a+b)^\epsilon \leq a^\epsilon +b^\epsilon $ for $a,b\geq 0$ and $\epsilon \in[0,1],$
	we have
	\begin{multline*}
	|v^{\natural}_{st}|_{W^{-3,p}}^{1/z}
	\leq (C_z)^{1/z}\Big[\omega _B(s,t)^{3\alpha/z }\|v\|_{L^\infty(L^p)}^{1/z}
	\\
	+\omega _B(s,t)^{\alpha/z} \Big(\int_s^t|f_r|_{W^{-2,p}}\d r\Big)^{1/z}\Big]
	+\frac{1}{2^{1/z}} \omega _z(s,t)
	\end{multline*}
	By \cite[p.22]{friz2010multidimensional},
	the above right hand side is a control, hence we infer from the property \eqref{alternative_caract} that
	\begin{multline*}
	\omega _z(s,t)\leq (C_z)^{1/z}\Big[\omega _B(s,t)^{3\alpha/z }\|v\|_{L^\infty(L^p)}^{1/z}
	\\
	+\omega _B(s,t)^{\alpha/z} \Big(\int_s^t|f_r|_{W^{-2,p}}\d r\Big)^{1/z}\Big]
	+\frac{1}{2^{1/z}} \omega _z(s,t),
	\end{multline*}
	which shows that for any $z\in(1,3\alpha ]$
	\begin{multline}
	\label{each_z}
	|v^\natural_{st}|_{W^{-3,p}}^{1/z}
	\leq \omega _z(s,t)\leq (C_z)^{1/z}\left(1-\frac{1}{2^{1/z}}\right)^{-1}\Big[\omega _B(s,t)^{3\alpha/z }\|v\|_{L^\infty(L^p)}^{1/z}
	\\
	+\omega _B(s,t)^{\alpha/z} \Big(\int_s^t|f_r|_{W^{-2,p}}\d r\Big)^{1/z}
	\Big].
	\end{multline}
	Letting now $z=3\alpha $ yields the inequality \eqref{estimate_remainder}.
	
	\bigskip
	
	\item[\textit{Proof of \eqref{precise_est_1}}.]
	Writing as before that
	$\delta v =(\id-J_\eta)\delta v+ J_\eta (\int f\d r + B^1v + B^2v + v^\natural),$
	and then using \eqref{J1}--\eqref{J3}, we have
	\[
	\nnn{\delta v}(s,t)
	\lesssim  \Big(\eta + \omega _B(s,t)^{\alpha } + \frac{\omega _B(s,t)^{2\alpha }}{\eta }\Big)\|v\|_{L^\infty(s,t;L^p)} 
	+\frac1\eta \int_s^t|f_r|_{W^{-2,p}}\d r + \frac{\nn{v^\natural}{3}(s,t)}{\eta ^2}\,.
	\]
	Combining with Proposition \ref{pro:apriori}, this gives
	\begin{multline}
	\label{pre_bound_delta}
	\nnn{\delta v}(s,t)
	\lesssim \Big(\eta + \omega _B(s,t)^{\alpha } + \frac{\omega _B(s,t)^{2\alpha }}{\eta } + \frac{\omega _B(s,t)^{3\alpha }}{\eta ^2}\Big)\|v\|_{L^\infty(s,t;L^p)} 
	\\
	+ (\frac1\eta + \frac{\omega _B(s,t)^\alpha }{\eta ^2}) \int_s^t|f_r|_{W^{-2,p}}\d r.
	\end{multline} 
	Upon choosing
	\[
	\eta :=\omega _B(s,t)^\alpha  + (\int_s^t|f_r|_{W^{-2,p}}\d r)^\alpha ,
	\]
	in \eqref{pre_bound_delta},
	we obtain the estimate
	\begin{multline*}
	|\delta v_{st}|_{W^{-1,p}}
	\lesssim
	\Big(\int_s^t|f_r|_{W^{-2,p}}\d r)^{\alpha }+\omega _B(s,t)^\alpha\Big)\|v\|_{L^\infty(s,t;L^p)}
	\\
	+(\int_s^t|f_r|_{W^{-2,p}}\d r)^{1-\alpha }+
	\omega _B(s,t)^\alpha (\int_s^t|f_r|_{W^{-2,p}}\d r)^{1-2\alpha } 
	\end{multline*}
	and the conclusion follows by the observation that $1-\alpha \geq \alpha .$
\end{proof}

\section{The parabolic class $\H^{\alpha ,p}_B$}
\label{sec:space}

This section is devoted to the definition of a natural functional setting for rough partial differential equations of the form \eqref{rough_parabolic}.
In a second part, we will address the problem of obtaining an explicit equation for the product of two elements $u\in L^\infty(L^p)$ and $v\in L^\infty(L^{p'}),$ where $1/p+1/p'=1$ and
such that
\[\d u= f\d t + \d\B u\]
while
\[\d v=g\d t+ \d\B v\]
on  $[0,T]\times\R^d,$ where $\B$ is a geometric, differential rough driver
(here we consider $f$ and $g$ as given distributions).
If $\B$ is ``built over'' a derivation-valued path, by which we mean that $B^1_{st}=B_t-B_s$ for some $B_t=X_t\cdot \nabla $,
one expects that $uv$ solves the problem
\begin{equation}
\label{product_sec_renorm}
\d (uv)=(ug+fv)\d t + \d\B (uv)\,.
\end{equation}
This indeed appears as a consequence of the Leibnitz-type identity $B_t(uv)=(B_tu)v + u(B_tv)$, the geometricity of $\B$ and a formal application of \cite[Proposition 7.6]{friz2014course} (apply first the It\^o formula on the square map, and then use polarization identities).
For a more general geometric $\B \sim (X,\LL)$ (i.e.\ with a non-zero multiplicative term $X^0_{st}$), a similar relation is expected, with the difference that $\B$ has to be ``shifted'' to a new object $\B^{(2)}$ of the same nature, but this time built over $X_t\cdot \nabla +2X_t^0.$ This fact will be made clear in the following paragraphs.

\subsection{A natural Banach space setting}
Let $p\in[1,\infty],$ fix a domain $U\subset\R^d,$ and
consider a $\V^\alpha $-differential rough driver $\B$ with $\alpha >1/3.$
We define a space $\H^{\alpha ,p}_B([0,T]\times U)$ as follows:
\begin{equation}
\label{nota:HB}
\begin{aligned}
&\H^{\alpha ,p}_B([0,T]\times U)
\\
&:=\Bigg\{
u\in L^\infty(L^p),\enskip \text{such that}\enskip (u,u)\in\mathcal D_B^{\alpha ,p},\enskip \text{and there is}\enskip f\in L^p(W^{-1,p}(U)),\enskip 
\\[-0.8em]
&\quad\quad \quad \quad 
\text{satisfying}\enskip \d u=f\d t+\d \B (u,u),\enskip \text{and with the property that}\enskip 
\\
&\quad\quad \quad
\|u\|_{\H_B^{\alpha ,p}([0,T]\times U)}:=\|u\|_{L^\infty(L^p(U))}+\|\nabla u\|_{L^p(L^p(U))}
+\|f\|_{L^p(W^{-1,p}(U))}
\\[-0.8em]
&\quad \quad \quad \quad \quad \quad 
\quad \quad \quad \quad \quad \quad 
+\|\delta u\|_{\V_2^\alpha (0,T;W^{-1,p}(U))}+ \|R^u\|_{\V_2^{2\alpha }(0,T;W^{-2,p}(U))}<\infty
\Bigg\},
\end{aligned}
\end{equation} 
where we recall notation \eqref{weak_rel}.
As before, in the case when $U=\R^d$ we omit to indicate the domain, and we define local versions $\H_{B,\loc}^{\alpha ,p}$ of these spaces by the property
\[
u\in \H_{B,\loc}^{\alpha ,p}\enskip \Leftrightarrow\enskip u|_{[0,T]\times K}\in \H_B^{\alpha ,p}([0,T]\times K)\enskip \text{for every }K\subset\subset \R^d\,.
\]

% For $U\subset\R^d,$ we also introduce a notion of weak convergence for elements of $\H^{\alpha ,p}_{B}(U).$
% If $u(n)\in \H^{\alpha ,p}_{B}$ is a sequence so that for each $n\geq 0$,
% \[\d u(n)\equiv f(n)\d t + \d \B u(n)\,,\]
% for some $f(n)\in L^p(W^{-1,p}),$ we shall say that $u(n)$ \emph{converges weakly} to $u\in \H^{\alpha ,p}_B(U)$ and we write
% \begin{equation}
% \label{weak_conv}
% u(n)\rightharpoonup  u
% \end{equation}
% if and only if
% %%%
% 
% where $f$ is such that $\d u= f\d t + \d \B u,$
% and where we recall that $R^u$ stands for the two-parameter quantity $(s,t)\mapsto\delta u_{st}-B^1_{st}u_s.$
% Note that the last property of \eqref{def:weak_sol} does not mean that the couple $(\delta u(n),R^{u(n)})$ converges with respect to the weak topology of $\V^{\alpha }(0,T;W^{-1,p})\times \V^{\alpha }(0,T;W^{-2,p})$ and so \eqref{weak_conv} does not mean that weak convergence holds in the ``natural'' sense.
% A reason to do so is to avoid dealing with the dual space of $\V^\alpha ,$ which is not easy manipulated.
% However, this issue is purely technical since in practice it will be always possible to take $1/3<\alpha '<\alpha $ in order to show the convergence of the rough part $(\delta u(n),R^{u(n)})$, thanks to a suitable Aubin-Lions type compactness argument. This is illustrated in the next Lemma.

One of the main interests in defining the above spaces is the next compactness-type result, which will be fundamental in the sequel.

\begin{lemma}[$\H^{\alpha ,p}_B$-weak stability]
	\label{lem:stability}
	Fix an open set $U\subset \R^d,$ let $p\in [1,\infty]$ and consider a family $\{\B(n),n\in\mathbb N\}\cup\{\B\}$ of differential rough drivers such that $\rho _{\alpha }(\B(n),\B)\to0$ where $\rho _{\alpha }$ is the distance introduced in \eqref{topology_B}.
	For each $n\geq 0$ consider $v(n)\in\H^{\alpha ,p}_{B(n)}(U)$ 
	and
	$f^i(n)\in L^p(L^p(U)),i=0,\dots d,$ 
	such that
	\[\d v(n)=(\partial _if^i(n)+f^0(n))\d t + \d\B (n)v(n)\,,
	\]
	weakly in $L^p.$ Assume that the corresponding family is uniformly bounded in the sense that for every $n\geq 0$:
	\begin{equation}
	\label{stab:Lp}
	\|v(n)\|_{\H^{\alpha ,p}_{B(n)}(U)}\leq C\,,
	\end{equation}
	for some constant $C>0.$
	
	The following assertions are true.
	\begin{enumerate}[label=(\arabic*)]
		\item If $p>1,$ there exists $n_k\nearrow \infty, k\to\infty,$
		some $f^i\in L^p(L^p),$ $i=0,\dots d,$
		and $v\in\H^{\alpha ,p}_{B}$ so that
		\begin{equation}
		\label{weak}
		\begin{aligned}
		&v(n_k)\to v\quad \text{weakly-$*$}\enskip \text{in}\enskip L^\infty(0,T;L^p(U))\cap L^2(0,T;W^{1,p}(U))\,,
		\\
		&f(n_k)\to f\quad \text{weakly-$*$}\enskip \text{in}\enskip L^p(0,T;W^{-1,p}(U))\,,
		\end{aligned}
		\end{equation} 
		while for any $\alpha '<\alpha$:
		\begin{equation}
		\label{stab:1}
		(\delta v(n_k),R^{v(n_k)})\to(\delta v,R^v)\quad \text{in}\enskip \V^{\alpha '}_2(0,T;W^{-1,p}_\w(U))\times \V^{\alpha '}_2(0,T;W^{-2,p}_\w(U))\,.
		\end{equation}
		
		Moreover, $v$ satisfies
		\begin{equation}
		\label{v_satisfies}
		\d v=(\partial _if^i+f^0)\d t + \d\B v\,\,,
		\end{equation} 
		in the $L^1$-sense.
		
		\item A similar conclusion holds for $p=1$ if the family $\{(v(n),f(n)),$ $n\in\mathbb{N}\}$ is equi-integrable. Recall that $f(n)$ is said to be equi-integrable if it is bounded in $L^1$ and such that for any $\epsilon >0,$
		there exists $\delta _\epsilon >0$ and $\Omega_\epsilon \subset U$ with $|\Omega _\epsilon |<\infty$ so that uniformly in $n\geq 0$:
		\begin{align*}
		\iint_{[s,t]\times A}|f(n)|\d x\d r\leq \epsilon 
		\intertext{for every $A\subset U$ measurable and $(s,t)\in\Delta $ such that $(t-s)|A|\leq \delta _\epsilon $, and}
		\iint_{[0,T]\times (U\setminus\Omega) }|f(n)|\d x\d r\leq \epsilon \,.
		\end{align*}
	\end{enumerate}
\end{lemma}

\begin{proof}
	We first address the case $p>1.$
	In that case, the two first properties of \eqref{weak} are just a consequence of Banach-Alaoglu Theorem, together with the definition of the spaces $\H^{\alpha ,p}_{B}.$
	Concerning the last one, it relies on the following Aubin-Lions-type compactness result. The proof follows exactly the same steps as \cite[Lemma A.2 \& Lemma A.3]{hofmanova2018navier}, and is therefore omitted.\\

	\textbf{Claim.} {\em
		Let $\omega :\Delta \to\R_+$ be a control function, let $p\in(1,\infty]$ and fix $L>0.$ For $\kappa >0,$ introduce the Banach space
		\begin{multline*}
		\mathfrak X^\kappa(\omega )
		:=L^\infty(0,T;L^p)\bigcap L^p(0,T;L^p)\bigcap \Big\{u\in \V^\alpha (0,T;W^{-1,p}),\enskip |\delta u_{st}|\leq \omega (s,t)^\kappa ,
		\\
		\forall (s,t)\in\Delta\enskip \text{with}\enskip \omega (s,t)\leq L \Big\}\,,
		\end{multline*}
		endowed with the norm 
		\[\interleave u\interleave _{\kappa ,\omega }:=\|u\|_{L^\infty(L^p)}+\|u\|_{L^p(W^{1,p})}+ \sup_{(s,t)\in\Delta }\frac{|\delta u_{st}|_{W^{-1,p}}}{\omega (s,t)^\kappa }\,.\]
		
		Then, 
		\begin{multline}
		\label{compact_X}
		X^\kappa (\omega )\enskip \text{is compactly embedded into}
		\\ 
		L^p(0,T;L^p_\loc)\cap L^\infty(0,T;W^{-1,p}_\loc) \cap \V^{\kappa '}(0,T;W^{-2,p}_\loc)
		\quad 
		\text{for any}\enskip 0<\kappa '<\kappa .
		\end{multline}
		\\
	}

	By definition of $\H^{\alpha ,p}_{B}$, the norm of $v(n)$ in the controlled path space forms a uniformly bounded sequence. But thanks to Proposition \ref{pro:apriori}, we also have the precise estimate
	\begin{equation}
	\label{precise_context}
	|\delta v_{st}(n)|_{W^{-1,p}}\leq C\left[\left(\int_s^t|f(n)|_{W^{-2,p}}\d r\right)^\alpha  + \omega _{B(n)}(s,t)^\alpha \right]
	\leq C'(\alpha )\omega _n(s,t)^\alpha 
	\end{equation}
	for any $(s,t)\in\Delta $ such that $\omega _n(s,t):=\omega _{B(n)}(s,t)+\int_s^t|f(n)|_{W^{-2,p}}\leq L,$ where $L=L(\alpha )$ is independent of $n\in\N.$
	Though the estimate \eqref{precise_context} suffers the fact that the control $\omega _n$ depends on $n\in\N,$  we note that proceeding as in \cite[Lemma 2.3]{lejay2003introduction}, it is always possible to build a control $\varpi$ (depending on the whole sequence $\{\omega _n,n\in\N\}$) so that \eqref{precise_context} holds with $\varpi$ for all $n\in \N.$ For such $\varpi,$ by definition of the space $\mathfrak X^\alpha (\varpi)$, we therefore obtain the uniform estimate:
	\[
	\interleave\delta v_{st}(n)\interleave_{\alpha ,\varpi }\leq C \|v(n)\|_{L^\infty(L^p)}\leq \widetilde C\,\,,
	\]
	Hence property in \eqref{stab:1} follows by the compact embedding \eqref{compact_X}, and the obvious inclusion $\mathfrak X^{\alpha }(\varpi )\subset \V^\alpha (0,T;W^{-1,p}).$
	
	% an immediate generalization of Proposition 5.28 in \cite{friz2010multidimensional}, together with the compact embeddings $W^{-k,p}\hookrightarrow W^{-k,p}_\w$ for $k=1,2$. 

	Now, let $f(n)\in L^p(W^{-1,p})$ such that $\d v(n)=f(n)\d t+\d\B v(n)$ for each $n\in\mathbb{N}.$
	Testing the equation against $\phi \in W^{3,p'}(U)$ then yields for every $(s,t)\in\Delta :$
	\begin{equation}
	\label{any_n}
	\langle\delta v_{st}(n),\phi \rangle -\langle [B^1_{st}(n)+B_{st}^2(n)]v_s(n),\phi \rangle - \int_s^t\langle f_r(n),\phi \rangle\d r= \langle v^\natural_{st}(n),\phi \rangle, 
	\end{equation} 
	where $v_{st}^\natural(n)\in \V^{1+}_2(0,T;W^{-3,p}(U))$ denotes the remainder term.

	We now show that $v$ belongs to $\H^{\alpha ,p}_{B}$ and satisfies \eqref{v_satisfies}.
	In \eqref{any_n}, the left hand side converges towards 
	\[\langle\delta v_{st},\phi \rangle -\langle [B^1_{st}+B_{st}^2]v_s,\phi \rangle 
	-\int_s^t\langle f_r,\phi \rangle\d r,\]
	for any $(s,t)\in\Delta ,$ as an obvious consequence of \eqref{weak}.
	Concerning the remainder term, it converges to some element $\langle v^{\natural},\phi \rangle \in \V^{3\alpha '}_2(0,T;\R)$ for any $\alpha '<\alpha ,$ as a consequence of \eqref{stab:1} and the continuity part of the Sewing Lemma.
	Using the convergence of $\B(n)$ and Proposition \ref{pro:apriori}, we see that $v^{\natural}$ defined above is actually an element of $\V_2^{3\alpha }(0,T;W^{-3,p}).$ By 
	\eqref{precise_est_1} and \eqref{precise_est_2}, one also obtains that $(\delta v,R^v)$ belongs to $\V_2^{\alpha }(W^{-1,p})\times\V^{2\alpha }_2(W^{-2,p}),$
	showing that $v$ is indeed an element of $\H^{\alpha ,p}_B.$ This proves the first part.

	Now, concerning the case $p=1,$ as is well-known the Dunford-Pettis Theorem (see e.g.\ \cite{albiac2016topics}) implies that a bounded family of $L^1$
	is relatively weakly compact if and only if it is equi-integrable. Hence, the second assertion follows by the same argument as before, using a slight modification of the above compactness claim. We omit the details.
\end{proof}

\subsection{Main result: product formula}

Let $u\in\H^{\alpha ,p}_{B},$ and $v\in \H^{\alpha ,p'}_B$ with $1/p+1/p'=1.$ If $\B$ is geometric, it seems natural to expect that the pointwise product $uv$ belongs to $\H^{\alpha,1}_{\tilde B}$ for some (possibly new) differential rough driver $\mathbf {\tilde B}.$ The main result of this section gives a justification of this intuition, by showing a product formula for $uv$ (it could be alternatively thought of an ``integration by parts'' formula).
By reiteration of the argument, a similar product formula will be shown on mononomials of bounded paths $u\in \H^{\alpha ,2}_B$, see Corollary \ref{cor:product}.

In what follows, we consider a fixed open set
$
U\subset \R^d\,.
$

\begin{proposition}[Product formula, general case]
	\label{pro:product}
	Let $\B$ be a geometric, $\V^\alpha $-differential rough driver with $\alpha \in(1/3,1/2]$,
	fix $p,p'\in [1,\infty]$ with $1/p+1/p'=1,$ and consider two elements $u,v\in \H_{B}^{\alpha ,1}(U)$ such that 
	\[u\in L^\infty(0,T;L^p(U))\cap L^p(0,T;W^{1,p}(U))
	\]
	while 
	\[v\in L^\infty(0,T;L^{p'}(U))\cap L^{p'}(0,T;W^{1,p'}(U)).
	\]
	Let $f^i,g^i\in L^1(0,T;L^1(U)),0\leq i \leq d,$ such that
	on $[0,T]\times U$,
	\[
	\left[\begin{aligned}
	&\d u= (\partial _i  f^i +f^0)\d t + \d\B  u,
	\quad \text{strongly in}\enskip L^p(U)\,,
	\\
	&\d v=(\partial _i  g^i + g^0) \d t +\d\B  v,
	\quad \text{strongly in}\enskip L^{p'}(U)\,,
	\end{aligned}\right.
	\]
	in the sense of Definition \ref{def:weak_sol}.
	Assume furthermore that for $i =0,\dots ,d,$ the pointwise products $\partial _i  u(\cdot )g^i(\cdot -a) $ and $ f^i(\cdot -a) \partial _i  v(\cdot )$ are in $L^1(0,T;L^1(U)),$ for any $a\in \R^d$ with $|a|\leq 1.$

	Then, the following holds:
	\begin{enumerate}[label=(\roman*)]
		\item \label{Q_shift}
		The two-parameter mapping $\B^{(2)}\equiv(B^{(2),1},B^{(2),2})$, whose components are defined for $(s,t)\in\Delta $ as the differential operators
		\begin{equation}
		\label{shifted_rough_driver}
		\left\{
		\begin{aligned}
		B^{(2),1}_{st}
		&:= B^1_{st}+X^0_{st} ,
		\\
		B^{(2),2}_{st}
		&:=B^2_{st}+ X^0_{st}X^i_{st}\partial _i+\LL^0_{st}+\frac32(X^0_{st})^2,
		\end{aligned}\right.
		\end{equation}
		is itself a geometric differential rough driver.
		
		\item \label{prod_uv}
		The pointwise product $uv$ belongs to $\H^{\alpha ,1}_{B^{(2)}}(U)$ and is an $L^1(U)$-energy solution of
		\begin{equation}
		\label{concl:prod}
		\d (uv)= \big[u(\partial _i  g^i + g^0) +(\partial _i f^i +f^0)v\big]\d t + \d \B^{(2)}(uv)\,.
		\end{equation}
	\end{enumerate}
\end{proposition}

Regarding the definition of the spaces $\H^{\alpha ,p}_{B,\loc},$ we have the following immediate consequence of Proposition \ref{pro:product}.

\begin{corollary}[Product formula, transport case]
	\label{cor:product}
	Let $\B \sim (X,\LL)$ be as in Proposition \ref{pro:product} with $X^0=0.$
	Fix $p,p'\in[1,\infty]$ so that $1/p+1/p'=1,$ and let $u\in \H_{B,\loc}^{\alpha ,p}$ be such that
	\begin{equation}
	\label{eq:cor:u}
	\d u = f\d t+ \d\B  u\quad ,\quad \text{on}\enskip [0,T]\times \R^d,
	\end{equation}
	in the $L^p_\loc$, strong sense, for some $f\in L^p(W^{-1,p}).$
	
	The following holds.
	\begin{enumerate}[label=(\Roman*)]
		\item 
		Let $v\in \H_{B,\loc}^{\alpha ,p'}$ be an $L^{p'}_{\loc}$-energy solution of
		\[\d v= g\d t + \d\B  v\quad  \text{on}\enskip  [0,T]\times \R^d,
		\]
		with $g\in L^{p'}(W^{-1,p'}).$
		Then, the product
		$uv$ belongs to $\H^{\alpha ,1}_{B,\loc}$ and moreover $uv$ is an $L^1_{\loc}$-energy solution of
		\begin{equation}
		\d (uv)= (ug+fv)\d t + \d\B (uv)\,.
		\end{equation} 
		
		\item In the case where $p=2$ and  $u$ belongs to $L^\infty_{\loc},$
		then for each $n\in\N$ we have $u^n\in \H^{\alpha ,1}_{B,\loc},$ and moreover:
		\begin{equation}
		\d (u^n)= nu^{n-1}f \d t + \d\B (u^n)\,,\quad \text{on}\enskip [0,T]\times\R^d
		\end{equation}
		($L^1_\loc$ sense).
	\end{enumerate}
\end{corollary}

\begin{remark}
	A similar conclusion as that of Corollary \ref{cor:product} holds when $\B \sim (X,\LL)$ with $X^0\neq 0.$
	In this case, it is easily seen by induction that for every $n\in\mathbb{N}:$
	\[
	\d (u^n)= nu^{n-1}f \d t + \d \B^{(n)}[u^n]
	\]
	in $L^1_\loc,$ where using the notation of Lemma \ref{lem:multiplication},
	$\B^{(n)}$ is the geometric differential rough driver defined as
	\[
	\left\{
	\begin{aligned}
	B^{(n),1}_{st}
	&:= X^i_{st}\partial _i + nX^0_{st},\quad
	\\
	B^{(n),2}
	&:= \frac12X_{st}^iX_{st}^j\partial _{ij} + \left(\LL_{st}^i + nX^0_{st}X^i_{st}\right)\partial _i + n\LL^0_{st}+ \frac{n^2}{2}(X_{st}^0)^2\,\,,
	\end{aligned}\right.
	\]
	or making use of notation \ref{not:summarize}:
	\[
	\B^{(n)}\sim \Big((nX^0,X^1,\dots,X^d);(n\LL^0,\LL^1,\dots,\LL^d)\Big)\,.
	\]
\end{remark}

Before we proceed to the proof of Proposition \ref{pro:product}, we need to introduce some additional notation.
In what follows, we fix a bounded, open set
$D\subset U,$ such that
\[
\gamma :=1\wedge\mathrm{dist}(D,\partial U)>0\,.
\]

\begin{nota}
	For $\epsilon \in(0,1]$
	we will denote by $D_\epsilon $ the $\epsilon \gamma $-fattening of $D$, namely
	\[
	D_\epsilon :=\{x+\epsilon h\in \R^d,\,x\in D\,\text{and}\enskip h\in B_\gamma  \}.
	\]
	For such $D,$ we further define a set $\Omega_\epsilon^{D} \subset \R^d\times\R^d$ as follows:
	\begin{equation}
	\label{nota:Omega}
	\Omega^{D}_\epsilon  
	:=\left\{(x,y)\in U\times U,\,\frac{x+y}{2}\in D\,,\frac{x-y}{2}\in B_\epsilon \right\}.
	\end{equation}
\end{nota}

\begin{nota}
	\label{nota:T_eps}
	For $k\in I\subset \mathbb Z$ 
	we define a linear, one-to-one transform $T_\epsilon$, by the formula
	\begin{equation}
	\label{blow_up}
	T_\epsilon \Phi (x,y):=\frac{1}{(2\epsilon )^d}\Phi \left(\frac{x+y}{2}+\frac{x-y}{2\epsilon },\frac{x+y}{2}-\frac{x-y}{2\epsilon }\right),
	\end{equation}
	for all $\Phi \in W^{k,\infty}_0(\R^d\times\R^d).$
	In particular, identifying $\Phi \in W^{k,\infty}_0(\Omega ^{D}_1)$ with its extension by $0$ outside its support, we have an isomorphism $T_\epsilon :W^{k,p}_0(\Omega ^{D}_1 )\to W^{k,p}_0(\Omega ^{D }_\epsilon )$.
\end{nota}
According to the terminology introduced in \cite{deya2016priori}, any geometric differential rough driver is ``renormalizable''.
This is the statement of the following Theorem, whose proof is rather technical and, for that reason, postponed in Appendix \ref{app:renorm}.
\begin{theorem}
	\label{thm:renorm}
	Let $\B$ be a geometric, differential rough driver with regularity $\alpha >1/3.$
	
	Introduce the differential rough driver $\Gamma(\B)\equiv(\Gamma ^1(\B),\Gamma ^2(\B))$ given for every $(s,t)\in\Delta $ by
	\begin{equation}
	\label{nota:S}
	\left\{
	\begin{aligned}
	&\Gamma_{st}^1(\B):=B^1_{st}\otimes \id + \id \otimes B^1_{st},
	\\
	&\Gamma_{st}^2(\B):= B^2_{st}\otimes \id + B^1_{st}\otimes B^1_{st}+ \id\otimes B^2_{st}
	\end{aligned}\right.
	\end{equation}
	(the fact that this is indeed a differential rough driver is elementary and hence left to the reader).
	
	Then, for each $i=1,2$ and $k=-3+i,\dots, 0,$ the following uniform bound holds
	\begin{equation}
	\label{bounds:renormalization}
	\big|T_\epsilon ^{-1,*}\Gamma _{st}^i(\B)T_\epsilon^{*} \big|_{\mathscr L(W^{1,k}(\Omega _1^D),W^{1,k-i}(\Omega _1^D))}\leq
	C\omega _B(s,t)^{i\alpha }
	\end{equation}
	where $C>0$ denotes a constant which is independent of $\epsilon \in(0,1],$ while $\omega _B$ is the control introduced in Definition \ref{def:rough_driver}.
\end{theorem}

Before we proceed to the proof of the main result, let us observe that if $a\in W^{-k,p'}$ and $b\in W^{k,p},$ then the product $ab$ has a well-defined meaning as an element of $ab\in W^{-1,1}$ (it suffices to write $a$ in terms of its antiderivatives, and to integrate by parts).
Moreover, if $a,b$ are measurable functions (i.e.\ not distributions), then the adjoint of $T_\epsilon $ is given by the formula
\begin{equation}
\label{T_epsilon_star}
T_\epsilon^*[a(x)b(y)]
=2^{-d}a\Big(\frac{x+y}{2}+\epsilon \frac{x-y}{2}\Big)b\Big(\frac{x+y}{2}-\epsilon \frac{x-y}{2}\Big).
\end{equation}
Testing against $\Phi \in W_0^{k,\infty}(\Omega _1^{D}),$ and doing the change of variables $(x_+,x_-):=\chi (x,y)\equiv(\frac{x+y}{2},\frac{x-y}{2}),$ this gives the formula
\begin{equation}
\label{representation}
\langle T_\epsilon ^*v,\Phi \rangle
=\int_{B_1}\BRAKET{W^{-k,1}(D)}{a\big(\cdot +\epsilon x_-\big)b\big(\cdot -\epsilon x_-\big),\Phi \circ \chi ^{-1}(\cdot ,x_-) }{W^{k,\infty}_0(D)}\d x_-\,.
\end{equation}
Now, in the general case where $a\in W^{-k,p'}$ is a distribution, it is easily seen that \eqref{representation} is still meaningful. This formula will be useful in the sequel.

We can now turn to the proof of the main result.

\begin{proof}[Proof of Proposition \ref{pro:product}.]
	\textit{Step $0$: doubling of variables.}
	In the sequel, we let for simplicity
	\[
	f:=\partial _i  f^i + f^0,\quad \enskip g:=\partial _i  g^i +g^0 \,,
	\]
	and denote by $u\otimes v$ the function of two variables
	\[
	(u\otimes v)_t(x,y):=u_t(x)v_t(y),\quad  \text{for every} \enskip (x,y)\enskip \text{in}\enskip \Omega ^{D}_1.
	\]
	
	For any $\epsilon \in(0,1)$ and $(s,t)\in\Delta ,$ we further introduce
	\begin{empheq}[left=\empheqlbrack]{align}
	&(u\otimes v)^\epsilon_s 
	:=T_\epsilon ^*\left((u_s\otimes v_s)\big|_{\Omega _\epsilon ^{D}}\right),
	\\
	&(f\otimes v +u\otimes g)_s^\epsilon 
	:=T_\epsilon ^*\left((f_s\otimes v_s +u_s\otimes g_s)\big|_{\Omega _\epsilon ^{D}}\right),
	\\
	&\Gamma ^\epsilon (\B)=(\Gamma ^{1,\epsilon }(\B),\Gamma ^{2\epsilon }(\B))\enskip \text{where}\enskip 
	\begin{cases}
	\Gamma^{1,\epsilon }_{st}(\B) :=T_\epsilon ^*\Gamma_{st}^1(\B) (T_\epsilon ^*)^{-1},\\
	\Gamma_{st}^{2,\epsilon }(\B):=T_\epsilon ^*\Gamma_{st}^2(\B) (T_\epsilon ^*)^{-1}\,.
	\end{cases}
	\end{empheq}

	Then, the following assertions are true.
	\begin{enumerate}[label=(\arabic*)]
		\item 
		$(u\otimes v)^\epsilon $ belongs to $\H^{\alpha,1}_{\Gamma (\B)}(\Omega ^{D}_1  ).$
		\item the mapping
		$t\mapsto (f_t\otimes v_t + u_t\otimes g_t)^\epsilon, $
		is Bochner integrable in the space $W^{-1,1}(\Omega ^{D}_1),$
		
		\item  $(u\otimes v)^\epsilon $ is an $L^1(\Omega ^{D}_1)$-energy solution of the equation
		\begin{equation}
		\label{eq:tensor_eps}
		\d (u\otimes v)^\epsilon  = \big(f\otimes v +u\otimes g\big)^\epsilon \d t 
		+\d \Gamma^\epsilon (\B) (u\otimes v)^\epsilon .
		\end{equation} 
	\end{enumerate}
	
	The proof of the above properties is rather technical, but follows exactly the same pattern as that of \cite[Section 5]{hocquet2017energy}, hence we leave the details to the reader.
	
	\bigskip
	
	\item[\indent\textit{Step 1: uniform bound on the drift.}]
	If $\Phi \in W^{1,\infty}_0(\Omega ^{D}_1)$ and $(s,t)\in\Delta ,$ we have by definition
	\begin{equation}
	\label{drift_product}
	\BRAKET{}{\int_s^t(u_r\otimes g_r +f_r \otimes v_r)^\epsilon \d r ,\Phi}{}
	=\int_s^t\BRAKET{}{u_r\otimes g_r +f_r \otimes v_r ,T_\epsilon \Phi }{}\d r.
	\end{equation} 
	
	Fix $r\in[s,t]$ such that $u\equiv u_r$ belongs to $W^{1,p},$  and let $\check\Phi (x_+,x_-):=\Phi \circ\chi ^{-1}(x_+,x_-)=\Phi (x_++x_-,x_+-x_-).$ 
	Making use of \eqref{representation}, we have for the first term in \eqref{drift_product}:
	\[
	\begin{aligned}
	\langle u\otimes g  ,T_\epsilon \Phi \rangle
	&=\int_{B_1}\BRAKET{W^{-1,p'}(D)}{g(\cdot -\epsilon x_-) ,u(\cdot +\epsilon x_-)\check\Phi (\cdot ,x_-)}{W^{1,p}(D)}\d x_-
	\\
	&=\iint_{B_1\times D}\Big\{g^i(x_++\epsilon x_-)(-1)^i\frac{\partial }{\partial x^i_+}\big[u(x_++\epsilon x_-)\check\Phi (x_+,x_-)\big]
	\\
	&\quad \quad \quad \quad \quad \quad \quad 
	+ g^0(x_++\epsilon x_-)u(x_++\epsilon x_-)\check\Phi (x_+,x_-)\Big\}\d x_+\d x_-\,.
	\end{aligned}
	\]
	Hence, we have
	\begin{align}
	\nonumber
	\langle u\otimes g  ,T_\epsilon \Phi \rangle
	&\leq \iint_{B_1\times D}\Big\{|g^i (x_+-\epsilon x_-)||\partial _i  u (x_++\epsilon x_-))|
	\\
	&\quad \quad \quad \quad \quad 
	+|g^0 (x_+-\epsilon x_-)|| u (x_++\epsilon x_-))|
	\Big\}(|\check\Phi|+|\nabla _+\check\Phi |)\d x_+\d x_-
	\\
	\nonumber
	&\leq |\Phi |_{W^{1,\infty}}\int_{B_1}\d x_-\int_{D+\epsilon x_- }\Big\{|g^i(x_+ - 2\epsilon x_-)||\partial _iu(x_+)|
	\\
	\nonumber
	&\quad \quad \quad \quad \quad \quad
	\quad \quad \quad 
	+|g^0 (x_+-2\epsilon x_-)|| u (x_+))|
	\Big\}\d x_+
	\\
	\label{gi_diu}
	&\leq |\Phi |_{W^{1,\infty}}\int_{B_1}\left(|g^i_{(\epsilon x_-)}\partial _iu |_{L^1(D_{\epsilon })}+ |g^0_{(\epsilon x_-)}u|_{L^1(D_\epsilon )}\right)\d x_-\,,
	\end{align}
	where for simplicity for $i=0,\dots ,d,$ we denote by 
	\[
	g^i_{(\epsilon x_-)}(x_+):=
	\begin{cases}
	g^i(x_+ - 2\epsilon x_-)\quad \text{if}\enskip x_+-2\epsilon x_-\in D_{\epsilon }
	\\
	0\quad \text{otherwise}\,.
	\end{cases}
	\]
	(Note that, by assumption, the right hand side in \eqref{gi_diu} is finite.)
	Doing similar computations for the second term, and then integrating in time,
	we end up with the estimate
	\begin{multline}
	\label{uniform_bound_drift}
	\Big|\int_s^t(u_r\otimes g_r +f_r \otimes v_r)^\epsilon \d r\Big|_{W^{-1,1}(\Omega ^{D}_1)}
	\\
	\leq 
	\int_{B_1}\Big(\|\partial _i  ug_{(\epsilon x_-)}^i,ug^0_{(\epsilon x_-)}\|_{L^1(s,t;L^1(D_{\epsilon }))}
	+\|f^i_{(-\epsilon x_-)} \partial _i  v, f^0_{(-\epsilon x_-)}v\|_{L^1(s,t;L^1(D_{\epsilon } ))}
	\Big)\d x_-
	\\
	=:\omega _{\mathscr D,D_{\epsilon }}(s,t)\,,
	\end{multline}
	where we further observe that $\omega _{\mathscr D,D_{\epsilon }}$ is a control since positive linear combinations of controls are controls.
	
	\bigskip
	
	\item[\indent\textit{Step 2: convergence of the remainder term.}]
	For a.e.\ $r\in [s,t],$ it is straightforward to check the inequality
	\[
	|(u\otimes v)^\epsilon _r|_{L^1(\Omega ^{D}_1)}\leq|D_{\epsilon }||u_r|_{L^p(D_{\epsilon })}|v_r|_{L^{p'}(D_{\epsilon })}.
	\]
	Therefore, by Theorem \ref{thm:renorm} together with Proposition \ref{pro:apriori} we obtain the following bound on the remainder $(u\otimes v)^{\epsilon ,\natural}$ associated to \eqref{eq:tensor_eps}:
	\begin{multline}
	\label{uniform_bound_remainder}
	|(u\otimes v)^{\epsilon ,\natural}_{st}|_{W^{-3,1}(\Omega ^{D}_1)}
	\leq C\Big(|D_{\epsilon }|\|u_r\|_{L^\infty(L^p(D_{\epsilon }))}\|v_r\|_{L^\infty(L^{p'}(D_{\epsilon }))}\omega _B(s,t)^{3\alpha }
	\\
	+\omega _{\mathscr D,D_{\epsilon }}(s,t)\omega _B(s,t)^\alpha \Big),
	\end{multline}
	for every $(s,t)\in\Delta$ such that $\omega _B(s,t)\leq L$ for some $L(\alpha )>0,$ and every $\epsilon \in(0,1).$
	
	Fix
	\begin{equation}
	\label{choice:psi}
	\psi \in W^{3,\infty}_0(B_1),\quad \text{with}\quad \int_{B_1}\psi(x_-) \d x_-=1\,,
	\end{equation}
	and for $(s,t)$ as above, denote by $\ell _{st}^\epsilon  $ the element of $W^{-3,1}(D)$ defined as
	\[
	\langle \ell^\epsilon  _{st}, \phi \rangle:= \langle (u\otimes v)_{st}^{\epsilon ,\natural},(\phi \otimes \psi )\circ\chi \rangle,
	\quad \text{for}\enskip \phi \in W^{3,\infty}_0(D).
	\]
	By definition of $\ell ^\epsilon $ and the estimate \eqref{uniform_bound_remainder}, 
	we deduce that $\ell ^\epsilon $ is uniformly bounded in $\V^{3\alpha }_{2,\loc}(0,T;W^{-3,1}_\w(D)).$ Proceeding as in the proof of Lemma \ref{lem:stability}, we infer the existence of $\ell \in \V^{3\alpha }_{2,\loc}(0,T;(W^{3,\infty}_0(D))^*)$ and $\epsilon _n\searrow0$ such that for any $\alpha '<\alpha $ and every $\phi \in W^{3,\infty}_0(D)$
	\begin{equation}
	\label{bd:Frechet}
	\langle \ell ^{\epsilon _n},\phi \rangle \to \langle \ell,\phi \rangle \quad \text{in}\enskip \V^{3\alpha '}_{2,\loc}(0,T;\R)
	\end{equation}
	which in particular implies convergence in the $C(\Delta ;\R)$-sense.
	
	It remains to show that $\ell _{st}$ belongs to $W^{-3,1}(D)$ for any $(s,t)\in\Delta .$
	In \eqref{uniform_bound_remainder}, substitute $D$ with any $K\subset D$ and then take the limit as $\epsilon \to0$. This yields
	\begin{multline}
	\label{ineq:l_psi}
	|\ell _{st}|_{(W^{3,\infty}_0(K))^*}\leq C
	\Big[|K|\|u\|_{L^\infty(L^p(K))}\|v\|_{L^\infty(L^{p'}(K))}\omega _B(s,t)^{3\alpha }
	\\
	+\big(\|\partial _iug^i,ug^0 \|_{L^1(s,t;L^1(K))}
	+\|f^i \partial _iv,f^0v\|_{L^1(s,t;L^1(K))}
	\big)\omega _B(s,t)^\alpha \Big].
	\end{multline}
	This implies that $|\ell _{st}|_{(W^{3,\infty}_0(K))^*}$ goes to $0$, as $|K|\to 0.$ 
	As is well-known (see e.g.\ \cite[Proposition 4.4.2 p.~263 \& Proposition 1.3.3 p.~9]{bogachev2007measure})
	this implies that $\ell $ is an element of the subspace $W^{-3,1}(D).$ This proves the claimed property.
	
	\bigskip
	
	\item[\indent\textit{Step 3: passage to the limit in the equation}]
	Fix any $\phi \in W^{3,\infty}(U)$ with compact support in $D,$
	and test \eqref{eq:tensor_eps} against 
	\[
	\Phi(x,y):=\phi (\frac{x+y}{2})\psi  (\frac{x-y}{2}),\enskip (x,y)\in\Omega ^{D} _1 ,
	\]
	which is indeed an element of $W^{3,\infty}(\Omega ^{D}_1 ).$ Observe furthermore that $T_\epsilon \Phi (x,y)=\phi (\frac{x+y}{2})\psi_\epsilon  (x-y)$
	where 
	\[\psi _\epsilon(\cdot )=\psi_\epsilon  (\cdot /2)(2\epsilon)^{-d}\]
	approximates the identity.
	
	Hence, using Theorem \ref{thm:renorm} and dominated convergence, we find that
	\[
	\BRAKET{W^{-1,1}(\Omega ^{D}_1)}{\int_s^t(u\otimes g+f\otimes v)^\epsilon ,\Phi}{W^{1,\infty}_0(\Omega ^{D}_1)}
	\underset{\epsilon \to0}{\longrightarrow} \int_s^t\BRAKET{W^{-1,1}(D)}{u_rg_r+ f_rv_r ,\phi }{W_0^{1,\infty}(D)}\d r\,.
	\]
	
	For the terms involving $\Gamma (\B),$ we first note that by Lemma \ref{lem:multiplication}, the following Leibniz-type formulas are satisfied:
	for every $a,b\in C^\infty$ it holds
	\begin{equation}
	\label{gene:leib:1}
	\left\{
	\begin{aligned}
	&B^1_{st}(ab) = (B^1_{st}a )b + a(B^1_{st}b) - X^0_{st}ab\,,
	\\
	&B^2_{st}(ab) = (B^2_{st}a)b + (B^1_{st}a)(B^1_{st}b) + a(B^2_{st}b) -X^i_{st}X^0_{st}\partial _i(ab) -\big(\LL^0_{st} + \frac32(X^0_{st})^2\big)ab\,.
	\end{aligned}\right.
	\end{equation}
	Now, using dominated convergence and \eqref{gene:leib:1} yields for the first term
	\begin{multline*}
	\BRAKET{W^{-1,1}(\Omega ^{D}_1)}{\Gamma ^{1,\epsilon }_{st}(\B) (u\otimes v)_s^\epsilon ,\Phi }{W^{1,\infty}_0(\Omega ^{D}_1)}
	\\
	\underset{\epsilon \to0}{\longrightarrow }\enskip \BRAKET{W^{-1,1}(D)}{(B^1_{st}u_s)v_s + v_sB^1_{st}u_s,\phi}{W^{1,\infty}_0(D)}
	\\
	=\langle (B^1_{st}+ X^0_{st})(uv),\phi \rangle
	= \langle B_{st}^{(2),1}(uv),\phi \rangle,
	\end{multline*}
	by definition of $\B^{(2),1}.$
	Similarly, using the second equation in \eqref{gene:leib:1}, it is easily seen that
	\begin{multline}
	\label{limit_Q2}
	\BRAKET{W^{-2,1}(\Omega ^{D}_1)}{ \Gamma^{2,\epsilon } _{st}(\B)(u\otimes v)^\epsilon _s ,\Phi }{W^{2,\infty}_0(\Omega ^{D}_1)}
	\\
	\underset{\epsilon \to0}{\longrightarrow }\enskip\BRAKET{}{(B^2_{st}u_s)v_s + (B^1_{st}u_s)(B^1_{st}v_s)+u_s (B^2_{st} v_s),\phi }{}
	\\
	=\left\langle (B^2_{st} + X^0_{st}X^i_{st}\partial _i+\LL^0_{st} + \frac32(X^0_{st})^2)(uv),\phi \right\rangle
	=\left\langle B_{st}^{(2),2}(uv),\phi \right\rangle\,.
	\end{multline}
	Finally, we have $\langle \delta (u\otimes v)^\epsilon _{st},\Phi \rangle\to_{\epsilon \to0} \langle\delta (uv)_{st},\phi \rangle,$
	and hence using the previous step:
	\begin{equation}
	\label{relation_finale}
	\langle\delta (uv)_{st},\phi \rangle=\int_s^t\big\langle ug+fv,\phi\big\rangle\d r 
	+ \left\langle (B^{(2),1}_{st}+B^{(2),2}_{st})(uv),\phi\right\rangle + \big\langle\ell_{st},\phi \big\rangle,
	\end{equation}
	for every $(s,t)\in $ such that $\omega _B(s,t)\leq L.$
	The equation \eqref{relation_finale} holds for any open and bounded $D\subset U$ with positive distance from $U.$ Thus, it remains true for $U$ itself, which shows that $uv$ is an $L^1(U)$-weak solution of \eqref{concl:prod}.

	It remains to show that $\B^{(2)}$ is a differential rough driver, for which it suffices to check that Chen's relations \eqref{chen} hold. But these are an immediate consequence of Lemma \ref{lem:multiplication} and the linearity of $\delta ,$ since:
	\begin{multline*}
	\delta B^{(2),2}_{s\theta t}\equiv \delta \left(B^2 + X^0X^i\partial _i + \LL^0 + \frac32(X^0)^2 \right)_{s\theta t}
	\\
	=B^1_{\theta t}\circ B^1_{s\theta } + \left(X^0_{\theta t} X^i_{s\theta } + X^0_{s\theta }X^i_{\theta t}\right)\partial _i + X^i_{\theta t}\partial _iX^0_{s\theta } + 3X^0_{\theta t}X^0_{s\theta }
	\\
	=(B^1_{\theta t}+ X^0_{\theta t})\circ(B^1_{s\theta }+ X^0_{s\theta })=B^{(2),1}_{\theta t}\circ B^{(2),1}_{s\theta },
	\end{multline*}
	for $(s,\theta ,t)\in\Delta _2.$
	This shows that $\B^{(2)}$ is a differential rough driver. Moreover, $\B^{(2)}$ is obviously geometric since $\B$ is.
	
	Finally, thanks to Proposition \ref{pro:apriori}, we further see that $uv$ is controlled by $B^{(2)},$ and thus it belongs to $\H^{\alpha ,1}_{B^{(2)},\loc}.$
	This achieves the proof of \ref{prod_uv} and the proposition.
\end{proof}

\section{Parabolic equations with free terms: proof of Theorem \ref{thm:free_intro}}
\label{sec:free}

In this section we investigate existence, uniqueness and stability for parabolic rough partial differential equations of the form
\begin{equation}
\label{free}
\begin{aligned}
\d u = (Au + f)\d t +\d\B  u ,\quad \text{on}\enskip [0,T]\times \R^d
\\
u_0\in L^2(\R^d),
\end{aligned}
\end{equation}
where $f$ belongs to the space $L^2(0,T;H^{-1}).$
This completes the case treated in \cite{hocquet2017energy}, where a more general elliptic operator $A$ was considered, but where the assumptions on $\B$ were more restrictive.
For the reader's convenience, we now restate Theorem \ref{thm:free_intro}.

\begin{theorem}
	\label{thm:free}
	Let $f\in L^2(0,T;H^{-1}),$ fix $u_0\in L^2$ and consider a geometric, differential rough driver $\B$ with regularity $\alpha>1/3.$
	There exists a unique $L^2$-energy solution $u=u(u_0,f;\B)$ to \eqref{free}, and it belongs to the space $\H^{\alpha,2}_B(\R^d).$

	Moreover, the solution map is continuous in the following sense
	\begin{enumerate}[label=(C\arabic*)]
		\item\label{cont_1}
		for every $(u_0,f)\in L^2\times L^2(H^{-1})$, the map $\B\mapsto u(u_0,f;\B)$ is continuous in the following sense:
		for any sequence $\{\B(n),n\in\N\}$ of geometric differential rough drivers such that $\rho _\alpha (\B(n),\B)\to 0,$
		denoting by $u(n)$ the solution of \eqref{free} obtained with $\B$ being replaced by $\B(n),$ it holds
		\begin{align*}
		u(n)\to u\quad \text{weakly-$*$}\enskip \text{in}\enskip L^\infty(0,T;L^p(U))\cap L^2(0,T;W^{1,p}(U))\,,
		\intertext{and for any $\alpha '<\alpha$:}
		(\delta u(n),R^{u(n)})\to(\delta u,R^u)\quad \text{in}\enskip \V^{\alpha '}_2(0,T;W^{-1,p}_\w(U))\times \V^{\alpha '}_2(0,T;W^{-2,p}_\w(U))\,.
		\end{align*}
		
		\item\label{cont_2}
		for $\B$ fixed the map $u(\cdot ,\cdot ;\B):L^2\times L^2(H^{-1}) \to \H^{\alpha ,2}_B$ is continuous, with respect to the strong topologies.
	\end{enumerate}
\end{theorem}

Note that the above result obviously implies Theorem \ref{thm:free_intro}. 
Its proof essentially follows the lines of \cite{hocquet2017energy} but since our assumptions on $\B$ are more general, we provide a complete proof.

\begin{proof}[Proof of Theorem \ref{thm:free}]
	Consider an $L^2$-energy solution $u\in\H^{\alpha,2}_B$ of the equation \eqref{free}.
	Applying Proposition \ref{pro:product} with $u=v,$
	we have that $u^2\in\H^{\alpha,1}_{B^{(2)}}$
	where $\B^{(2)}$ is the shifted differential rough driver defined in \eqref{shifted_rough_driver}.
	Moreover, $u^2$ solves in the $L^1$-sense:
	\begin{equation}
	\label{ito_square_free}
	\d u^2 =2u(Au+f)\d t + \d \B^{(2)}(u^2).
	\end{equation} 
	We want to test against $\phi =1,$ and then apply Rough Gronwall, but for this we need first an estimate on $u^{2,\natural},$
	which itself follows from Proposition \ref{pro:apriori}, together with the estimate on the drift.
	The analysis of the linear part of the drift leads to the estimate:
	\begin{equation}
	\label{bd:reduced_drift}
	\Big| \int_s^t (uAu)\d r\Big|_{W^{-1,1}}
	\leq \lambda^{-1}(\|\nabla u\|^2_{L^2(s,t;L^2)}+\|u\nabla u\|_{L^1(s,t;L^1)})
	\end{equation}
	whereas for the free term, considering anti-derivatives, we find
	\begin{equation}
	\label{bd:free_drift}
	\int_s^t |uf|_{W^{-1,1}}\d r\leq \left(\|u\|_{L^2(s,t;L^2)}+\|\nabla u\|_{L^2(s,t;L^2)}\right)\|f\|_{L^2(s,t;H^{-1})}.
	\end{equation}
	The proof is then divided into 3 steps.

	\bigskip
	
	\item[\indent\emph{Step 1: Energy inequality and application to uniqueness.}]
	Letting $\omega_{\mathscr D} (s,t)$ be the sum of the right hand sides in \eqref{bd:reduced_drift} and \eqref{bd:free_drift},
	one can then apply Proposition \ref{pro:apriori} to obtain
	\begin{equation}
	|u^{2,\natural}_{st}|_{W^{-3,1}}
	\leq C\left(\omega _B(s,t)^\alpha \omega _{\mathscr D}(s,t)  + \|u\|^2_{L^\infty(s,t;L^2)}\omega _B(s,t)^{3\alpha }\right).
	\end{equation}
	for every $(s,t)\in\Delta $ with $\omega _B(s,t)\leq L$ for some absolute constant $L>0.$
	
	Next, consider $f=\partial _i\f^i + \f^0$ where $\f^i\in L^2, i=0,\dots ,d.$
	One can take $\phi =1\in W^{3,\infty}$ in \eqref{ito_square_free}, 
	so that by Assumption \ref{ass:A} it holds for every $s,t$ as above:
	\begin{multline*}
	\delta E_{st}:=
	\delta (|u|_{L^2}^2)_{st} + \int_s^t|\nabla u_r|_{L^2}^2\d r
	\\ 
	\lesssim_{\lambda}
	\iint_{[s,t]\times \R^d}-\partial _iu_r(x)\f^i_r(x)\d x\d r 
	+ \Big\langle (B^{(2),1}_{st}+B^{(2),2}_{st})u^2_s +u^{2,\natural}_{st},1\Big\rangle
	\\
	\lesssim_{\lambda}
	\|\nabla u\|_{L^2(s,t;L^2)}\|\f\|_{L^2(s,t;L^2)}
	+|u_s|^2_{L^2}(\omega _B(s,t)^\alpha+\omega _B(s,t)^{2\alpha })
	+|u^{2,\natural}_{st}|_{W^{-3,1}},
	\\
	\lesssim_{\lambda}
	\|\nabla u\|_{L^2(s,t;L^2)}\|\f\|_{L^2(s,t;L^2)}
	+
	(\omega _B(s,t)^\alpha +\omega _B(s,t)^\alpha +\omega _B(s,t)^{3\alpha })
	\sup_{r\in[s,t]}E_r
	\\
	+\omega _B(s,t)^\alpha \|f\|_{L^2(s,t;H^{-1})}(\|\nabla u\|_{L^2(s,t;L^2)}+\|u\|_{L^2(s,t;L^2)})
	\end{multline*}
	Making use of Young Inequality \[\|\nabla u\|_{L^2(s,t;L^2)}\|\f\|_{L^2(s,t;L^2)}
	\leq \frac{\epsilon }{2}\|\nabla u\|^2_{L^2(s,t;L^2)}+\frac{1}{2\epsilon }\|\f\|_{L^2(s,t;L^2)}^2 \]
	for $\epsilon (\lambda)>0$ sufficiently small, the first term in the right hand side can be absorbed to the left.
	Hence, taking $L$ smaller if necessary, we infer that for any $(s,t)\in\Delta $ with $\omega _B(s,t)\leq L,$
	it holds the incremental inequality
	\[
	\delta E_{st} \leq \omega _B(s,t)^\alpha (\sup\nolimits_{r\in[s,t]}E_r) +\|f\|_{L^2(s,t;H^{-1})}^2.
	\]
	By Lemma \ref{lem:gronwall}, we deduce the estimate
	\begin{multline}
	\label{estimate:f}
	\|u\|^2_{L^\infty(0,T;L^2)}+\|\nabla u\|^2_{L^2(0,T;L^2)}
	\leq C(\lambda)\exp \left\{\frac{\omega _B(0,T)}{\tau _{\alpha ,L}}\right\}\left[|u_0|_{L^2}^2 + \|f\|_{L^2(0,T;H^{-1})}^2\right].
	\end{multline}
	The uniqueness is now straightforward, because the difference $v\equiv u_1-u_2$ of two $L^2$-energy solutions to \eqref{free}
	ought to be itself an $L^2$-energy solution of \eqref{free}, with $f=0$ and $v_0=0$, hence yielding from \eqref{estimate:f} that $v=0.$

	\bigskip
	
	\item[\indent\emph{Step 2: Existence}.]
	Existence and continuity rely mostly on the stability result shown in Lemma \ref{lem:stability}, together with the fact that $\B$ is geometric.
	
	Consider a sequence $\B(n) \to \B$ as in Definition \ref{def:geometric}.
	By standard results on parabolic equations, there exists a unique $u(n) $ in the energy space $L^\infty(L^2)\cap L^2(H^1),$ solving \eqref{free} in the sense of distributions.
	Using moreover the fact that $\B(n)=S_2(B(n) ),$ it is easily deduced from \eqref{free} that $u(n) $ is an $L^2$-energy solution of \eqref{free}, in the sense of Definition \ref{def:var_sol}. Consequently, the previous analysis shows that we have a uniform bound
	\[
	\|u(n) \|_{L^\infty(0,T;L^2)}^2 + \|\nabla u(n) \|^2_{L^2(0,T;H^1)}
	\leq C\left(\lambda,\|f\|_{L^2(0,T;H^{-1})},|u_0|_{L^2},T\right).
	\]
	As a consequence of this bound and Proposition \ref{pro:apriori}, we also obtain the uniform estimate
	\[
	\|u(n) \|_{\H^{\alpha,2}_{B(n)}}\leq C',
	\]
	for another such constant $C'.$
	By Lemma \ref{lem:stability} we see that $\{u(n),n \in\mathbb N\} $ has a (possibly non-unique) limit point $u\in\H^{\alpha,2}_B$ such that the weak-type convergences of \eqref{weak}-\eqref{stab:1} hold, up to some subsequence $u(n_k)$ $n_k\nearrow \infty.$
	In particular, each of the terms in the equation on $u(n_k) $ converges to the expected quantities associated to the limit $u$. This shows the claimed existence.
	
	\bigskip
	
	\item[\indent\emph{Step 3: Stability }.]
	We can now repeat the argument of Step 3 with \emph{any} sequence $\B(n)$ of geometric, differential rough drivers (not necessarily defined as canonical lifts). This will imply the convergence of a subsequence $u(n_k)\rightharpoonup u$, in the sense of \eqref{weak} and \eqref{stab:1}. From the uniqueness part, there can be at most one such limit $u,$ and therefore every subsequence of $u(n) $ converges to $u.$ This implies the convergence of the full sequence, and the claimed continuity \ref{cont_1}.

	To show \ref{cont_2}, note that if $u$ and $v$ are $L^2$-energy solutions of 
	\[
	\begin{aligned}
	&\d u=(Au +f)\d t +\d\B  u,\quad u_0=u^0,\quad 
	\\
	&\d v= (Av+g)\d t +\d\B  v,\quad v_0=v^0 ,
	\end{aligned}
	\]
	where $u^0,v^0 \in L^2,$ and $f,g\in L^2(H^{-1}),$ then $w:=u-v$ solves the problem
	\[
	\d w=(Aw + f-g) \d t +\d\B w,\quad w_0=u^0 -v^0\, .
	\]
	Therefore, the strong continuity of the solution map with respect to $(u_0,f)$
	follows from the estimate \eqref{estimate:f}, together with Proposition \ref{pro:apriori}.
\end{proof}

\section{Local boundedness of solutions}
\label{sec:boundedness}

In this section, we take a step further by investigating the boundedness, away from $t=0$ and on any compact set of the space variable, for solutions
of parabolic RPDEs of previous form, namely
\begin{equation}
\label{eq:bounded}
\begin{aligned}
\d u= (A u+f)\d t + \d\B  u,\quad \text{in}\enskip [0,T]\times \R^d,
\\
u_0\in L^2(\R^d),
\end{aligned}
\end{equation}
where the free term $f$ will be subject to additional conditions, see Assumption \ref{ass:free_bounded},
and $A$ fulfills Assumption \ref{ass:A}. 

First, let us recall a classical interpolation inequality, the proof of which can be found in \cite{ladyzhenskaya1968linear}.
\begin{proposition}
	\label{pro:interp}
	For each $f$ in the space $L^\infty(0,T;L^2)\cap L^2(0,T;W^{1,2}),$ 
	$f$ belongs to $L^\rho (0,T;L^\sigma  )$ for every $\rho, \sigma $ such that
	\begin{equation}\label{conditions}
	\frac1\rho +\frac{d}{2\sigma }\geq \frac d4\quad\text{and}\quad \left\{\begin{aligned}
	&\rho \in[2,\infty]\,, \quad \sigma \in[2,\tfrac{2d}{d-2}]\quad\text{for}\enskip d>2
	\\
	&\rho \in(2,\infty]\,,\quad \sigma \in[2,\infty)\quad \text{for}\enskip d=2
	\\
	&\rho \in[4,\infty]\,,\quad \sigma \in[2,\infty]\quad \text{for}\enskip d=1\,.
	\end{aligned}\right.
	\end{equation}
	In addition, there exists a constant $C_{\rho ,\sigma }>0 $ (not depending on $f$ in the above space) such that
	\begin{equation}
	\label{interpolation_inequality}
	\|f\|_{L^\rho (0,T;L^\sigma )}\leq C_{\rho ,\sigma } \left(\|\nabla f\|_{L^2(0,T;L^2)}+\esssup_{r\in [0,T]}|f_r|_{L^2}\right)\,.
	\end{equation}
\end{proposition}
As an immediate consequence of \eqref{interpolation_inequality}, it can be checked that whenever $r,q\in[1,\infty]$ are numbers satisfying
\begin{equation}
\label{values:r_q}
\frac{1}{r}+\frac{d}{2q}\leq1,
\end{equation}
then it holds the inequality
\begin{equation}
\label{interp_u}
\|u\|_{L^{\frac{2r}{r-1}}(L^{\frac{2q}{q-1}})}\leq C_{r,q} \|u\|_{L^\infty(L^2)\cap L^2(H^1)}.
\end{equation}

\subsection{Moser Iteration}

Recall the basic idea of Moser's iteration.
If $u\in L^\infty(0,T;L^2)\cap L^2(0,T;H^1)$ solves a parabolic equation of the form \eqref{eq:bounded} where the coefficients are smooth enough, the new unknown $|u|^{\varkappa}$ for $\varkappa\geq 2$ is, roughly speaking, solution of a similar equation. 
By a slight modification of the arguments of the Section \ref{sec:free}, it is possible thanks to the above interpolation inequality to find suitable moment bounds for $v:=|u|^{\varkappa /2},$ the value of which depend on similar moments, but for a \textit{lower} exponent.
Thanks to \eqref{interpolation_inequality}, we will then obtain a recursive relation between these quantities, which will take the form of the following inequality
\begin{equation}
\label{hyp:recursive}
\Phi _{n+1} \leq \gamma \tau ^{n}\Phi _n^{1+\epsilon },\quad \text{for any}\enskip n\geq 0\,\,,
\end{equation}
where $\epsilon ,\gamma ,\tau>0$ are constants.
It is worth noting that the above inequality is non-linear, and that the coefficent $\tau ^n$ will blow up unless $\tau$ is smaller than one.
Hence, an upper bound of $\Phi _n$ may blow-up as well when $n\to \infty.$
However, the next result shows that this explosion is ``not too strong'' for our purposes.
The proof is immediate by induction, and therefore omitted.
\begin{lemma}[Recursive estimate]
	\label{lem:recursion}
	Assume that we are given a sequence of non-negative numbers $\Phi _n,n\geq 0,$ and constants $\epsilon ,\gamma ,\tau>0$
	such that \eqref{hyp:recursive} holds.
	Then, the following estimate is true:
	for any $n\geq 0$ we have
	\begin{equation}
	\label{estim:recursive}
	\Phi _n\leq
	\gamma^{\frac{(1+\epsilon )^n-1}{\epsilon }} \tau ^{\frac{(1+\epsilon )^n-1}{\epsilon ^2}-\frac{n}{\epsilon }}\Phi _0^{(1+\epsilon )^n}.
	\end{equation} 
\end{lemma}
Now, a classical result states that
\[
|f|_{L^\varkappa(X,\mathcal M,\mu )}\underset{\varkappa\to\infty}{\to}|f|_{L^\infty(X,\mathcal M,\mu )},
\]
for any $\sigma $-finite measure space $(X,\mathcal M,\mu )$ and every $f\in L^\infty$ such that $f\in L^q$ for some $q\in[1,\infty).$
Using that result and the fact that $\Phi _n$ will be taken below to be an appropriate sequence of moments with diverging exponents,
we will be able to obtain an a priori estimate for the $L^\infty$-norm of $u$. This will prove the boundedness of solutions.

We need now to specify our conditions on $f.$

\begin{assumption}
	\label{ass:free_bounded}
	We assume that
	\[
	f \in \mathscr M:=L^r(0,T;W^{-1,q})\cap L^{2r}(0,T;W^{-1,2q}) \cap L^1(0,T;W^{-1,1})\cap L^2(0,T;H^{-1}),
	\]
	where the exponents $r\in(1,\infty]$ and $q\in(1\vee\frac{d}{2},\infty)$ are subject to the conditions 
	\begin{equation}
	\label{cond:r_q_strict}
	\frac{1}{r} + \frac{d}{2q}<1.
	\end{equation}
\end{assumption}

Using Sobolev embeddings, it is easily checked that Assumption \ref{ass:free_bounded} is fulfilled for $f$ satisfying the assumptions of Theorem \ref{thm:boundedness}, i.e.\ $f\in L^r (0,T;L^q ),$ where $r$ and $q $ verify the condition \eqref{cond:r_q_strict}. Hence, the following result implies Theorem \ref{thm:boundedness}.

\begin{proposition}
	\label{pro:boundedness}
	Let Assumption \ref{ass:free_bounded} hold,
	suppose that $u_0\in L^2$,
	and assume that 
	$u$ is the solution of \eqref{eq:bounded} given by Theorem \ref{thm:free}.
	Then, the essential supremum of $u$ is bounded on each compact subset of $(0,T]\times\R^d.$
	
	In addition, for any $Q\subset\subset (0,T]\times \R^d,$ it holds the estimate
	\[
	\|u\|_{L^\infty(Q)}\leq C(Q,|u_0|_{L^2},\lambda,\|f\|_{\mathscr M},\omega _B,\alpha ,r,q),
	\]
	for a constant depending only on the indicated quantities.
\end{proposition}

\subsection{The recursive estimate}
Our purpose in the present paragraph is to show that a suitable sequence $\{\Phi _n,n\in\N\}$ can be defined, so that Lemma \ref{lem:recursion} will be applicable and provide the claimed $L^\infty_\loc$ estimate.
Consider $u\in\H^{\alpha,2}_B\cap L^\infty,$ $L^2$-energy solution of \eqref{rough_parabolic}, and let $\varkappa \geq 2.$
Assuming for the moment that the conclusions of Theorem \ref{thm:L_p} are true, we have in the $L^1$-sense:
\begin{equation}
\label{ito:real_power}
\delta |u|^\varkappa_{st} = \int_s^t \varkappa u_r|u_r|^{\varkappa-2} (A_r u_r +f_r) dr + (B^1_{st} + B^2_{st})|u_s|^\varkappa + u_{st}^{\varkappa, \natural}\,.
\end{equation}

Defining
\[
v_t(x):= |u_t(x)|^{\varkappa/2},
\]
we have the identities:
\begin{equation}
\label{identities}
v\partial _iv=\frac{\varkappa}{2}(\partial _iu )|u|^{\varkappa-1}
\quad 
\partial _iv\partial _jv=\frac{\varkappa^2}{4}(\partial _iu)(\partial _ju)|u|^{\varkappa-2}.
\end{equation}
Hence denoting by $(\f^i)$ any antiderivative of $f,$
and by $v^{2,\natural}:=|u|^{\varkappa,\natural},$
it holds for every $\phi \in W^{3,\infty}$:
\begin{multline}
\label{eq:v_squared}
\langle \delta (v^2)_{st},\phi \rangle
-\Big\langle (B^1_{st}+B^2_{st})(v^2_s)+ v^{2,\natural}_{st},\phi \Big\rangle
\\
=\iint_{[s,t]\times \R^d}\Big[-4\big(\frac{\varkappa-1}{\varkappa}\big)a^{ij}(\partial _iv)(\partial _jv)\phi -2a^{ij}v(\partial _iv)(\partial _j\phi)  \Big]\d x\d r
\\
+\iint_{[s,t]\times \R^d} \Big[-2(\varkappa-1)\f^i(\partial _iv)v^{1-\frac{2}{\varkappa}}\phi  -\varkappa \f^iv^{2-\frac{2}{\varkappa}}\partial _i\phi \Big]\d x\d r
=:\langle\delta \mathscr D^{(\varkappa)} _{st},\phi \rangle.
\end{multline}

Next, define two cylinders $Q,Q'$ as follows:
let $R,\tau >0,$ and introduce
\[\begin{aligned}
Q':=\{(t,x):\quad 2\tau \leq t\leq T\enskip \text{and}\enskip |x|\leq R/2\}
\\
Q:=\{(t,x):\quad \tau \leq t\leq T\enskip \text{and}\enskip |x|\leq R\}\,.
\end{aligned}
\]
Since $\tau '>0$ and $R>0$ are arbitrary, is is obviously sufficient to show the local $L^\infty$ estimate in $Q'$ instead of any compact set of $(0,T]\times\R^d.$

To this end, let for each $n\geq 0$ 
\[\begin{aligned}
&R_n:=\frac{R}{2}(1+2^{-n}) \underset{n\to \infty}{\searrow} \frac{R}{2}
\\
&\tau _n:= \tau (2-2^{-n})\underset{n\to \infty}{\nearrow}2\tau 
\end{aligned}
\]
and define the cylinders $Q_n$ accordingly.
With this definition, observe that $\tau _0=\tau ,R_0=R$ and that
for each $n\geq 0$
\[
Q'= \cap _{k=0}^\infty Q_k \subset
Q_{n+1}\subset Q_n
\subset Q=Q_0\,.
\]

Now, choose any sequence of smooth test functions $\psi_\cdot (n;\cdot )$ such that
\[\begin{aligned}
\psi _t(n;x)=
\begin{cases}
1\enskip \text{for}\enskip (t,x)\in Q_{n+1}
\\
0\enskip \text{for}\enskip (t,x)\in([0,T]\times \R^n)\setminus Q_n
\end{cases}\,,
\end{aligned}\]
and such that 
\[
\sup_{(t,x)\in[0,T]\times \R^d}\left(|\partial _t\psi_t (n;x)| + \sum\nolimits_{i=0}^3|\nabla ^i\psi _t(n;x)|\right)
\leq C 8^n\,,
\]
where the constant $C>0$ is independent of $n\geq 0$ (it is easy to see that such sequence exists).

Since $\psi (n)$ is smooth in time, thanks to the identity 
$\delta (v^2\psi (n))_{st}=\delta v^2_{st}\psi _s + v^2_t\delta \psi _{st}(n),$
we have for any $s,t\geq 0$ such that $\tau _n\leq s\leq t\leq T$:
\begin{multline}
\label{eq:v_squared_psi}
\delta (\int_{\R^d}v^2\psi (n)\d x)_{st}
+
\iint_{[s,t]\times \R^d}|\nabla v|^2\psi_s(n) \d x\d r
\\
\lesssim_{\lambda }
\iint_{[s,t]\times \R^d} \Big[|v_t|^2|\partial _t\psi (n)| +|v||\nabla v||\nabla \psi_s(n)|+\varkappa|\f||\nabla v||v|^{1-\frac{2}{\varkappa}}\psi_s(n)  +\varkappa |\f||v|^{2-\frac{2}{\varkappa}}|\nabla \psi_s (n)|\Big]\d x\d r
\\
\quad +\left(|\psi _s(n)|_{W^{1,\infty}}\omega _B(s,t)^\alpha +|\psi _s(n)|_{W^{2,\infty}}\omega _B(s,t)^{2\alpha } \right)\int_{\R^d}v_s^2\d x
+|v^{2,\natural}_{st}|_{W^{-3,1}}|\psi_s(n)|_{W^{3,\infty}}\,.
\end{multline}
Making use of the following estimates for $\varkappa\geq 2$:
\[
v^{1-2/\varkappa }\leq 1+v,\quad 
v^{2-2/\varkappa }\leq 1+v^2,
\]
and then letting
\[\rho:=\frac{2r}{r-1}\quad \text{and}\quad \sigma :=\frac{2q}{q-1},
\]
we infer thanks to H\"older Inequality that
\begin{multline}
\label{estim:drift_v}
|\delta \mathscr D^{(\varkappa)}_{st}|_{W^{-1,1}}
\lesssim
\|\nabla v\|_{2,2}^2 +\|v\nabla v\|_{1,1}
\\
+\varkappa\Big(\|\f\|_{2r,2q}\|\nabla v\|_{2,2}\|v\|_{\rho,\sigma }
+\|\f\|_{r,q}\|v\|_{\rho,\sigma }^2
+\|\f\|_{2,2}\|\nabla v\|_{2,2} + \|\f\|_{1,1}
\Big).
\end{multline}
where for notational ease we now use the shorthand notation:
\[
\|\cdot \|_{a,b}:=\|\cdot \|_{L^a(s,t;L^b)}\,.
\]
Going back to \eqref{eq:v_squared_psi} and applying Proposition \ref{pro:apriori} and H\"older Inequality, we obtain the inequality
\begin{multline}
\label{estim:v_B}
E_{Q_{n+1}}:=\sup_{\tau _{n+1}<t<T}\int_{|x|< R_{n+1}}|v_t|^2\d x
+ \iint_{\tau _{n+1}<t<T,\,|x|<R_{n+1}}|\nabla v_t|^2\d x\d t
\\
\leq C(r,q,\lambda)8^n\varkappa^2 \Big(E_{Q_n}
+\|\f\|_{2,2}^2 +\|\f\|_{1,1}
+(\|\f\|_{2r,2q}^2+\|\f\|_{r,q})\|v\mathbf 1_{Q_n}\|_{\rho , \sigma }^2\Big).
\end{multline}
where $\mathbf 1_{Q_n}(x)$ is the indicator function of $Q_n$,
and where the above constant depends on the indicated quantities but not on $\varkappa\geq 2.$

We now want to apply Lemma \ref{lem:recursion}.
To this end, observe first that thanks to \eqref{cond:r_q_strict}, there exists $\epsilon >0$ such that
\begin{equation}
\frac{1}{r} + \frac{d(1+\epsilon q)}{2q}\leq 1\,.
\end{equation} 
For such $\epsilon >0,$ is is easily seen that
\[
\frac{1}{\rho(1+\epsilon )}+ \frac{d}{2(1+\epsilon )\sigma }\geq \frac{d}{4},
\]
which means in particular that the exponents 
\[
\rho (1+\epsilon ),\quad 
\sigma (1+\epsilon )
\]
still satisfy the condition \eqref{conditions}.

Let $n\geq 0.$
In \eqref{estim:v_B}, making the substitution
$\varkappa:= \varkappa_n = 2(1+\epsilon )^n,$
we obtain thanks to Proposition \ref{pro:interp}
\[\begin{aligned}
&\||u|^{(1+\epsilon )^n}\mathbf 1_{Q_{n+1}}\|_{\rho (1+\epsilon ),\sigma(1+\epsilon)}
\\
&\leq C(E_{Q_{n+1}})^{1/2}
\\
&\leq \widetilde C8^{n} (1+\epsilon )^{n}
\Big(1+(E_{Q_n})^{1/2} +\||u|^{(1+\epsilon )^n}\mathbf 1_{Q_n}\|_{\rho , \sigma }\Big)\,.
\end{aligned}\]
from which it follows that
\begin{equation}
\label{interp_rec}
\begin{aligned}
\|u\mathbf 1_{Q_{n+1}}\|^{(1+\epsilon )^{n}}_{\rho (1+\epsilon )^{n+1},\sigma(1+\epsilon)^{n+1}}
\leq \widetilde C8^{n} (1+\epsilon )^{n}
\Big(1+(E_{Q_n})^{1/2} +\|u\mathbf 1_{Q_n}\|^{(1+\epsilon)^n }_{\rho (1+\epsilon )^n,\sigma (1+\epsilon )^n}\Big),
\end{aligned}
\end{equation}
where to obtain the first estimate we have used the interpolation inequality \eqref{interp_u} on $|u|^{(1+\epsilon )^n}.$
Otherwise stated, if one defines the sequence
\[
\Phi _n:=1+ E_{Q_n}^{1/2} + \|u\mathbf 1_{Q_n}\|^{(1+\epsilon )^n}_{\rho (1+\epsilon )^{n+1},\sigma (1+\epsilon )^{n+1}}
\,,\quad n\geq 0,
\]
one sees that for every $n\geq 0$:
\[
\Phi _{n+1}\leq \gamma [8(1+\epsilon )]^n\Phi _n^{1+\epsilon }
\]
for some constant
$\gamma =\gamma \left(\lambda,r,q,\|f\|_{\mathscr M},\omega _B,\alpha\right) >0.$
Applying now \eqref{estim:recursive}, this yields for every $n\in\N:$
\begin{equation}
\Phi _n\leq 
\gamma ^{\frac{(1+\epsilon )^n-1}{\epsilon }}[8(1+\epsilon )]^{\frac{(1+\epsilon )^n-1}{\epsilon ^2}-\frac{n}{\epsilon }}\|u\mathbf 1_Q\|_{\frac{2r}{r-1},\frac{2q}{q-1}}^{(1+\epsilon )^n}\,,
\end{equation}
and it follows that
\begin{equation}
\|u\|_{L^\infty(Q')}\leq \lim_{n\to\infty}(\Phi _n)^{(1+\epsilon )^{-n}}
\leq 
C\|u\|_{L^{\frac{2r}{r-1}}\left(L^{\frac{2q}{q-1}}\right)},
\end{equation}
for another constant $C>0$ as above.
By estimating the right hand side thanks to another application of the interpolation inequality, Proposition \ref{pro:interp}, we obtain the following $L^\infty$ bound
\begin{equation}
\|u\|_{L^\infty(Q')}\leq 
C'\left(\|u\|_{L^\infty(L^2)}+ \|u\|_{L^2(H^1)}\right)\,,
\end{equation}
but using the same Gronwall argument as in Section \ref{sec:free}, this quantity is in turn bounded in terms of $\lambda ,\alpha ,\omega _B,$ $|u_0|_{L^2}$ and $\|f\|_{L^2(H^{-1})}.$

Having this apriori estimate at hand, we can now proceed to the proof of Proposition \ref{pro:boundedness}.

\subsection{Proof of Proposition \ref{pro:boundedness}}
Consider an approximating sequence $\B(n)=S_2(\B(n)) $ as in Definition \ref{def:geometric}.
By the classical PDE theory, if we denote by $u(n) $ the corresponding weak solution (in the sense of distributions) of 
\begin{equation}
\label{classical}
\begin{aligned}
\frac{\partial u(n) }{\partial t} - A u(n) = f + \dot B(n) u(n)
\quad \text{on}\enskip [0,T]\times\R^d\,,
\\
u_0(n)=u^0\,.
\end{aligned}
\end{equation} 
then $u(n) $ is well defined and unique in the class $L^\infty(L^2)\cap L^2(H^1).$
It is easily seen that in fact, $u(n) \in \H^{\alpha ,2}_B$ and is an $L^2$-energy solution of 
\[
\d u(n) =(Au(n) +f)\d t +\d\B (n) u(n) \,.
\]

Moreover, for $f$ as in \eqref{ass:free_bounded}, it is known that $u(n) $ is continuous as a mapping from $[0,T]\times \R^d$ to $\R$ (it is even $\gamma $-H\"older for some $\gamma (\lambda )>0$ \cite{moser1964new}). For such level of regularity, it is shown by classical arguments (see for instance \cite[Chapter 3]{ladyzhenskaya1968linear}) that $v(n) := |u(n)|^{\varkappa/2}$ satisfies the chain rule \eqref{ito:real_power}, where $\B$ is replaced by $\B(n) .$
Consequently, the analysis made in the above paragraph ensures that for any compact set 
\[
Q\subset\subset (0,T]\times \R^d
\]
there is a constant $C_Q>0$ which is independent of $n\geq 0$ such that 
\[
\|u(n) \|_{L^\infty(Q)}\leq C_Q,
\]
Using Banach Alaoglu Theorem, the weak-$*$ lower-semicontinuity of the essential supremum,
and also the uniqueness of the limit $u$ in $L^\infty(L^2)\cap L^2(H^1)$, we see that $u$ satisfies the same estimate.
This proves the proposition.
\hfill\qed

\section{Proof of It\^o Formulas}
\label{sec:proof:ito}

In order to prove Theorem \ref{thm:ito_transport}, we first demonstrate that the It\^{o} Formula holds when $u$ is locally bounded and $F$ is admissible.
The proof of this fact is based on a reiteration of the product formula obtained in Section \ref{sec:space}, allowing to show the claimed property on polynomials of a solution. The fact that polynomials are dense in $C^2$ is then used together with the remainder estimates of Section \ref{sec:space} (it should be noted that this approach is similar to that of \cite[Theorem (3.3)]{revuz1999continuous}).
Approximating our solution by a sequence of such locally bounded elements, we will then show that the latter formula is preserved at the limit, proving the result in the general case.

\subsection{Case when $u$ is locally bounded}

Let $u$ be an $L^2$-energy solution of 
\begin{equation}
\label{generic}
\begin{aligned}
\d u= (A u+f)\d t + \d\B  u
\\
u_0\in L^2,
\end{aligned}
\end{equation}
where $f$ belongs to $L^2(H^{-1}),$ and such that moreover $\|u\|_{L^\infty(Q)}<\infty,$ for any $Q\subset\subset (0,T]\times\R^d.$

Fix a compact set of the form $Q:=[\tau ,T]\times K,$ where $K$ is compact and $\tau >0.$
If $P$ is a polynomial we infer by linearity and Corollary \ref{cor:product} that $P\circ u\in \H^{\alpha ,1}_{B}(Q)$ and that
\[
\d P(u) = P'(u)(Au+f)\d t + \d \B P(u)\,,\quad \text{on}\enskip [\tau ,T]\times K\,,
\]
in the $L^1(K)$-sense.

Since $P$ is admissible, i.e., $P'(0)=P''(0)=0$ and $|P''|_{L\infty}<\infty,$ then the inequalities
\[\begin{aligned}
&|P(z)|\leq |z|^2|P''|_{L^\infty},
\\
&|P'(z)|\leq |z||P''|_{L^\infty},
\quad \forall z\in \R,
\end{aligned}
\]
ensure that $P \circ u$ belongs to $L^\infty(0,T;L^1(\R^d))$ and similarly that $|\nabla u||P'(u)|$ is an element of $L^1(0,T;L^1(\R^d)).$
Hence, a direct evaluation shows that for $P$ as above, it holds
\begin{equation}
\label{pre_BLT:1}
\|P(u)\|_{L^\infty(L^1)\cap L^1(W^{1,1})}\leq C(|P''|_{L^\infty(\R)})\|u\|_{L^\infty(L^2)\cap L^2(H^1)}\,.
\end{equation} 

Similarly, the drift term $\mathscr D:=\int^{\cdot }_0P'(u)(Au+f)\d r$ belongs to $\V_1^1(0,T;W^{-1,1}(\R^d))$ as can be seen by the estimate
\begin{multline}
\label{pre_BLT:2}
|\delta \mathscr D_{st}|_{W^{-1,1}}
\leq 
\int_s^t\big|P'(u)(Au +f)\big|_{W^{-1,1}}\d r = : \omega _{\mathscr D}(s,t)
\\
\leq 
C\left(\lambda ,\|u\|_{L^\infty(L^2)\cap L^2(H^1)},\|f\|_{L^2(H^{-1})},|P''|_{L^\infty}\right)\,.
\end{multline}
Hence, from Proposition \ref{pro:apriori}, we obtain the following estimate in $\H^{\alpha ,1}_B(Q):$
\begin{equation}
\label{pre_BLT:3}
\|P(u)\|_{\H^{\alpha ,1}_B(Q)}\leq C\left(\lambda ,\|u\|_{L^\infty(L^2)\cap L^2(H^1)},\|f\|_{L^2(H^{-1})},|P''|_{L^\infty(\R)}\right)\,.
\end{equation}
Denote by $\mathcal{P}_{\mathrm{adm}}$ the set of admissible polynomials as above, equipped with the norm 
\[|P|_{C^2_{\mathrm{adm}}}:=|P''|_{L^\infty(\R)}.\]
The estimate \eqref{pre_BLT:3} shows that we have constructed a map
\[
\begin{aligned}
\varphi _u \colon\mathcal{P}_{\mathrm{adm}} 
&\longrightarrow \H^{\alpha,1}_{B,\loc}((0,T]\times\R^d)\,,
\\
P\enskip 
&\longmapsto  \varphi_u(P):=P\circ u\,\,,
\end{aligned}
\]
which is \textit{linear and bounded}.
By a classical result of functional analysis, it can therefore be uniquely extended to a mapping 
\begin{equation}
\label{u_star}
u^* \colon C^2_{\mathrm{adm}}\longrightarrow \H^{\alpha,1}_{B,\loc}((0,T]\times\R^d)
\end{equation} 
which satisfies the same estimates as $\varphi _u $, namely \eqref{pre_BLT:3} holds with $F\in C^2_{\mathrm{adm}}$ instead of $P.$ Considering any converging sequence $P_n\to F$ in $C^2_{\mathrm{adm}},$ and then making use of Lemma \ref{lem:stability}, it is easily checked that $(u^*(F))_t(x)=F(u_t(x)),$ for every $t\in[0,T]$ and almost every $x\in\R^d.$ This demonstrates in particular that $F\circ u$ is a well-defined element of $\H^{\alpha ,1}_{B,\loc}((0,T]\times\R^d)$ and that in the $L^1$-sense:
\begin{equation}
\label{F_holds}
\d (F(u)) = F'(u)(Au +f)\d t + \d\B  (F(u))\quad \text{on}\enskip [\tau ,T]\times K.
\end{equation} 
Since by assumption, $u$ belongs to the class $\H_B^{\alpha ,2}([0,T]\times\R^d)$ and $F$ is admissible, neither of the terms in the right hand side of \eqref{pre_BLT:3}, with $P$ replaced by $F$, depend on the choice of $Q\subset\subset \R^d$. It is therefore easy exercise left to the reader that the localization (with respect to both variables) can be removed. Hence \eqref{F_holds} holds in fact on $[0,T]\times\R^d,$ which shows the claimed It\^o formula when $u$ is locally bounded.

We can now turn to the proof of the general case.

\subsection{Proof of Theorem \ref{thm:ito_transport}}

By density one can consider sequences $(f(n))$ and $(u_0(n))$ such that for every $n\in\N,$
$f(n)$ satisfies Assumption \ref{ass:free_bounded}, 
and such that as $n\to \infty:$
\begin{align}
\label{convergence_2}
&f(n)\to f\quad \text{strongly in}\enskip  L^2(H^{-1})\,.
\intertext{By Proposition \ref{pro:boundedness}, the corresponding solution $u(n)\in\H^{\alpha,2}_B$ is locally bounded away from $t=0$,
	and moreover, by the continuity shown in Theorem \ref{thm:free} we have}
\label{convergence_3}
&u(n)\to u\quad \text{strongly in}\enskip L^\infty(L^2)\cap L^2(H^1).
\intertext{Moreover, from \eqref{convergence_3},
	there exists a subsequence (still denoted by $u(n)$ in the sequel) such that}
\label{convergence_4}
& u(n)\to u\quad \text{almost everywhere on}\enskip [0,T]\times \R^d.
\end{align}

By the intermediate result shown in the above paragraph, if $F\in C^2_{\mathrm{adm}},$
we have $F(u(n))\equiv u(n)^*(F)\in \H^{\alpha ,1}_B,$ and moreover, for every $\phi \in W^{3,\infty}:$
\begin{multline}
\label{eq:Fn_final}
\langle \delta F(u(n)),\phi \rangle
- \Big\langle (B^{1}_{st}+B^{2}_{st})\left[F(u_s(n))\right] + F(u(n))^{\natural}_{st},\phi \Big\rangle
\\
= -\iint_{[s,t]\times\R^d}\Big[
a^{ij}F'(u(n))\partial _ju(n)\partial _i\phi 
+a^{ij}F''(u(n))\partial _ju(n)\partial _iu(n)\phi
\\
+\f^i(n)\partial _iu(n)F''(u(n))\phi +\f^i(n)F'(u(n))\partial _i\phi +\f^0(n)F'(u(n))\phi 
\Big]\d x\d r,
\end{multline}
where $(\f^i(n))_{i=0,\dots,d},$ denotes any anti-derivative associated with $f(n).$

As mentioned before, for each $n\in\mathbb{N},$ the operator norm of the extended linear map $u(n)^*,$ which is defined in \eqref{u_star}, is the same as that of $\varphi _u$. As a consequence, the estimate \eqref{pre_BLT:3} remains true if the polynomial $P$ is replaced by $F.$ In particular, there is a constant $C$ such that for any $n\in\mathbb{N}:$
\begin{equation}
\label{bd:H_uniform}
\| F(u(n))\|_{\H^{\alpha ,1}_B}\leq C.
\end{equation}
By Lemma \ref{lem:stability}, the conclusion will follow by \eqref{convergence_4} and identification of the weak limits, provided one can show that
\begin{multline*}
(v(n);g^0(n),g^i(n)):=\Big(F(u(n));a^{ij}\partial _iu(n)\partial _ju(n),a^{ij}\partial _jF'(u(n))\Big)\,,\quad n\in\N\,,
\\
\text{is uniformly integrable.}
\end{multline*}
But using the pointwise estimates $|v(n)|\lesssim u(n)^2,$ $|g^0(n)|\lesssim |\nabla u(n)|^2$ and $|g^i(n)|\lesssim|\nabla u(n)|^2 + |u(n)|^2$, this property is an obvious consequence of the strong convergence \eqref{convergence_3}.
This finishes the proof of Theorem \ref{thm:ito_transport}-(i). The proof of the second item is similar and therefore omitted.
\hfill\qed

\subsection{The $L^p$-norm of $L^p$ solutions: proof of Corollary \ref{cor:L_p_transport}}
For $R>0$ we define an admissible truncation $F_R$ of $|\cdot |^p$ as follows.
Let $\theta \in C^\infty_c,$ supported in $[0,2)$ such that $\theta =1$ on $[0,1]$ while $0\leq \theta \leq 1.$
Define
\[
F_R(z):= \int_0^{|z|} \d y\int_0^y \theta \left(\frac{|\tau |}{R}\right) p(p-1)|\tau |^{p-2}\d\tau ,\quad z\in\R.
\]
Clearly, $|F_R''|_{L^\infty}<\infty,$ and $F_R(0)=F_R'(0)=0$, so $F_R$ is admissible. Moreover, as $R\to\infty,$ $F_R\nearrow |\cdot |^p$ almost everywhere and locally uniformly.

We have by Theorem \ref{thm:ito_transport}:
\begin{multline}
\Big\langle \delta F_R(u)_{st}-(B^{1}_{st}+B^{2}_{st})[F_R(u)] - F_R(u)^\natural_{st},\phi \Big\rangle 
\\
= -\iint_{[s,t]\times\R^d}\Big[
a^{ij}F_R'(u)\partial _ju\partial _i\phi 
+a^{ij}F_R''(u)\partial _ju\partial _iu\phi
\\
+\f^i\partial _iuF_R''(u)\phi +\f^iF_R'(u)\partial _i\phi
\Big]\d x\d r
\\
\lesssim _{\lambda ,p,\theta }|\phi |_{W^{1,\infty}}\left(\iint_{[s,t]\times \R^d}|u|^{p-1}|\nabla u|+|u|^{p-2}|\nabla u|^2 + |\f||\nabla u||u|^{p-2} + |\f||u|^{p-1}\right)
\\
\lesssim_{\lambda ,p,\theta }\|u\|_{L^p(L^p)}^{p-1}\left(\|\nabla u\|_{L^p(L^p)}+\|f\|_{L^p(W^{-1,p})}\right)
\\
+ \|u\|_{L^p(L^p)}^{p-2}\left(\|\nabla u\|_{L^p(L^p)}+\|\nabla u\|^2_{L^p(L^p)}\|f\|_{L^p(W^{-1,p})} \right)\,.
\end{multline}
The above drift term is therefore uniformly bounded in $R>0,$ and so is $\|F_R(u)\|_{\H^{\alpha ,1}_B}$ by Proposition \ref{pro:apriori}.

By Lemma \ref{lem:stability}, this implies that one can take limits as $R\to \infty,$ in the above weak formulation. But this means that \eqref{ito_p_transport} holds, which finishes the proof.\hfill\qed

\subsection{The $L^p$-norm in the general case: proof of Theorem \ref{thm:L_p}}
Uniqueness is easy and therefore we only sketch the proof.
If $u^1$ and $u^2$ are two such solutions, then $v:=u^1-u^2$ is also a solution of the same equation with $0$ instead of $f.$ Using the It\^o formula on $|v|^p,$ and testing against $\phi =1,$ we find thanks to Proposition \ref{pro:apriori} that the $L^\infty(s,t;L^1)$-norm of $v$ satisfies an incremental inequality of the form \eqref{rel:gron} with $\varphi (s,t)=0$. The conclusion then follows by the rough Gronwall argument, Lemma \ref{lem:gronwall}, and the fact that $v_0=0.$

To show existence, we first adapt the compactness argument used in Section \ref{sec:free} for the $L^2$-theory.
\\

\textit{Step 1: compactness argument}
Let us first consider the case when $B=X\cdot \nabla +X^0\in C^\infty(0,T;\DD_1),$
and let $u$ be the unique distributional solution of 
\[\begin{aligned}
\partial _tu-Au=f+ \left(\dot X\cdot \nabla + \dot X^0\right)u\quad \text{on}\enskip (0,T]\times \R^d\,,
\\
u_0:=u^0\in L^p\,.
\end{aligned}
\]
From the classical PDE theory and our definition of the spaces $\H^{\alpha ,p}_{B}$ it is straighforward to check that $u\in \H^{\alpha ,2}_{B}.$ Moreover, it is standard that in the distributional sense
\[\partial _t(|u|^p)= pu|u|^{p-2}(Au + \partial _i\f^i+\f^0) + \dot X\cdot \nabla (|u|^p) + p\dot X^0|u|^p
\]
and, by the consistence of rough integration with Lebesgue/Stieljes integration, it holds in that case
\begin{equation}
\label{ito_approx_Lp}
\d |u|^p-pu|u|^{p-2}(Au+ f)\d t = \d \B^{(p)}|u|^p \,,
\end{equation} 
in the sense of Definition \ref{def:weak_sol} in $L^1,$ and where $\B^{(p)}:=S_2(X\cdot \nabla +pX^0).$
Let $\f^i\in L^1(L^p)\cap L^2(L^2),i=0,\dots ,d$ be any antiderivative of $f$.
Integrating, we have using H\"older Inequality
\begin{equation}
\label{integrated_Lp}
\begin{aligned}
\delta \left(|u|_{L^p}^p\right)_{st} 
&+ \iint_{[s,t]\times\R^d} |u|^{p-2}|\nabla u|^2\d x\d t
\\
&\lesssim_{\lambda ,p}
\delta \left(|u|_{L^p}^p\right)_{st}
+\iint_{[s,t]\times\R^d} p(p-1)a^{ij}|u|^{p-2}\partial _iu\partial _ju \d x\d t
\\
&=
\iint_{[s,t]\times\R^d}\big[pu|u|^{p-2}\f^0 - p(p-1)|u|^{p-2}\partial _iu \f^i \big]\d x\d t
\\
&\quad \quad \quad 
+\int_{\R^d}|u_s|^p(B^{(p),1,*}_{st} + B^{(p),2,*}_{st})1\d x 
+\langle |u_{st}|^{p,\natural},1\rangle
\\
&\lesssim_{\lambda ,p}
\|u\|_{\infty,p}^{p-1}\|\f^0\|_{1,p} + \||u|^{p-2}|\nabla u|^2\|_{1,1}^{1/2}\|u\|_{\infty ,p}^{\frac{p-2}{2}}\|\f^i\|_{2,p}
\\
&\quad \quad \quad \quad 
+\|u\|_{\infty,p}^p\left(\omega _B(s,t)^\alpha +\omega _B(s,t)^{2\alpha }\right)
+\nn{|u|^{p,\natural}}{3}(s,t)
\end{aligned}
\end{equation} 
where we recall the shorthand notation $\|\cdot \|_{a,b}:=\|\cdot \|_{L^a(s,t;L^b)}.$
But thanks to the remainder estimates, Proposition \ref{pro:apriori}, we find for $|t-s|\leq L(\rho _\alpha (\B))$ small enough:
\[\begin{aligned}
\delta &\left(|u|_{L^p}^p\right)_{st} 
+ \iint_{[s,t]\times\R^d} |u|^{p-2}|\nabla u|^2\d x\d t
\\
&\lesssim_{\lambda ,p}
\left(\|u\|_{\infty,p}^{p-1}\|\f^0\|_{p,p}(t-s)^{\frac{p-1}{p}} + \||u|^{p-2}|\nabla u|^2\|_{1,1}^{1/2}\|u\|_{\infty ,p}^{\frac{p-2}{2}}\|\f^i\|_{2,p}\right)(1+\omega _B(s,t)^\alpha )
\\
&\quad \quad 
\omega _B(s,t)^\alpha \left(\|u\|_{\infty,p}^{p-1}\|\f^i\|_{p,p}(t-s)^{\frac{p-1}{p}}
+\||u|^{p-2}|\nabla u|^2\|_{1,1}^{1/2}\|u\|^{p/2}_{\infty,p}(t-s)^{1/2}
\right)
\\
&\quad \quad 
+\|u\|_{\infty,p}^p\left(\omega _B(s,t)^\alpha +\omega _B(s,t)^{2\alpha } + \omega _{B}(s,t)^\alpha \right)
\end{aligned}
\]
Using Young Inequality, taking $L(\rho _\alpha (\B),\lambda )$ smaller if necessary and then absorbing to the left, we end up with the inequality
\[
\delta \left(|u|_{L^p}^p\right)_{st} 
+ \iint_{[s,t]\times\R^d} |u|^{p-2}|\nabla u|^2\d x\d t
\lesssim_{\lambda ,p}
\|u\|_{\infty,p}^p[\omega _B(s,t)^\alpha +(t-s)]+\|\f^0,\f^i\|_{p,p}^p
\]
By the rough Gronwall Lemma, Lemma \ref{lem:gronwall}, we obtain the estimate on 
\begin{equation}
\label{est:u_p_nabla_p}
\|u\|_{L^\infty(L^p)}^p+ \iint_{[0,T]\times\R^d}|u|^{p-2}|\nabla u|^2\,\d x\d t \leq C\left(\lambda ,p,\rho_\alpha (\B),\|f\|_{L^p(W^{-1,p})}\right)\,.
\end{equation}

Now, consider a sequence of canonical lifts $\B(n)=S_2(X(n)\cdot\nabla +X^0(n))$ such that $X(n)$ is smooth in time, $\B(n)\to \B,$
and define the differential rough driver $\B^{(p)}(n)$ correspondingly.
Note that for each $n\geq 0,$ the map $v(n):=|u(n)|^p$ belongs to $\H^{\alpha ,1}_{B^{(p)}(n)},$ since the smoothness of $X(n)$ in time makes trivial the statement about the remainder 
\[R_{st}^{v(n)}=\delta v_{st}(n)-B^{(p),1}_{st}(n)v_s(n)\,,\]
in the definition of the controlled path space $\mathcal D^{\alpha ,1}_{B(n)}$.
Thanks to the convergence of $\B(n),$ it is immediately checked that $\rho _\alpha (\B^{(p)}(n),\B^{(p)})\to 0$ (the $\rho _\alpha $-convergence sense is equivalent to the convergence of the coefficients, see Appendix \ref{app:algebraic}).

% 
% By the $L^2$ theory, we have the uniform bound
% \[
% \|u(n)\|_{\H^{\alpha ,2}_{B(n)}}\leq C\left(\lambda ,\|f\|_{L^2(H^{-1})}\right)\,,
% \]
% and $u(n)$ converges, in the sense of \eqref{weak}-\eqref{stab:1}, towards the unique $L^2$-energy solution $u\in\H^{\alpha ,2}_B$ of the corresponding equation with $\B.$

Moreover, thanks to the identities \eqref{identities} and the remainder estimates (Proposition \ref{pro:apriori}), the estimate \eqref{est:u_p_nabla_p} implies the following uniform estimate on $v(n)=|u(n)|^{p/2}$
\[
\||u(n)|^{p/2}\|_{\H^{\alpha ,2}_{B^{(p/2)}(n)}}\leq C\left(\lambda ,p,\|f\|_{L^p(W^{-1,p})}\right)\,.
\]
Applying Lemma \ref{lem:stability}, one infers the existence of $v\in \H^{\alpha ,2}_{B^{(p/2)}}$ such that $v(n)\to v$ weakly-$*$ in $L^\infty(L^2)\cap L^2(H^1).$
Interpolating the $L^2(H^1)$-estimate with the $\V^\alpha (H^{-1})$ estimate, it is easily seen that the convergence of $v(n)$ holds strongly in $L^2(L^2_\loc)$ and thus, upon taking a subsequence we can assume that
\[
\begin{aligned}
|u(n)|^{p/2}\to |u|^{p/2}\,,\quad \text{in}\enskip L^2(0,T;L^2_\loc)\enskip \text{strong, and}
\\
u(n)\to u\quad \text{almost everywhere in}\enskip [0,T]\times\R^d\,.
\end{aligned}
\]

Using again the remainder estimates, Proposition \ref{pro:apriori}, it follows from the equation on $|u(n)|^p$ that  
\[
\||u(n)|^p\|_{\mathcal D^{\alpha ,1}_{B(n)}}\leq C\,.
\]
Therefore, by the same compactness argument as in the proof of Lemma \ref{lem:stability}, there exists $w$ and $g^i,i=0\dots ,d$ in $(L^\infty)^*$ so that for any $\Phi \in L^\infty([0,T]\times\R^d),$
\begin{align}
\label{convergence_w}
&\iint_{[0,T]\times \R^d} |u(n)|^p\Phi \d x \d t \to \langle w,\Phi \rangle_{(L^\infty)^*,L^\infty}
\\
\label{g_zero}
&p(p-1)\iint_{[0,T]\times \R^d} a^{ij}|u(n)|^{p-2}\partial _iu(n)\partial _ju(n)\Phi \d x \d t \to \langle g^0,\Phi \rangle_{(L^\infty)^*,L^\infty}
\\
&p\iint_{[0,T]\times \R^d} u(n)|u(n)|^{p-2}\partial _ju(n)\Phi \d x\d t \to \langle g^i,\Phi \rangle_{(L^\infty)^*,L^\infty}\,,\quad i=1,\dots,d\,.
\end{align} 
It remains to show that the above limits are the expected ones (thereby proving that the above convergences hold in $L^1$-weak).
\\

\textit{Identification of the limits and conclusion}
Using the strong convergence of $v(n)=|u(n)|^{p/2},$
we also find
\[
\iint_{[0,T]\times \R^d} |u(n)|^p\Phi \d x \d t=\iint_{[0,T]\times \R^d} |u(n)|^{p/2}(|u(n)|^{p/2}\Phi) \d x \d t\to \iint_{[0,T]\times \R^d} |u|^p\Phi \d t\d x\,,
\]
and therefore we see that $w= |u|^p.$
To conclude, it remains to show that
\begin{align}
\label{assertion_L1}
&g^0=a^{ij}|u|^{p-2}\partial _i u \partial _ju\quad
\\
&g^j=u|u|^{p-2}\partial _ju\,.
\end{align}
We content ourselves to show the first assertion since the other one is similar.

In order to prove \eqref{assertion_L1}, observe first that since $p\geq4$, it is also larger than $2$ and thus the sequences $\{u(n),n\in\N\}$ and $\{u^2(n),n\in \N\}$ are also uniformly bounded in the $\H^{\alpha ,2}_{B(n),\loc}$ (respectively $\H^{\alpha ,1}_{B^{(2)}(n)}$)-sense. The Banach Alaoglu Theorem implies the existence of $\mu \in (L^\infty)^*$ so that $a^{ij}\partial _iu(n)\partial _ju(n)\rightharpoonup \mu $ weakly-$*.$ On the other hand $\nabla u(n)\rightharpoonup \nabla u$ in $L^2_{\w}$, and thus applying the local product formula of $u$ with itself, we find that necessarily
\begin{equation}
\label{mu_equal}
\mu =a^{ij}\partial _iu\partial _ju\,.
\end{equation} 
But since $p\geq 4,$ replacing $p$ by $p-2$ in the previous step, we see that there exists $h^0$ in $(L^\infty)^*$ so that \eqref{g_zero} holds with $(p-2,h^0)$ instead of $(p,g^0),$ and it is easily seen that \[p(p-1)u^2h^0=(p-2)(p-3)g^0\,.\]
Applying the product formula, Proposition \ref{pro:product}, to $|u|^{p-2}$ with $u^2$, we see thanks to \eqref{mu_equal} that
$g^0=u^2h^0+ 2\mu |u|^{p-2}=\frac{(p-2)(p-3)}{p(p-1)}g^0 + 2a^{ij}\partial _iu\partial _ju|u|^{p-2},$ which after simplification provides the relation \eqref{assertion_L1}.

Hence the chain rule \eqref{ito_approx_Lp} remains true for $u$ which we recall is the unique solution in the class described by the hypotheses of the theorem. This finishes the proof. \hfill\qed

\section{Proof of Theorem \ref{thm:max_principle}}
\label{sec:max}

We start with the following elementary observation.
For a domain $D\subset \R^d$ with smooth boundary, elements of $W_0^{k,p}(D)$ for $0\leq k\leq 3$ and $p\in[1,\infty]$
are naturally identified in $W^{k,p}(\R^d)$ through the embedding map
\[
\iota_ D:W_0^{3,p}( D)\hookrightarrow W^{3,p}(\R^d),
\]
where for any $\phi $ in $W_0^{3,p}( D),$ we define
\[
\iota _ D\phi (x):=
\begin{cases}
\phi (x)\quad \text{if}\enskip x\in  D
\\
0\quad \text{if}\enskip x\notin  D\,.
\end{cases}
\]
This operation is of course linear and continuous.
In particular, by duality, for every distribution $g\in W^{-3,p'}(\R^d),$
the restriction $g|_{ D}\equiv\iota_ D^*g$ to a smooth domain $ D$ is well defined.

\subsection{Proof of the solvability}

Identify the test functions $W^{k,p}_0( D)$ as elements of $W^{k,p}(\R^d)$ as in the above discussion, and then define
\[
\tilde\sigma :=\iota_D(\sigma) ,\quad 
\mathbf{\tilde B}:=(\tilde B^1,\tilde B^2):= (Z^1{\tilde\sigma} \cdot \nabla ,Z^2 ({\tilde\sigma} \cdot \nabla )^2).
\]
Moreover, let ${\tilde u}_0 :=\iota_D (u_0).$
Concerning the elliptic part, we define
\begin{align*}
&\tilde a^{ij}(t,x):=\begin{cases}
a^{ij}(t,x)\quad \text{if}\enskip (t,x)\in [0,T]\times D\\
\ind_{i=j}\quad \text{otherwise},
\end{cases}
\end{align*}
and we let $\tilde A:=\partial _i(\tilde a^{ij}\partial _j\cdot ).$
With these definitions, $\tilde A,\mathbf{\tilde B},$ fulfill the hypotheses of Theorem \ref{thm:free} so that there exists a unique $L^2$-energy solution $u\in \H^{\alpha,2}_B([0,T]\times\R^d)$ to 
\begin{equation}
\d u=\tilde A u \d t + \d \mathbf{\tilde B}u,\quad \text{on}\quad [0,T]\times \R^d.
\end{equation} 

The restriction $v:= u|_{[0,T]\times D}$ is the natural candidate to solve the Dirichlet problem \eqref{dirichlet}. In order to check that this is indeed the case,
let us remark that $w:=u|_{[0,T]\times(\R^d\setminus D)}$ is a classical solution to
\[
\partial _tw=\Delta w \quad \text{on}\enskip [0,T]\times (\R^d\setminus D),
\quad w_0=0,
\]
and hence $w=0.$ This shows that $u$ is supported in $[0,T]\times D.$ Since on the other hand $u$ belongs to $L^2(H^1(\R^d)),$ this implies that its trace onto $[0,T]\times \partial D$ is well defined, so that $v\in L^2(H^1_0(D)).$ This shows that $v$ solves the Dirichlet problem \eqref{dirichlet}.

\subsection{Proof of the maximum principle}
The proof uses the so-called Stampacchia truncatures approach.
We first assume that
\begin{equation}
\label{hyp:a_add}
a\in L^1(0,T;W_0^{1,\infty}( D)).
\end{equation}
Namely, let us fix a map $G\in C^1(\R)$ such that the following properties are satisfied:
\[\left[\begin{aligned}
&|G'|_{L^\infty(\R)}<\infty,
\\
& G\enskip \text{is increasing on}\enskip (0,\infty),
\\
&G(x)=0\enskip \text{whenever}\enskip x\leq 0.
\end{aligned}\right.
\]
Let $F\in C^2(\R)$ be defined by 
\[
F(x):=\int_0^{x-M}G(y)\d y,\quad x\in\R,
\]
where we denote by
\[
M=\max(0,\esssup\nolimits_{ D}u_0)<\infty.
\]

By Theorem \ref{thm:ito_transport} applied to $F$ (note that $u$ has compact support) the following equation holds:
\[
\langle\delta  F(u)_{st},\phi \rangle = \int_s^t \langle G(u_r-M)Au_r,\phi \rangle\d r
+\langle(B^1+B^2)_{st}F(u_s),\phi \rangle + \langle F^\natural_{st},\phi \rangle,
\]
for some remainder $F^\natural\in \V^{1+}(0,T;W^{-3,1}).$
Next, we arrange the drift term as follows:
\[
\begin{aligned}
\langle G(u-M)Au,\phi \rangle 
+\langle a^{ij}G'(u-M)\partial _iu\partial _ju,\phi  \rangle
&=\langle -a^{ij}G(u-M)\partial _ju,\partial _i\phi  \rangle
\\
&=\langle F(u),\partial _j(a^{ij}\partial _i\phi) \rangle\,.
\end{aligned}
\]
Hence, denoting by $\mathscr D :=\int_0^\cdot G(u_r-M)A_ru_r\d r,$
we have for each $(s,t)\in \Delta :$
\[
|\delta \mathscr D _{st}|_{W^{-2,1}}
\leq \lambda^{-1}\iint_{[s,t]\times D} G'(u-M)|\nabla u|^2\d x\d r
+\|a\|_{L^1(s,t;W^{1,\infty})}\|F(u)\|_{L^\infty(s,t;L^1)}\,.
\]
Therefore, testing the equation against $\phi =1$ and then using Assumption \ref{ass:A} gives
\begin{multline}
\delta (|F(u)|_{L^1})_{st}
+\iint_{[s,t]\times  D} G'(u-M)|\nabla u|^2\d x\d r
\\
\lesssim_{\lambda}
\lambda^{-1}\|F(u)\|_{L^\infty(s,t;L^1)}\omega _B(s,t)^\alpha 
+ \|F(u)\|_{L^\infty(s,t;L^1)}\|a\|_{L^1(s,t;W^{1,\infty})},
\end{multline}
for any $(s,t)$ such that $\omega _B(s,t)\leq L(\lambda).$
Applying Lemma \ref{lem:gronwall}, we obtain that 
\[
\|F(u)\|_{L^\infty(L^1)}\leq C\left(\lambda,\|a\|_{L^1(W^{1,\infty})},\omega _B,\alpha \right)|F(u_0)|_{L^1}\equiv0,
\]
from which we conclude that $u\leq M$ a.e.
The proof of the estimate below is similar, hence omitted. This proves the desired inequality, when \eqref{hyp:a_add} holds.

For general coefficients $a^{ij}$, we consider an approximating sequence $a^{ij}(n),n\in\mathbb N,$ which converges almost everywhere and in $L^1$ to $a^{ij}$, and  such that for each $n,$  Assumption \ref{ass:A} is satisfied (with a uniform $\lambda$) and \eqref{hyp:a_add} holds. By Lemma \ref{lem:stability}, we can assume without loss of generality that the corresponding solution $u(n)$ converges almost everywhere to that associated with $a^{ij}.$ Taking the limit in \eqref{comparison} then proves the result.
This finishes the proof of Theorem \ref{thm:max_principle}.
\hfill\qed

\appendix
\section{Appendix: some technical proofs}
\label{sec:appendix}

\subsection{Proof of Lemma \ref{lem:multiplication}}
\label{app:algebraic}

It is well-known that a multiplication operator $M_f$ of the form $M_fh:=x\mapsto f(x)h(x)$ for $h\in L^2,$ is bounded if and only if $|f|_{L^\infty}<\infty,$ and that the map $f\in L^\infty\mapsto M_f\in \mathscr L(L^2,L^2)$ is an isometry (see for instance \cite{reed1980methods}). By an immediate generalization, for $i=1,2,$ we see that the couple $(j_1,j_2)$ defined as
\begin{equation}
\label{iso_j}
\begin{aligned}
&j_1\colon(W^{3,\infty})^d\times W^{2,\infty} \to \DD_1
&& (X,Y )\mapsto  X^i\partial _i +Y \,,
\\
&j_2\colon(W^{3,\infty})^{d\times d}\times (W^{2,\infty})^d \times W^{1,\infty} \to \DD_2
&&(\XX,\YY,\mathbb Z)\mapsto \XX ^{ij}\partial _{ij} + \YY^i\partial _i + \mathbb Z\,,
\end{aligned}
\end{equation} 
is a continuous isomorphism, where $\DD_i,i=1,2,$ are equipped with the operator-norm topologies as in Definition \ref{def:rough_driver}.

Let $t\mapsto B_t=X_t\cdot \nabla + X^0_t$ be in $C^1(0,T;\DD_1)$ and, as in \eqref{can_lift},
define the canonical lift 
\[(B^1,B^2):=S_2( B)\,.
\]

By definition of $B^2_{st} $,
we have for $0\leq s\leq t \leq T$:
\begin{equation}
\label{def:B2}
\begin{aligned}
B^2_{st}
&:=\int_s^t\d B_{r}\circ \delta B_{sr}
\\
&=\int_s^t(\d X^i_r\partial _i + \d X^0_r)\circ(X^j_{sr}\partial _j + X^0_{sr})
\\
&=\XX^{ij}_{st}\partial _{ij} + (\LL_{st}^i + 2\mathbb S^{0i}_{st})\partial _i + \LL^0_{st}+\mathbb S^{00}_{st}\,,
\end{aligned}
\end{equation} 
where we recall the notation $X_{st}:=X_t-X_s,$ and where we introduce
\begin{equation}
\label{coef_3}
\left[
\begin{aligned}
\XX _{st}^{ij}
&=\int_s^t X ^i_{sr}\d X^j_r,\quad 
\\
\LL _{st}^i
&=\int_s^t\d X ^\mu _r \partial _\mu X_{sr}^i\,,
\\
\mathbb S^{ij}_{st}
&=\sym \XX^{ij}_{st}:=
\frac12\left(\int_s^t X^i_{sr}\d X^j_r + \int_s^tX^j_{sr}\d X^i _r\right),
\quad \text{for all}\enskip 0\leq i,j\leq d\,.
\end{aligned}\right.
\end{equation}
The above integrals are understood in the sense of Bochner, in $W^{3,\infty},$ $W^{2,\infty}, W^{1,\infty}.$
As seen through immediate algebraic computations, the generalized Chen's relations \eqref{gene:chen} hold in this case, since
\begin{equation}
\label{app_chen_LL}
\begin{aligned}
\delta \XX_{s\theta t}
&=\LL _{st}^{i}-\LL_{s\theta }^{i} - \LL_{\theta t}^{i}
\\
&=\left(\int_s^t -\int_s^\theta \right)( \d X_r^\mu \partial _\mu X_{sr}^i)\d r 
-\int_{\theta}^ t(\d X_r^\mu \partial _\mu X_{\theta r}^\mu )\d r 
\\
&= X_{\theta t}^\mu \partial _\mu X_{s\theta }^i\,.
\end{aligned}
\end{equation} 

Next, for almost every $x\in\R^d$, an integration by parts in the time variable yields
the identity
\begin{equation}
\label{weak_geo_XX}
\mathbb S^{ij}_{st}(x)= \frac12X_{st}^i(x) X_{st}^j(x)\,,\quad i=0,\dots d.
\end{equation} 
Denoting by $\mathbb A_{st}^{ij}:=\XX_{st}^{ij}-\mathbb S_{st}^{ij},$ we further observe that Schwarz Theorem implies
\[
\XX_{st}^{ij}\partial _{ij} = \mathbb S_{st}^{ij}\partial _{ij} + \mathbb A_{st}^{ij}\partial _{ij} = \mathbb S_{st}^{ij}\partial _{ij}\,,
\]
since $\mathbb A_{st}$ is antisymmetric.
Hence, only the symmetric part of $\XX$ contributes to the second order part of $B^2_{st}$ in \eqref{def:B2}. This yields the desired expression, namely
\begin{equation}
\label{app_B2}
\begin{aligned}
B^2_{st}=\frac12 X^i_{st}X^j_{st}\partial _{ij} + \left(\LL^i_{st} + X^0_{st}X^i_{st}\right)\partial _i + \LL^0_{st} + \frac12(X^0_{st})^2\,.
\end{aligned}
\end{equation} 

To show \eqref{weak_G_2}, note that
\[\begin{aligned}
B^1_{st}\circ B^1_{st} 
&= ( X^i_{st}\partial _i+  X^0_{st})\circ( X^j_{st}\partial _j +  X^0_{st})
\\
&=X^i_{st} X^j_{st}\partial _{ij} + \left( X^j_{st}\partial _j X_{st}^i + 2 X^0_{st}  X^i_{st}\right)\partial _i + X^j _{st}\partial _jX^0_{st} +(X^0_{st})^2\,.
\end{aligned}
\]
This yields, by definition of $[\B]$:
\begin{equation}
\label{app_bracket}
\begin{aligned}
[\B]_{st}
&\equiv B_{st}^2-\frac12B_{st}^1\circ B_{st}^1
\\
&=\left(\LL^i_{st}- \frac12 X^j_{st}\partial _jX_{st}^i\right)\partial _i + \LL^0_{st} - X^j_{st}\partial _jX^0_{st}
\end{aligned}
\end{equation} 
which is the claimed equality.

Now, pick any geometric differential rough driver $\B$, and let $ \B (n)\in C^1(0,T;\DD_1),n\in \N,$ be such that $\B(n)\equiv S_2( B (n))\to_{\rho _\alpha } \B.$
Making use of the isomorphisms $(j_1,j_2)$ we see that the coefficients 
\[
( X(n),  Y(n);\XX(n),\YY (n),\mathbb Z(n))\equiv(j_1^{-1}B^1(n);j_2^{-1}B^2(n))
\]
converge to $(j_1^{-1}B^1;j_2^{-1}B^2),$ in the space 
\[
\left((W^{3,\infty})^d\times W^{2,\infty}\right)\times \left((W^{3,\infty})^{d\times d}\times (W^{2,\infty})^d\times W^{1,\infty}\right).
\]
In particular, one can take the limits in the identities \eqref{app_chen_LL},\eqref{app_B2},\eqref{app_bracket}, proving the corresponding relations for the limit $\B$.
\hfill\qed

\subsection{Renormalization property for geometric differential rough drivers}
\label{app:renorm}

In what follows, we fix $D\subset U\subset \R^d$ as in Section \ref{sec:space} and, recalling Notation \ref{nota:Omega}, we will further denote by
$\Omega :=\Omega ^{D}$ while $\Omega_\epsilon :=\Omega ^{D}_\epsilon. $

Given $\Phi (\cdot ,\cdot )$, we have for $(x,y)\in\Omega$, by definition of $T_\epsilon $:
\[
T_\epsilon \Phi (x,y):=\frac{1}{(2\epsilon )^d}\Phi \left(x_++\frac{x _-}{\epsilon },x_+-\frac{x _-}{\epsilon }\right),
\]
where we introduce the new coordinates
\begin{equation}
\label{new_coordinates}
x_+:=\frac{x+ y}{2},\quad ,x_-:=\frac{x-y}{2}\,.
\end{equation}
Note that the Jacobian determinant of the map $\chi \colon\Omega \to \R^d\times B_1,$ $(x,y)\mapsto(x_+,x_-)$ is equal to $2^{-d}$ (in fact $\sqrt2 \chi $ is a rotation).
By a common abuse of notation, we will denote by $\nabla _\pm$ the gradient with respect to the new coodinates $x_+(x,y)$ and $x_-(x,y).$ Formally, we have the relation $\nabla _\pm = \nabla_x \pm\nabla _y. $

The proof of Theorem \ref{thm:renorm} is based on the following result, whose proof is implicitly contained in \cite{deya2016priori}, and therefore omitted.

\begin{lemma}
	\label{lem:V}
	Let $V=\sigma ^i(\cdot )\partial _i$ be in $\DD_1.$ For a generic function $\psi :\R^d\to\R,$ denote by $\Psi (x,y):=\psi ((x-y)/2),$ and let $V_x$ (resp.\ $V_y$) be a shorthand for $V\otimes \id,$ (resp.\ $\id\otimes V$).
	For each $k=1,2,3$ and $\psi \in W^{3,\infty}$ with compact support in the unit ball $B_1\subset\R^d,$ it holds uniformly in $\epsilon \in(0,1]$:
	\[
	\left|(\nabla _{\pm})^{k-1}\circ T_\epsilon^{-1}\circ(V_x+V_y)\circ T_\epsilon [\Psi (x,y)]\right|\leq |\sigma |_{W^{k,\infty}}|\psi |_{W^{k,\infty}}\,.
	\]
	for a.e.\ $(x,y)\in \R^d\times\R^d$.
\end{lemma}

\bigskip

\textit{Proof of the Theorem.}
\begin{trivlist}
	\item[\indent\emph{Step 1: the key estimate}.]
	We first show that for $\Phi \in W^{k,\infty}_0(\Omega ),$ and with $V$ as in Lemma \ref{lem:V}:
	\begin{equation}
	\label{estimates_area}
	|(\nabla _{\pm})^{k-1}(T_\epsilon )^{-1}(V_x + V_y)T_\epsilon \Phi|_{L^\infty(\Omega )}
	\leq C|\sigma |_{W^{k,\infty}}|\Phi |_{W^{k,\infty}_0(\Omega )}\,.
	\end{equation} 
	By density, it will be enough to show \eqref{estimates_area} on functions of the form
	$\Phi (x,y)=\phi (\frac{x+y}{2})\psi (\frac{x-y}{2}),$ with $\psi $ compactly supported in $B_1.$
	For such $\Phi ,$ we have
	\begin{multline*}
	T_\epsilon ^{-1}(V_x + V_y)T_\epsilon \Phi(x,y)
	=T_\epsilon ^{-1}(V_x+V_y)\left[\phi (\frac{x+y}{2})\right]\psi (\frac{x-y}{2\epsilon })
	\\
	+\phi (\frac{x+y}{2})T^{-1}_\epsilon (V_x+V_y)\left[\psi (\frac{x-y}{2\epsilon })\right]
	=I_\epsilon +II_\epsilon  \,.
	\end{multline*}
	Using the new coordinates, we have the following expression for the first term:
	\[
	I_\epsilon= \frac{1}{2}(\sigma (x_++\epsilon x_-)+\sigma (x_+-\epsilon x_-))\cdot \nabla \phi (x_+)\psi (x_-)\,.
	\]
	By the commutation relations
	\begin{equation}
	\label{commutation_T}
	\nabla _+T_\epsilon =T_\epsilon \nabla _+,\quad \text{and}\quad \nabla _- T_\epsilon =\epsilon^{-1} T_\epsilon \nabla _-.
	\end{equation}
	it is then easily seen 
	(see \cite[Proposition 6.1]{hocquet2017energy} for details)
	that for $k=1,2,3:$
	\[
	\esssup_{x_+,x_-}|(\nabla _{\pm})^{k-1}I_\epsilon |\leq |\sigma|_{W^{k,\infty}}|\Phi |_{W^{k,\infty}}
	\leq C\omega _B(s,t)^{i\alpha } |\Phi |_{W^{k,\infty}}
	\]

	For the second term, we can use Lemma \ref{lem:V}, since by assumption $\psi $ is supported on the unit ball of $\R^d.$
	We have
	\begin{equation}
	\esssup_{x_+,x_-}|(\nabla _{\pm})^{k-1}[II_\epsilon ]|
	\leq |\sigma |_{W^{k}}|\Phi |_{W^{k,\infty}}\,.
	\end{equation}
	
	\bigskip
	
	\item[\indent\textit{Step 2: uniform estimates on the first component.}]
	For $V\in \DD_1$ define \[\bbGamma(V):=V\otimes \id+\id\otimes V\,,\] and further let
	\begin{equation}
	\bbGamma ^{\epsilon }(V):=
	T_\epsilon ^{*}\bbGamma(V) (T_\epsilon^*)^{-1}\,.
	\end{equation}
	Particularizing \eqref{estimates_area} with $V=B^1_{st}\in\DD_1$ for fixed $s,t,$ we see by definition of $\Gamma ^{1,\epsilon }_{st}(B)$ that
	\begin{multline*}
	|\Gamma _{st}^{1,\epsilon }(\B)|_{\mathscr L(W^{-k+1,1}(\Omega ),W^{-k,1}(\Omega ))}
	\equiv|\bbGamma ^\epsilon (B^{1}_{st})|_{\mathscr L(W^{-k+1,1}(\Omega ),W^{-k,1}(\Omega ))}
	\\
	\leq 
	|\bbGamma ^{\epsilon }(B^1_{st})^*|_{\mathscr L(W^{k,\infty}_0(\Omega ),W^{k-1,\infty}_0(\Omega ))}
	\\
	\quad 
	\equiv |T_\epsilon ^{-1}(B^{1,*}_{x,st}+B^{1,*}_{y,st})T_\epsilon|_{\mathscr L(W^{k,\infty}_0(\Omega ),W^{k-1,\infty}_0(\Omega ))}
	\leq C\omega _B(s,t)^{\alpha }\,,
	\end{multline*}
	for any $k\in \{1,2,3\}.$
	This yields the first part of the claimed estimate.
	
	Note that, since the bracket $[\B]_{st}$ has order one ($\B$ is geometric), we can let $V=[\B]_{st}$ in the previous computations in order to obtain
	\begin{equation}
	\label{estimates_area_2}
	|\bbGamma ^{\epsilon }([\B]_{st})|_{\mathscr L(W^{-k+1,1}(\Omega ),W^{-k,1}(\Omega ))}
	\leq C\omega _B(s,t)^{2\alpha }\,.
	\end{equation} 
	
	\bigskip
	
	\item[\indent\textit{Step 3: uniform estimates on the second component.}]
	Recalling that $[\B]:=B^2-B^1\circ B^1/2,$ we have by definition of $\Gamma ^2_{st}(B)$:
	\[
	\begin{aligned}
	\Gamma _{st}^{2,\epsilon }(\B)
	&=T_\epsilon ^{-1}\left(B^2_x +B^1_xB^1_y + B^2_y\right)_{st}T_\epsilon 
	\\
	&\equiv T_\epsilon ^{-1}\left(\frac12B^1_xB^1_x+[\B]_x +B^1_xB^1_y + \frac{1}{2}B^1_yB^1_y +[\B]_y\right)_{st}T_\epsilon 
	\\
	&=T_\epsilon ^{-1}\left(\frac12(B^1_x + B^1_y)^2 + [\B]_x+[\B]_y\right)_{st}T_\epsilon \,.
	\end{aligned}
	\]
	Otherwise said, we have the algebraic identity
	\begin{equation}
	\label{expression_S}
	\Gamma ^{2,\epsilon}_{st}(\B) = \frac12\Gamma_{st}^{1,\epsilon}(\B)\circ\Gamma_{st}^{1,\epsilon}(\B) + \mathscr B_{st}^\epsilon \,,
	\quad 
	\text{where}\enskip 
	\mathscr B_{st}^\epsilon 
	:=T_\epsilon ^{-1}([\B]_x+[\B]_y)_{st}T_\epsilon .
	\end{equation} 
	But if $k\in\{-1,0\},$ the estimate \eqref{estimates_area_2} shows that
	\begin{equation}
	\label{uniform_L2}
	|\mathscr B^\epsilon _{st}|_{\mathscr L(W^{k,1},W^{k-1,1})}\leq C\omega _B(s,t)^{2\alpha }\,.
	\end{equation}
	We can now conclude thanks to \eqref{uniform_L2} and Step 2, since for $k=0,-1:$
	\begin{multline*}
	|\Gamma ^{2,\epsilon }_{st}(\B)|_{\mathscr L(W^{k,1},W^{k-2,1})}
	\\
	\leq \frac12|\Gamma ^{1,\epsilon }_{st}(\B)|_{\mathscr L(W^{k,1},W^{k-1,1})}|\Gamma_{st}^{1,\epsilon }(\B)|_{\mathscr L(W^{k-1,1},W^{k-2,1})}
	+|\mathscr B_{st}^\epsilon |_{\mathscr L(W^{k,1},W^{k-2,1})}
	\leq C\omega _B(s,t)^{2\alpha }\,,
	\end{multline*}
	which finishes the proof of Theorem \ref{thm:renorm}.
	\hfill\qed
\end{trivlist}

\section{Further remarks and comments}
\label{app:further}

\subsection{Uniqueness of the Gubinelli derivative}
\label{app:true_roughness}

Let $u$ be such that
\[
\d u = f\d t + \d \B(g,g')\,,
\]
where $f\in L(0,T;W^{-1,p})$ while $(g,g')\in\mathcal D^{\alpha ,p}_B,$
and write 
\[
u\simeq (f;g,g')\,.
\]

It is natural to ask under which condition one can have uniqueness of the triple $(f;g,g')$ such that $u\simeq (f;g,g'),$ a question that relates the Doob-Meyer decomposition for semi-martingales.
Such uniqueness is certainly not true in general because our definition of a differential rough driver could accomodate that of $\dot B:=\dot Z\partial _x,$ where $Z\in C^\infty(0,T;\R).$ Indeed, in this case one can arbitrarily choose $g'=0$ for any $u$ and alternatively represent the element $u\simeq(f;g,0)$
by writing instead $u\simeq(f+\dot Z\partial _xg;0,0).$

In the finite-dimensional case however (for instance replacing $B$ by a path $Z$ of $\frac1\alpha $-finite variation with values in $\R$), the decomposition \eqref{weak_rel} is indeed unique in the case where $Z$ is \emph{truly rough} \cite{friz2012doob}, i.e.\ when there exists a dense set of times $t\in [0,T]$ such that
\begin{equation}
\label{true_roughness}
\limsup_{s\rightarrow t}\frac{|Z_{st}|}{\omega _Z(s,t)^{2\alpha }}=\infty.
\end{equation}

The situation here is different in the sense that assuming $B=Z\sigma \cdot \nabla $
with $Z$ as in \eqref{true_roughness} does not guarantee uniqueness of the couple $(f,g)$ in \eqref{weak_rel}.
To wit, assume that $d=2,$ and let $B$ as above with $\sigma =(0,1).$
If $(f,g)$ satisfy \eqref{weak_rel}, then it is immediately seen that any path of the form $t\mapsto g_t(x,y)+\tilde g_t(x)$ where $\tilde g \in \V_1^\alpha (0,T;L^2(\R))$ is a function of the first variable only, will also satisfy \eqref{weak_rel}.
In this counterexample, one sees that the space variable plays an important role in the discussion, and that if one aims at the uniqueness of the above decomposition, then some ``non-degeneracy'' assumptions on the differential operator $\sigma \cdot \nabla $ are in order. Let us now formulate a natural sufficient condition under which uniqueness of the Gubinelli derivative holds.

Assume that we are given a family $B_t$ of (non-necessarily differential) operators such that the mapping $[0,T]\to \cap_{-2\leq k\leq 0} \mathscr L(H^k,H^{k-1}),$ $t\mapsto B_t$ is $\alpha $-H\"older continuous, where as before $\alpha >1/3.$
For notational simplicity, we denote in the sequel $B_{st}:=\delta B_{st}.$

\begin{theorem}
	Assume the existence of $\gamma \in [\alpha ,\frac32\alpha ),$
	such that the following ellipticity condition is satisfied:
	there is a constant $\Lambda>0,$ such that for every $\varphi$ in $H^{-1},$ and for each $(s,t)\in\Delta \cap D^2 ,$
	\begin{equation}
	\label{non_degeneracy}
	|B_{st}\varphi |_{H^{-2}}
	\geq \Lambda(t-s)^{\gamma }|\varphi |_{H^{-1}}
	\end{equation}
	where we are given some dense subset $D$ of $[0,T].$
	
	Let $u\in L^\infty(0,T;L^2)\cap C^\alpha (0,T;H^{-1})$ and suppose that $g,\tilde g\in C^\alpha (0,T;H^{-1})$ are both Gubinelli derivatives for $u$ in the H\"older sense, by which we mean that
	\[
	\sup_{0\leq s<t\leq T}\frac{|R^g_{st}|_{H^{-2}}}{(t-s)^{2\alpha }}=
	\sup_{0\leq s<t\leq T}\frac{|\delta u_{st}-B_{st}g_s|_{H^{-2}}}{(t-s)^{2\alpha }}<\infty\,\,,
	\]
	and similar for $\tilde g.$
	Then, $g=\tilde g.$
\end{theorem}

\begin{proof}
	Fix $(s,t)\in \Delta \cap D^2.$
	The assumption \eqref{non_degeneracy} implies that the bilinear form
	\[
	a_{st}:H^{-1}\times H^{-1}\to \R,\quad 
	a_{st}(u,v):= (B_{st}u,B_{st}v)_{H^{-2}}
	\]
	is $H^{-1}$-coercive.
	Therefore, if $F:H^{-1}\to \R$ is linear and continuous, the variational problem
	\begin{equation}
	\label{problem_var}
	\left\{\begin{aligned}
	&\text{Find}\enskip u\in V:=H^{-1}\enskip \text{such that}
	\\
	&\forall v\in V\,,\quad a_{st}(u,v)= F(v)\,.
	\end{aligned}\right.
	\end{equation} 
	admits a unique solution
	\[
	u=\mathbf T_{st}F\in H^{-1}\,.
	\]
	Moreover, it is easily seen that the Riesz isomorphism between $H^{-2}$ and its dual identifies the dual of $H^{-1}$ with $H^{-3}$, hence 
	the operator norm of $\mathbf T_{st}:(H^{-1})^*\simeq H^{-3}\to H^{-1}$ is estimated above as
	\[|\mathbf T_{st}|_{\mathscr L(H^{-3},H^{-1})}\leq \Lambda ^{-2}(t-s)^{-2\gamma } \,.\]
	Furthermore,
	if $B^\dagger_{st}$ denotes the adjoint of $B_{st}$ with respect to the $H^{-2}$-inner product,
	observe thanks to \eqref{problem_var} that $\mathbf T_{st}$ is the inverse transform of
	\[B_{st}^\dagger \circ B_{st}:H^{-1}\to H^{-3}.\]

	Let $g$ be a Gubinelli derivative for $u.$
	From the above discussion, one infers the relation
	\[
	g_s= \mathbf T _{st}B_{st}^{\dagger }\delta u_{st}
	- \mathbf T _{st}B_{st}^{\dagger }R^u_{st}=:I+II.
	\]
	By assumption on $R^u_{st}:=\delta u_{st}-B_{st}g_s,$ it holds
	\[
	|II|_{H^{-1}}
	\leq \Lambda^{-2}(t-s)^{-2\gamma }|B^{\dagger}_{st}R^u_{st}|_{H^{-3}}
	\leq \Lambda^{-2}(t-s)^{3\alpha -2\gamma }\|R^u\|_{C^\alpha_2 (H^{-2})}.
	\]
	Hence, letting $t_n\searrow s,$ $t_n\in D,$
	one sees that 
	\[
	|II|_{L^2}\leq C(t_n-s)^{3(\alpha -\frac{2}{3}\gamma) }
	\to 0\quad \text{as}\enskip n\to\infty.
	\]
	This implies that $g_t$ is uniquely determined by the relation
	\[
	g_t= \lim_{s\to t,s\in D}\mathbf T_{st}B^{\dagger}_{st}\delta u_{st}\quad \text{in}\enskip H^{-1},
	\]
	thus proving our claim.
\end{proof}

\begin{example}
	Let $d=1,$ and consider a $1$-dimensional, $\alpha $-H\"older rough path $(Z^1,Z^2)\in\mathscr C^\alpha (0,T;\R)$  such that for some $D$ as above it holds
	\[
	|Z_{st}|\geq c (t-s)^\gamma,\quad \text{for every}\enskip (s,t)\in\Delta \cap D^2,
	\] 
	where we are given some constant $\gamma \in[\alpha ,2\alpha )$ (this implies in particular true roughness for $Z$, in the sense of \eqref{true_roughness}).
	Moreover, let $\sigma \in W^{3,\infty}$ be bounded below, namely such that
	there exist constants $\underline{\sigma} >0$
	with the property that
	$\sigma (x)\geq \underline{\sigma }$,
	for almost every $x\in \R^d.$
	
	Then, it is easily seen that \eqref{non_degeneracy} holds with the differential rough driver $\B$ given by Example \ref{ex:URD_gene} with $\rho =0$, where $\Lambda =\Lambda (c,\underline{\sigma })>0.$
\end{example}

\subsection{Brackets}
For a geometric rough path $(Z^{1,\mu },Z^{2,\mu \nu })_{1\leq \mu ,\nu \leq m}$ it is well-known that the symmetric part of $Z^2$ is expressed in terms of $Z^1,$ as follows
\begin{equation}
\label{geometric_rp}
\sym Z^{2,\mu \nu }_{st}\equiv \frac{Z_{st}^{2,\mu \nu }+Z_{st}^{2,\nu \mu }}{2}=\frac{Z_{st}^{1,\nu} Z_{st}^{1,\mu} }{2}, \quad \text{for all}\enskip 1\leq \mu ,\nu \leq m,
\end{equation}
and every $(s,t)\in\Delta $ (see \cite{lyons1998differential}). Alternatively, this means that the \emph{bracket} $[Z]_{st}:=\sym Z^2_{st}-\frac12(Z^1_{st})^2$ vanishes for geometric rough paths.
By analogy, in the case of an differential rough driver $\B$, we introduced the bracket as the following family of differential operators:
\begin{equation}
\label{bracket_rem}
\L_{st}:=B^2_{st}-\frac12B^1_{st}\circ B^1_{st},\quad (s,t)\in\Delta \,,
\end{equation} 
(see Lemma \ref{lem:multiplication}).
In contrast with what is encountered in the classical theory, note that the bracket does not vanish in general for $\B$ geometric, which is a side effect of the non-commutativity of the algebra of differential operators.
Nevertheless, we saw in Lemma \ref{lem:multiplication} that, as a consequence of geometricity, $\L$ takes values in the space of $\DD_1$. In particular, unless $B^1_{st}\in\DD_0,$ we see that a a cancellation occurs, since in that case $\L_{st}$ has stricly lower order than $B_{st}^2.$ This can be seen as a non-commutative counterpart of the fact that the bracket of geometric rough paths is zero.

\begin{remark}
	\label{rem:ito_enhancement}
	If $\B$ denotes a differential rough driver, then by definition of the bracket $\L$ in \eqref{bracket_rem},
	we have
	\[
	B_{st}^2(\phi \psi ) = (B^2_{st}\phi )\psi + (B^1_{st}\phi) (B^1_{st}\psi ) + \phi (B^2_{st}\psi )
	- \mathfrak{l}_{st}(\phi ,\psi )
	\]
	where $\mathfrak{l}_{st}$ denotes the (generally unbounded) bilinear operator
	\begin{equation}
	\label{nota:l_frak}
	\phi ,\psi \mapsto \mathfrak l_{st}(\phi ,\psi )=
	\L_{st}(\phi \psi) -(\L _{st}\phi) \psi - \phi (\L_{st}\psi )\,.
	\end{equation} 
	
	To give a concrete example, consider a filtered probability space $(\Omega ,\mathcal A,\mathbb{P},\{\mathcal F_t\}_{t\in[0,T]}),$
	let $W:\Omega \times [0,T]\to \R$ be a Brownian motion, and fix $V\in \DD_1\setminus\DD_0.$
	Define the (random) differential rough driver $\B^\ito(\omega )$ by
	$B^{\ito,1}_{st}:=(W_t-W_s)V$ and, observing that $\mathbb{P}$-a.s., $\int_s^t(W_r-W_s)\d W_r=\frac12[(W_t-W_s)^2-(t-s)]$ (It\^o sense), let
	\[B^{\ito,2}_{st}:=\frac12[(W_t-W_s)^2 -(t-s)]V^2.\]
	With this definition, we have
	\[
	[\B^{\ito}]_{st}= -\frac{(t-s)}{2}V^2\,,
	\]
	showing that $\L\in \DD_2\setminus\DD_1,$ almost surely.
\end{remark}

\begin{remark}
	\label{rem:sto_par}
	As seen in the above remark, if $\B$ is not geometric, its bracket $\L$ (see \eqref{bracket_rem}) is generally not first order. In the stochastic context, this has to do with the violation of stochastic parabolicity assumption, as can be seen as follows.
	Using the notations of Remark \eqref{rem:ito_enhancement},
	we see that in the proof of the product formula, the equation \eqref{limit_Q2} must be changed to
	\begin{multline*}
	\lim_{\epsilon \to0}\braket{}{ \Gamma^{2,\epsilon } _{st}(\B)(u\otimes v)_s^\epsilon ,\Phi }{}
	=\braket{}{(B^2_{st}u_s)v_s + (B^1_{st}u_s)( B^1_{st}v_s)+u_s (B^2_{st} v_s),\phi }{}
	\\
	\equiv\langle B^2_{st}(u_sv_s),\phi \rangle
	+\langle \frak l_{st}(u_s,v_s),\phi \rangle\,.
	\end{multline*}
	If we let furthermore $u=v$ where $u$ is an $L^2$-energy solution of \eqref{rough_parabolic}, $\B=\B^{\ito},$ and $\phi =1,$ we have
	\[
	\langle \frak l_{st}(u_s,u_s),1\rangle = (t-s)\int_{U}(Vu_s)^2\d x\,.
	\]
	The latter competes with the term $-2\lambda\iint_{[s,t]\times U}|\nabla u|^2\d x\d r,$ which is brought by the elliptic part of the equation.  In particular, the usual technique to obtain the energy estimate on $u$ fails, unless the coefficients of $V$ are taken small with respect to $\lambda.$
	This illustrates the importance of the geometricity assumption in our results.
\end{remark}

\section*{Acknowledgements}
The authors would like to thank the anonymous referee who significantly helped to improve the quality of this manuscript.
Support by the DFG via Research Unit FOR 2402 is gratefully acknowledged.

% Non-BibTeX users please use

%
\end{document}